\documentclass[11pt,a4paper]{article}

\usepackage[utf8]{inputenc}
\usepackage[english]{babel}
\usepackage{latexsym}

\usepackage{amsmath,amsfonts,amsthm,amssymb,amscd}
\usepackage{mathrsfs}

\usepackage{enumerate}

\usepackage[T1]{fontenc}

\usepackage[all]{xy}
\usepackage{graphicx}
\usepackage{url}
\usepackage{color}

 \usepackage{geometry}
\theoremstyle{definition}
\newcounter{Polya}
\newcounter{FZ}

\newtheorem{theo}{Theorem}[section]
\newtheorem*{theo*}{Theorem}
\newtheorem*{teo*}{Teorema}
\newtheorem{pro}[theo]{Proposition}
\newtheorem*{pro*}{Proposition}
\newtheorem*{prop*}{Proposici\'on}
\newtheorem{lemma}[theo]{Lemma}
\newtheorem{cor}[theo]{Corollary}
\newtheorem*{cor*}{Corollary}
\newtheorem*{coro*}{Corolario}
\newtheorem{defi}[theo]{Definition}

\newtheorem{rema}[theo]{Remark}

\newtheorem{exam}[theo]{Example}

\swapnumbers

\newenvironment{eq-text}
{\begin{equation} \begin{minipage}[t]{0.85\linewidth}}
{\end{minipage} \end{equation} \ignorespacesafterend}

\makeatletter
\DeclareRobustCommand\widecheck[1]{{\mathpalette\@widecheck{#1}}}
\def\@widecheck#1#2{%
    \setbox\z@\hbox{\m@th$#1#2$}%
    \setbox\tw@\hbox{\m@th$#1%
       \widehat{%
          \vrule\@width\z@\@height\ht\z@
          \vrule\@height\z@\@width\wd\z@}$}%
    \dp\tw@-\ht\z@
    \@tempdima\ht\z@ \advance\@tempdima2\ht\tw@ \divide\@tempdima\thr@@
    \setbox\tw@\hbox{%
       \raise\@tempdima\hbox{\scalebox{1}[-1]{\lower\@tempdima\box
\tw@}}}%
    {\ooalign{\box\tw@ \cr \box\z@}}}
\makeatother

\def\pd{\partial}
\def\fy{\varphi}
\def\oo{\infty}
\def\to{\rightarrow}
\def\a{\alpha}
\def\b{\beta}
\def\de{\delta}
\def\De{\Delta}
\def\ep{\varepsilon}
\def\o{\omega}

\def\ga{\gamma}
\def\Ga{\Gamma}
\def\ro{\rho}

\def\ze{\zeta}

\def\ba{\boldsymbol{a}}
\def\bb{\boldsymbol{b}}

\def\bv{\boldsymbol{v}}

\def\Re{\mathop{\rm Re}\nolimits}

\def\en{\subseteq}

\def\N{\mathbb{N}}
\def\Z{\mathbb{Z}}

\def\R{\mathbb{R}}
\def\C{\mathbb{C}}      

\def\uC{\mathbb{S}^1}
\def\M{\mathbb{M}}

\def\bM{\mathbb{M}}
\def\L{\mathbb{L}} 

\def\U{\mathbb{U}}
\def\A{\mathbb{A}} 

\def\V{\mathbb{V}}
\def\m{\textbf{m}}
\def\mf{\mathfrak{m}}

\def\bm{\textbf{m}}
\def\l{\ensuremath{\boldsymbol\ell}}

\def\no{\nonumber}
\def\ds{\displaystyle}

\def\para{\par\smallskip}

\newcommand{\un}[1]{{\underline{#1}}}
\newcommand{\sA}{\mathscr{A}}
\newcommand{\cA}{\mathcal{A}}
\newcommand{\fA}{\mathfrak{A}}
\newcommand{\fB}{\mathfrak{B}}
\newcommand{\fC}{\mathfrak{C}}
\newcommand{\fD}{\mathfrak{D}}
\newcommand{\bL}{\mathbb{L}}
\newcommand{\cO}{\mathcal{O}}
\newcommand{\calS}{\mathcal{S}}
\newcommand{\cT}{\mathcal{T}}
\newcommand{\fS}{\C[[z]]}
\newcommand{\cS}{\C\{z\}}
\newcommand{\se}{\simeq}

\newcommand{\lto}{\longrightarrow}
\newcommand{\slto}{\overset{\sim}{\lto}}
\numberwithin{equation}{section}
\begin{document}

\title{Multisummability in Carleman ultraholomorphic classes\\
by means of nonzero proximate orders}

\author{Javier Jim\'enez-Garrido, Shingo Kamimoto, Alberto Lastra, Javier Sanz}
\date{\today}

\maketitle
\abstract{We introduce a general multisummability theory of formal power series in Carleman ultraholomorphic classes. The finitely many levels of summation are determined by pairwise comparable, nonequivalent weight sequences admitting nonzero proximate orders and whose growth indices are distinct. Thus, we extend the powerful multisummability theory for finitely many Gevrey levels, developed by J.-P.~Ramis, J.~\'Ecalle and W.~Balser, among others. We provide both the analytical and cohomological approaches, and obtain a reconstruction formula for the multisum of a multisummable series by means of iterated generalized Laplace-like operators.}


\section{Introduction}\label{sect.Intro}

The main aim of this paper is to give some steps forward in the objective of establishing a general multisummability theory of formal power series for finitely many levels determined by pairwise comparable, nonequivalent weight sequences admitting nonzero proximate orders. Thus, we will extend the powerful multisummability theory for finitely many Gevrey levels, developed by J.-P.~Ramis, J.~\'Ecalle, W.~Balser, Y.~Sibuya, J.~Martinet, B.~Malgrange and B.L.J.~Braaksma, among others. The departure point will be the summability theory in ultraholomorphic classes defined in terms of a weight sequence admitting a nonzero proximate order (i.e., in the case of just one level), developed by the last two authors and S.~Malek~\cite{lastramaleksanz15}, which generalized the Gevrey $k-$summability of J.-P.~Ramis~\cite{Ramis78,Ramis80} by means of the moment methods of W.~Balser~\cite{Balser2000}. We start by commenting on the development of these tools and justifying the interest of their generalization.

A remarkable result of E.~Maillet~\cite{Maillet} in 1903 states that for any formal solution $\widehat{f}=\sum_{p\ge 0}a_pz^p$ of an analytic differential equation there will exist $C,A,k>0$ such that $|a_p|\le CA^p(p!)^{1/k}$ for every $p\in\N_0$. This suggests that, although the formal power series solutions to differential equations are frequently divergent, under fairly general conditions the rate of growth of their coefficients is not arbitrary.
Inspired by this fact, in the late 1970's J.-P.~Ramis introduces and
structures the notion of $k-$summability, that rests on classical results by G.~N.~Watson and R.~Nevannlina and generalizes Borel's summability method. His developments are based on a modification of H.~Poincar\'e's concept of asymptotic expansion: in the general estimates
\begin{equation*}
\big|f(z)-\sum_{n=0}^{p-1} a_n z^n\big|\leq C_p|z|^p,\quad   p\in\N_0=\N\cup\{0\}, 
\end{equation*}
involving a complex function $f$ holomorphic in a sector $S$
and the partial sums of the formal power series $\widehat{f}=\sum_{p=0}^\oo a_p z^p$ of asymptotic expansion at the origin of $f$,
the growth of the constant $C_p$ is made explicit in the form
$C_p=C A^p (p!)^{1/k}$ for some $A,C>0$, what entails the same kind of estimates for the coefficients $a_p$ in $\widehat{f}$. The sequence $\M_{1/k}=(p!^{1/k})_{p\in\N_0}$ is the Gevrey sequence of order $1/k$, $f$ is said to be $1/k-$Gevrey asymptotic to $\widehat{f}$ (denoted by $f\in\widetilde{\mathcal{A}}_{\M_{1/k}}(S)$), and $\widehat{f}$, because of the estimates satisfied by its coefficients, is said to be a $1/k-$Gevrey series ($\widehat{f}\in\C[[z]]_{\M_{1/k}}$). The Borel map, defined from $\widetilde{\mathcal{A}}_{\M_{1/k}}(S)$ to $\C[[z]]_{\M_{1/k}}$ and sending $f$ to $\widehat{f}$, is surjective if and only if the opening of the sector $S$ is smaller than or equal to $\pi/k$ (Borel-Ritt-Gevrey Theorem), and it is injective if and only if the opening is greater than $\pi/k$ (Watson's Lemma).

This last fact enables the definition of $k-$summable power series in a direction $d$ as those in the image of the Borel map for a wide enough sector $S$ bisected by $d$, to which a $k-$sum (the unique holomorphic function in $S$ asymptotic to it) is assigned.
J.-P.~Ramis proved, by a purely theoretical method, that
every formal solution to a linear system
of meromorphic ordinary differential equations in the complex domain at
an irregular singular point can
be written as some known functions times a finite product of  formal power series,
each of which is $k-$summable (i.e., $k-$summable in every direction except for a finite number of them) for some level $k$ depending on the series.
New insight was obtained by the introduction of a powerful tool, accelerosummability, due to
J.~Ecalle \cite{Ecalle1981} and which, in the case involving only a finite number of Gevrey levels, is named multisummability (in a sense, an iteration of elementary $k-$summability procedures). Indeed,
in 1991 W. Balser, B.L.J.~Braaksma, J.-P.~Ramis and Y.~Sibuya~\cite{BalserBraaksmaRamisSibuya91} (see also \cite{Balser2000, MartinetRamis91}) proved the multisummability of the formal solutions of linear meromorphic differential equations at a singular point, and B.L.J.~Braaksma~\cite{Braaksma92} (for different proofs, see~\cite{Balser1994,RamisSibuya94}) extended this result for nonlinear equations in 1992, which allows in every case to compute actual solutions from formal ones.
This technique has also been proven to apply successfully to a plethora of situations concerning the study of formal power series solutions at a singular point for partial differential equations (see, for example, \cite{Balser04,BalserMiyake99,Hibino2008,Malek09,Ouchi02}), as well as for singular perturbation problems (see~\cite{BalserMozo02,CanalisMozoSchafke07,lastramaleksanz13}, among others). Six different approaches to (Gevrey) multisummability can be found in the recent book of M.~Loday-Richaud~\cite{Loday16}.\para

Although Gevrey multisummability may also be applied to the formal power series solutions of some classes of difference equations (see \cite{BraaksmaFaber96,BraaksmaFaberImmink00}), G.K.~Immink~\cite{Immink88,Immink96} showed that nonGevrey asymptotics (specially, those associated to a so-called ``$1^+$ level'') needs to be considered for the study of the formal power solutions of even some linear difference equations, which are not Gevrey summable in countably many singular  directions. So, more exotic types of summation processes, with operators other than the usual Borel-Laplace ones, must be considered. This was achieved by G.K.~Immink in a series of papers~\cite{Immink01,Immink08,Immink11} which may be considered as
an illustration of the theory of weak accelerations and
`cohesive' functions of J.~\'Ecalle.
More recently, S.~Malek~\cite{Malek14} has studied some singularly perturbed small step size difference-differential nonlinear equations whose formal solutions with respect to the perturbation parameter can be decomposed as sums of two formal series, one with Gevrey order $1$, the other of $1^+$ level.
As another example, V.~Thilliez has proven some results on solutions  for algebraic equations within these general (not necessarily Gevrey) ultraholomorphic classes in~\cite{ThilliezSmooth10}. All these results motivated
the introduction of generalized summability methods by A.~Lastra, S.~Malek and J.~Sanz~\cite{lastramaleksanz15,SanzFlat,SanzAsymptoticAnalysis}. On the one hand, their definition heavily rests on Watson's Lemma for weight sequences, obtained in full generality by the first and fourth authors together with G.~Schindl~\cite{JimenezSanzSchindlInjectSurject}, and that guarantees injectivity of the Borel map in wide enough sectorial regions. The reconstruction of the sum of a summable series in a direction requires kernels for Borel and Laplace transforms, whose existence is assured whenever the sequence $\M$ admits a nonzero proximate order, a property which is satisfied by the sequences appearing in applications and that has been completely characterized (Theorem~\ref{theo.charact.admit.p.o}; see~\cite[Th.\ 4.14]{JimenezSanzSchindl} and~\cite[Th.\ 2.2.19]{JimenezPhD}).\para

As aforementioned, the aim of this paper is to put forward the corresponding multisummability theory, in Balser's sense, by suitably combining the methods for different sequences $\M_1,\M_2, \dots, \M_n$   instead of different Gevrey levels $k_1,k_2,\dots ,k_n$.
Section \ref{sect.preliminaries} is devoted to gathering the main relevant information regarding weight sequences $\M$ and their associated functions (Subsection~\ref{subsectLCsequences}), in terms of which the classes of functions with $\M-$asymptotic expansion are defined and studied (Subsection~\ref{subsect.ultrahol.Borel.map}). The main result here is Watson's Lemma, characterizing the injectivity of the Borel map. Proximate orders and the theory of regular variation for functions and sequences are treated in Subsection~\ref{subsect.SequencesAdmitProxOrder}, where the role both tools play in the extension of $k-$summability to this more general framework is emphasized. In order to complete the framework for what follows, the construction and the properties of the kernels for $\M-$summability and the corresponding, formal and analytic, Laplace and Borel transforms are recalled in Subsection~\ref{subsect.M.summability}, allowing us to explicitly obtain the $\M-$sums of formal power series $\M-$summable in a direction. Section~\ref{sect.tauberian.theorems} starts with a preliminary discussion concerning the necessity, for the problem of multisummability to make sense, that the sequences $\M_j$ are pairwise comparable and nonequivalent, what will be studied in Subsection~\ref{subsect.comp.sequence}.
After establishing the basic properties of the quotient and product sequences of two weight sequences (Subsection~\ref{subsect.product.quotient.seq}),  the Tauberian
Theorem~\ref{th.tauberian.distinct.index} will be obtained, which
%
%
allows for a consistent definition of multisummability whenever the growth indices $\o(\M_j)$ of the sequences involved (see Subsection~\ref{subsectLCsequences}) are mutually distinct. Section~\ref{sect.multisummab.analyt} contains a purely analytical approach to multisummability. A formal power series will be said to be multisummable if it can be split into the sum of finitely many formal power series $\widehat{f}_j$, each of them summable for a corresponding sequence $\M_j$ admitting a nonzero proximate order. With the aim of obtaining again an explicit expression for the multisum of a power series in a multidirection, we first prove in Subsection~\ref{subsect.momentkernelduality} that a kernel of summability is uniquely determined by its sequence of moments. Strong kernels of summability have to be introduced (Subsection~\ref{subsect.strongkernels}) in order to obtain Theorems~\ref{prop.convolution.kernel} and~\ref{prop.acceleration.kernel}, where the natural summability kernels for the quotient and product sequences of two sequences will be built, so giving rise, respectively, to convolution and acceleration kernels and operators, as developed by W. Balser in~\cite{Balser2000}.
Thanks to them, in Subsection~\ref{subsect.multisum.through.accel} we will be able to devise a procedure for the explicit reconstruction, through acceleration, of the multisum of a multisummable series, that is, the sum of the corresponding ones to every summand in the previous splitting (see Theorem~\ref{th.multisum.construct}).
The last part of the paper, Section~\ref{sect.cohomologicalapproach}, contains a cohomological approach to multisummability, following the one established by B. Malgrange and J.-P. Ramis~\cite{MR}, so focusing on the algebraic framework of the theory. We first prove a relative Watson's Lemma (Subsection~\ref{subsect.relativeWatson}), reducing it to the Tauberian Theorem~\ref{th.tauberian.distinct.index}. We then discuss about the structure of the space of quasi-functions, giving a description of the space corresponding to Balser's decomposition of a multisummable series into a sum of summable series. Finally, three different perspectives are given for multisummability: via quasi-functions, via the decomposition or via acceleration, i.e., through iterated Laplace transforms as at the end of the previous section.

It is worth mentioning that, although $\M-$summability methods have been applied in~\cite{lastramaleksanz15,lastramaleksanz16}, their development remains still on a quite theoretical level. The multisummability techniques developed here allow us to work at the same time with $\M-$ and $k-$summability. For example, since the level $1^{+}$ corresponds to the weight sequence $\M_{1,-1}$ (see Example~\ref{exampleSequences}), which admits a nonzero proximate order, it is expected that this new tool can be applied to the formal solutions of difference equations whenever the other levels, apart from the $1^{+}$, are distinct from $1$. In case the levels 1 and $1^+$ coexist, and since they are associated with comparable, nonequivalent sequences sharing the same growth index,
it seems necessary to redefine the $\M-$summability notion in such a way that the analogue of Tauberian Theorem~\ref{th.tauberian.distinct.index}
is available, but this is still work in progress.


\section{Preliminaries}\label{sect.preliminaries}

We start by fixing some notations.
We set $\N:=\{1,2,...\}$, $\N_{0}:=\N\cup\{0\}$.
$\mathcal{R}$ stands for the Riemann surface of the logarithm.
We consider bounded \emph{sectors}
$$S(d,\gamma,r):= \{z\in\mathcal{R}:|\hbox{arg}(z)-d|<\frac{\gamma\,\pi}{2},\ |z|<r\},
$$
respectively unbounded sectors
$$
S(d,\gamma):=\{z\in\mathcal{R}:|\hbox{arg}(z)-d|<\frac{\gamma\,\pi}{2}\},
$$
with \emph{bisecting direction} $d\in\R$, \emph{opening} $\gamma \,\pi$ ($\gamma>0$) and (in the first case) \emph{radius} $r\in(0,\infty)$. For unbounded sectors of opening $\gamma\,\pi$ bisected by direction 0, we write
$S_{\gamma}:=S(0,\gamma)$. In some cases, it will also be convenient to consider sectors whose elements have their argument in a half-open, or in a closed, bounded interval of the real line.\par\noindent

A \emph{sectorial region} $G(d,\gamma)$ with bisecting direction $d\in\R$ and opening $\gamma\,\pi$ will be an open connected set in $\mathcal{R}$ such that $G(d,\ga)\subset S(d,\ga)$, and
for every $\beta\in(0,\gamma)$ there exists $\rho=\rho(\beta)>0$ with $S(d,\beta,\rho)\subset G(d,\gamma)$.
In particular, sectors are sectorial regions. If $d=0$ we just write $G_\ga$.\par\noindent

A bounded (respectively, unbounded) sector $T$ is said to be a \emph{proper subsector} of a sectorial region (resp. of an unbounded sector) $G$, and we write $T\ll G$ (resp. $T\prec G$), if $\overline{T}\subset G$ (where the closure of $T$ is taken in $\mathcal{R}$, and so the vertex of the sector is not under consideration).

$\C[[z]]$ stands for the set of formal power series in $z$ with complex coefficients.

\subsection{Weight sequences and associated functions}\label{subsectLCsequences}

In what follows, $\M=(M_p)_{p\in\N_0}$ always stands for a sequence of positive real numbers, and we always assume that $M_0=1$. We specify some conditions on the sequence $\M$, and introduce a growth index and auxiliary functions, all of which will be relevant in the study of the forthcoming classes or spaces defined in terms of $\M$.


For a sequence $\M$ we define {\it the sequence of quotients}  $\m=(m_p)_{p\in\N_0}$ by
$m_p:=M_{p+1}/M_p$, $p\in \N_0$.  $\M$ can be recovered from $\m$ because
$M_p=\prod_{k=0}^{p-1}m_k$ for every $p\in\N$.

We say $\bM$ is \textit{logarithmically convex} (for short, (lc)) if
$$M_{p}^{2}\le M_{p-1}M_{p+1},\qquad p\in\N.$$
It is obvious that $\M$ is (lc) if, and only if, $\m$ 
is nondecreasing.
Moreover, if $\M$ is (lc), we have that
$$M_p=m_0\cdots m_{p-1}\leq (m_{p-1})^p\leq (m_p)^p, \qquad p\in\N$$
and we deduce that
$$ M_{p}^{(p+1)/p}=M_p (M_{p})^{1/p}\leq M_p m_p =M_{p+1}, \qquad p\in\N,$$
then $((M_p)^{1/p})_{p\in\N}$ is nondecreasing.

We say that a sequence $\M$ is a \textit{weight sequence} if it is (lc) and $\lim_{p\to\oo} (M_p)^{1/p}= \oo$ or, equivalently, if $\m$ is nondecreasing and $\lim_{p\to\oo} m_p= \oo$.

For a weight sequence $\bM=(M_{p})_{p\in\N_0}$
the map $h_{\bM}:[0,\infty)\to\R$, defined by
\begin{equation*}
h_{\bM}(t):=\inf_{p\in\N_{0}}M_{p}t^p,\quad t>0;\qquad h_{\bM}(0)=0,
\end{equation*}
turns out to be a nondecreasing continuous map in $[0,\infty)$ onto $[0,1]$. In fact
$$
h_{\bM}(t)= \left \{ \begin{matrix}  t^{p}M_{p} & \mbox{if }t\in\big[\frac{1}{m_{p}},\frac{1}{m_{p-1}}\big),\ p=1,2,\ldots,\\
1 & \mbox{if } t\ge 1/m_{0}. \end{matrix}\right.
$$
One may also associate with a weight sequence $\bM$ the function
\begin{equation}\label{equadefiMdet}
\o_{\M}(t):=\sup_{p\in\N_{0}}\log\big(\frac{t^p}{M_{p}}\big)=-\log\big(h_{\bM}(1/t)\big),\quad t>0;\qquad \o_{\M}(0)=0,
\end{equation}
which is a nondecreasing continuous map in $[0,\infty)$ with $\lim_{t\to\infty}\o_{\M}(t)=\infty$.
Indeed,
\begin{equation*}
\o_{\M}(t)=\left \{ \begin{matrix}  p\log t -\log(M_{p}) & \mbox{if }t\in [m_{p-1},m_{p}),\ p=1,2,\ldots,\\
0 & \mbox{if } t\in [0,m_{0}), \end{matrix}\right.
\end{equation*}
and one can easily check that $\o_{\M}$ is convex in $\log t$, i.e., the map $t\mapsto \o_{\M}(e^t)$ is convex in $\mathbb{R}$.
As it can be found in~\cite[p.~17]{Mandelbrojt} or~\cite[Prop.~3.2]{komatsu}, $\M$ is determined by $\o_{\M}(t)$:
\begin{equation}\label{eq.WeightSeqFromOmegaM}
M_p=\sup_{t>0} \frac{t^p}{e^{\o_{\M}(t)}} = \sup_{t>0}  t^p h_{\M} (1/t) , \qquad p\in\N_0.
\end{equation}

For a weight sequence $\M$ we define the growth index
$$\o(\M):=\liminf_{p\to\infty} \frac{\log(m_p)}{\log(p)}\in[0,\infty],$$
and we consider a new auxiliary function, given for $t$ large enough by
$$
d_{\M}(t):=\frac{\log(\o_{\M}(t))}{\log(t)}.
$$
A connection between the index $\o(\M)$ and the function $d_{\M}$ is given by the following result.

\begin{theo}[\cite{JimenezSanz}, Th.\ 3.2;\ \cite{SanzAsymptoticAnalysis}, Th.\ 2.24 and Th.\ 4.6 ]\label{theo.limsup.dM.omega}
Let $\M$  be a weight sequence, then
$$\limsup_{t\to\infty} d_{\M}(t)
=\limsup_{p\to\infty}\frac{\log(p)}{\log(m_{p})}=\frac{1}{\o(\M)}$$
(where the last quotient is understood as 0 if $\o(\M)=\infty$, and as $\infty$ if $\o(\M)=0$).
\end{theo}

We introduce now the classical notion of comparable and (weakly) equivalent sequences.

Let $\M$ and $\L$ be sequences, we write  $\M\precsim\L$
if there exists $A>0$ such that
$$M_p\leq A^p L_p, \qquad \text{for all}\,\,\, p\in\N_0.$$
We say that $\M$ and $\L$ are \emph{comparable} if $\M\precsim\L$ or $\L\precsim\M$ holds. If both conditions hold, we say that $\M$ is \emph{equivalent} to $\L$, and we write $\M\approx\L$; in this case,
there exist $A,B>0$ such that
\begin{equation}\label{eq:equiv.seq.associated.ommega}
\o_{\bM}(At)\le \o_{\L}(t)\le \o_{\bM}(Bt),\qquad t\ge 0,
\end{equation}
and the previous theorem implies then that $\o(\M)=\o(\L)$.

The strongly regular sequences, introduced by V. Thilliez~\cite{thilliez03}, will play a prominent role in the following.

\begin{defi}\label{defiStronglyRegular}
 A sequence $\M$ is {\it strongly regular} if it satisfies
 \begin{itemize}
\item[(i)]  $\M$ is (lc).

\item[(ii)]  $\M$ is of \textit{moderate growth} (briefly, (mg)): there exists $A>0$ such that
\begin{equation*}
 M_{p+q}\leq A^{p+q} M_p M_q, \qquad p,q\in\N_0.
\end{equation*}

\item[(iii)]  $\bM$ satisfies the \textit{strong nonquasianalyticity condition} (for short, (snq)): there exists $B>0$ such that
$$
\sum_{q\ge p}\frac{M_{q}}{(q+1)M_{q+1}}\le B\frac{M_{p}}{M_{p+1}},\qquad p\in\N_0.
$$
\end{itemize}
\end{defi}

Since (lc) and (snq) together imply that $\bm$ tends to infinity, every strongly regular sequence is a weight sequence.

Sometimes, it also appears the condition of \textit{derivation closedness}, for short (dc): there exists $A>0$ such that
\begin{equation*}
 M_{p+1}\leq A^{p+1} M_p, \qquad p\in\N_0.
\end{equation*}
It is clear that (mg) implies (dc).

The next characterization of (mg), already appearing in the works of H.~Komatsu~\cite[Prop. 3.6]{komatsu} and V.~Thilliez~\cite{thilliez03}, plays a fundamental role in many of our arguments.

\begin{lemma}\label{lemma.MG.associated.function}
Let $\M=(M_{p})_{p\in\N_0}$ be a weight sequence. The following are equivalent:
\begin{enumerate}[(i)]
\item $\M$ has (mg),
 \item For every real number with $s\ge1$, there exists $\rho(s)\ge1$ (only depending on $s$ and $\M$) such that
\begin{equation*}
h_{\M}(t)\le(h_{\M}(\rho(s)t))^{s}\qquad\hbox{for }t\ge0,
\end{equation*}
or, equivalently,
 \begin{equation}\label{ineq.def.moderate.growth.ommegaM} 
s\o_{\M}(t)\le\o_{\M}(\rho(s)t) \qquad\hbox{for }t\ge0.
\end{equation}
\item  There exist $H\ge1$ and $t_0>0$ (only depending on $\M$) such that
\begin{equation*}
h_{\M}(t)\le(h_{\M}(Ht))^{2}\qquad\hbox{for }t\leq 1/t_0,
\end{equation*}
or, equivalently,
 \begin{equation*}
2\o_{\M}(t)\le\o_{\M}(Ht) \qquad\hbox{for }t\ge t_0.
\end{equation*}
\end{enumerate}
\end{lemma}

\begin{exam}\label{exampleSequences}
We mention some interesting examples of weight sequences. In particular, those in (i) and (iii) appear in the applications of summability theory to the study of formal power series solutions for different kinds of equations.
\begin{itemize}
\item[(i)] The sequences $\M_{\a,\b}:=\big(p!^{\a}\prod_{m=0}^p\log^{\b}(e+m)\big)_{p\in\N_0}$, where $\a>0$ and $\b\in\R$, are strongly regular (in case $\b<0$, the first terms of the sequence have to be suitably modified in order to ensure (lc)).
    For $\b=0$, we have the best known example of strongly regular sequence, $\M_{\a,0}=(p!^{\a})_{p\in\N_{0}}$, called the \textit{Gevrey sequence of order $\a$}. It is immediate that $\o(\M_{\a,\b})=\a$.
\item[(ii)] The sequence $\M_{0,\b}:=(\prod_{m=0}^p\log^{\b}(e+m))_{p\in\N_0}$, with $\b>0$, is a weight sequence with (mg), but (snq) is not satisfied. In this case, $\o(\M_{0,\b})=0$.
\item[(iii)] For $q>1$, $\M_q:=(q^{p^2})_{p\in\N_0}$ is (lc), (dc) and (snq), but not (mg); $\o(\M_q)=\infty$.
\end{itemize}
\end{exam}

As it was proved in~\cite[Th. 3.4]{SanzFlat}, or as it can be deduced from the theory of O-regular variation (see~\cite[Remark 2.1.19]{JimenezPhD}), for every strongly regular sequence $\M$ one has $\o(\M)\in(0,\infty)$.

\subsection{Ultraholomorphic classes and the asymptotic Borel map}
\label{subsect.ultrahol.Borel.map}
Here we introduce the classes of holomorphic functions in a sectorial region which the sum of a summable formal power series will belong to.

We say a holomorphic function $f$ in a sectorial region $G$ admits the formal power series $\widehat{f}=\sum_{n=0}^{\infty}a_{n}z^{n}\in\C[[z]]$ as its $\bM-$\emph{asymptotic expansion} in $G$ (when the variable tends to 0) if for every $T\ll G$ there exist $C_T,A_T>0$ such that for every $p\in\N_0$ one has
\begin{equation*}
|f(z)-\sum_{n=0}^{p-1}a_nz^n |\leq C_TA_T^pM_{p}|z|^p,\qquad z\in T.
\end{equation*}
We will write $f\sim_{\bM}\widehat{f}$ in $G$, and $\widetilde{\mathcal{A}}_{\M}(G)$ will stand for the space of functions admitting $\bM-$asymptotic expansion in $G$.

As a consequence of Taylor's formula and Cauchy's integral formula for the derivatives, one can check that $f\in\widetilde{\mathcal{A}}_{\M}(G)$ if and only if for every $T\ll G$ there exist $C_T,A_T>0$ such that for all $p\in\N_0$ one has
\begin{equation*}
|f^{(p)}(z)|\leq C_TA_T^pp!M_{p},\qquad z\in T.
\end{equation*}
Moreover, if $f\sim_\M\sum^\oo_{p=0} a_p z^p$ then
for every $T\ll G$ and $p\in\N_0$ one has
$$ a_p=\lim_{ \genfrac{}{}{0pt}{}{z\to0}{z\in T}} \frac{f^{(p)}(z)}{p!},$$
and 
we see that the $\M-$asymptotic expansion $\widehat{f}$ is unique for $f$ and, moreover, $|a_p|\le C_TA_T^pM_{p}$ for every $p\in\N_0$. So, if we define
$$\C[[z]]_{\M}=\Big\{\widehat{f}=\sum_{n=0}^\oo a_nz^n\in\C[[z]]:\,\textrm{there exists $A>0$ with } \left|\,\ba \,\right|_{\M,A}:=\sup_{p\in\N_{0}}\displaystyle \frac{|a_{p}|}{A^{p}M_{p}}<\infty\Big\},
$$
it is natural to consider the \textit{asymptotic Borel map} $\widetilde{\mathcal{B}}$, sending a function $f\in\widetilde{\mathcal{A}}_{\M}(G)$ into its $\M-$asymptotic expansion $\widehat{f}\in\C[[z]]_{\M}$.



If $\M$ is (lc), $\widetilde{\mathcal{B}}$ is a homomorphism of algebras; if $\M$ is also (dc), $\widetilde{\mathcal{B}}$ is a homomorphism of differential algebras.

Note that if $\M\approx\L$, then $\widetilde{\mathcal{A}}_{\M}(G)=\widetilde{\mathcal{A}}_{\L}(G)$ and $\C[[z]]_{\M}=\C[[z]]_{\L}$.

We say $f\in\widetilde{\mathcal{A}}_{\M}(G)$ is \emph{flat} if $f\sim_{\bM}\widehat{0}$ in $G$, where $\widehat{0}$ stands for the null power series. The next characterization of flatness is due to V. Thilliez.

\begin{pro}[\cite{ThilliezSmooth10}, Prop.\ 4]\label{prop.thilliez.flat.function.M}
Let $\M$ be a weight sequence, $G$ be a sectorial region and $f$ be holomorphic in $G$. The following are equivalent:
\begin{enumerate}[(i)]
 \item $f\in\widetilde{\mathcal{A}}_{\M}(G)$ and $f$ is \emph{flat}, i.e., $f\sim_{\bM}\widehat{0}$ in $G$ where $\widehat{0}$ stands for the null power series.
 \item For every bounded proper subsector $T$ of $G$ there exist $c_1,c_2>0$ with
 $$|f(z)|\le c_1e^{-\o_{\M}(1/(c_2|z|))},\qquad z\in T. $$
\end{enumerate}
 \end{pro}

The existence or not of nontrivial flat functions was characterized in~\cite[Coro. 4.12]{SanzFlat} for classes defined from a strongly regular sequence $\M$ (see Definition~\ref{defiStronglyRegular}) such that the function $d_{\M}$ is a nonzero proximate order (Definition~\ref{OAD:1}). However, the result has been recently proved to be valid for any weight sequence. 

\begin{theo}[\cite{JimenezSanzSchindlInjectSurject}, Cor. 3.16]
\label{TheoWatsonlemma}
Let $\M$ be a weight sequence, $\ga>0$ and $G(d,\gamma)$ be a sectorial region of opening $\pi\gamma$. The following statements are equivalent:
\begin{itemize}
\item[(i)] $\widetilde{\mathcal{A}}_{\M}(G(d,\gamma))$ is \emph{quasianalytic}, i.e., it does not contain nontrivial flat functions (in other words, the Borel map is injective in this class).
\item[(ii)] $\gamma>\omega(\M)$.
\end{itemize}
\end{theo}

Since the bisecting direction is irrelevant for quasianalyticity, we will frequently consider only the case $d=0$.

\subsection{Sequences admitting a nonzero proximate order}
\label{subsect.SequencesAdmitProxOrder}

In this subsection, we first recall the notion of proximate order, appearing in the theory of growth of entire functions and developed, initially, by E. Lindel\"of and G. Valiron.

\begin{defi}\label{OAD:1}
We say a real function $\ro(r)$, defined on $(c,\oo)$ for some $c\ge 0$, is a {\it proximate order},
if the following hold:
 \begin{enumerate}[(A)]
  \item $\ro$ is continuous and piecewise continuously differentiable in $(c,\oo)$
  (that is, differentiable except possibly at a sequence of points, tending to infinity, at any
  of which it is continuous and has distinct finite lateral derivatives),\label{OA1:1}
  \item $\ro(r) \geq 0$ for every $r>c$,\label{OA2:1}
  \item $\lim_{r \to \oo} \ro(r)=\ro< \oo$, \label{OA3:1}
  \item $\lim_{r  \to \oo} r \ro'(r) \ln r = 0$. \label{OA4:1}
 \end{enumerate}
 In case the value $\rho$ in (C) is positive (respectively, is 0), we say $\rho(r)$ is a \textit{nonzero}
 (resp. \textit{zero}) proximate order.
\end{defi}

\begin{rema}\label{rema.prox.order.RV}
The classical theory of proximate orders (see \cite[Th. 2.2]{GoldbergOstrowskii}) guarantees that if $\ro(r)$ is a proximate order with limit $\ro\geq 0$ at infinity, then the function $V(r)=r^{\ro(r)}$ is of \emph{regular variation}, that is,
$$\lim_{r\to\oo}\frac{V(sr)}{V(r)}=s^\ro,$$
with uniform convergence in the compact intervals $[a,b]\en(0,\oo)$.
\end{rema}

G.~Valiron~\cite{valiron13} showed that, up to asymptotic equivalence, the function  $r^{\ro(r)}$ has an analytic continuation to a sector in the complex plane containing the positive real axis (see \cite[Th. \ 7.4.3]{bingGoldTeug}). We will use the next improved result, due to L.S.~Maergoiz, which provides holomorphic functions in an arbitrary sector bisected by the positive real axis whose restriction to $(0,\oo)$ is real, have a growth at infinity specified by a prescribed proximate order and satisfy several regularity properties.
%

\begin{theo}[\cite{Maergoiz}, Th.\ 2.4]\label{Th.Maergoiz.analytic.PO}
Let $\ro(t)$ be a nonzero proximate order and $\ro=\lim_{t \to \oo} \ro(t)$. For every $\ga>0$ there exists an analytic
function $V(z)$ in $S_\ga$ such that:
  \begin{enumerate}[(I)]
  \item  For every $z \in S_\ga$,
 \begin{equation*}
    \lim_{t \to \infty} \frac{V(zt)}{V(t)}= z^{\ro},
  \end{equation*}
uniformly in the compact sets of $S_\ga$ (that is, $V$ is regularly varying in $S_\gamma$).
\item $\overline{V(z)}=V(\overline{z})$ for every $z \in S_\ga$ (where, for $z=(|z|,\arg(z))$, we put $\overline{z}=(|z|,-\arg(z))$).
\item $V(t)$ is positive in $(0,\infty)$, strictly increasing and $\lim_{t\to 0}V(t)=0$.
\item The function $r\in\R\mapsto V(e^r)$ is strictly convex (i.e. $V$ is strictly convex relative to $\log(r)$).
\item The function $\log(V(t))$ is strictly concave in $(0,\infty)$.
\item  The function $\ro_V(t):=\log( V(t))/\log(t)$, $t>0$, is a proximate order equivalent to $\ro(t)$, that is,
$$\lim_{t\to\oo} V(t)/t^{\ro(t)}=\lim_{t\to\oo} t^{\ro_V(t)} / t^{\ro(t)} = 1.$$
    \end{enumerate}
\end{theo}

%
 Given $\ga>0$ and $\ro(t)$  a nonzero proximate order, $ MF(\ga,\ro(t))$ denotes the set of \emph{Maergoiz functions} $V$ defined in $S_{\ga}$ and satisfying the conditions (I)-(VI) of Theorem~\ref{Th.Maergoiz.analytic.PO}.%

Suppose $\ro(t)$ $(t\geq c\geq0)$ is a nonzero proximate order tending to $\ro>0$ at infinity, then the function $V(t)=t^{\ro(t)}$ is strictly increasing for $t>R$ large enough. The inverse function $t=U(s)$, defined for every $s>V(R)$, is such that $\ro^*(s):=\log( U(s) )/ \log (s)$ is a proximate order
which tends to $1/\ro$ as $s$ tends to $\oo$ (see~\cite[Property 1.8]{Maergoiz}), and it is called the \textit{proximate order conjugate to $\ro(t)$}. This conjugate proximate order can be also extended, up to equivalence, to an analytic function.

\begin{theo}[\cite{Maergoiz}, Th.\ 2.6]\label{th.inv.analytic.PO}
Let $\ro(t)$ be a nonzero proximate order, $\ga>0$ and $V \in MF(\ga, \ro(t) )$. Let $t=U(s)$, defined for all $s>0$, be the
function inverse to $s=V(t)$, for every $t\in(0,\oo)$, and let $\ro^*(s)$ be the proximate order conjugate to $\ro(t)$.
Then $\ln U(s)/ \ln s$ is a proximate order equivalent to $\ro^* (s)$, and the function $U(s)$ admits an analytic
continuation to a function $U(W)$ in a domain $T \subset S_{\ro\ga}$, symmetric relative to the real axis and such that for $\b<\ga$  there exists $R_\b>0$ such that the domain $T$ contains $S_{\ro\b} \cap \{|z|>R_\b\}$. Furthermore, the function $U$ satisfies, in its domain, the
properties (I)-(VI) in Theorem~\ref{Th.Maergoiz.analytic.PO} of the functions of the class $MF (\ro\ga, \ro^{*}(s))$.
\end{theo}

Many of the results in~\cite{lastramaleksanz15} about $\M-$summability were initially stated for strongly regular sequences $\M$ such that $d_{\M}$ is a proximate order (\emph{a fortiori} a nonzero proximate order, as deduced from~\cite[Th. 3.4]{SanzFlat}), since then Maergoiz functions were available. We know now how to characterize those weight sequences for which $d_{\M}$ is a nonzero proximate order.

\begin{theo}[\cite{JimenezSanzSchindl},\, Th. 3.6]\label{theorem.charact.prox.order.nonzero}
Let $\M$ be a weight sequence. The following are equivalent:
\begin{enumerate}[(a)]
 \item $d_{\M}(t)$ is a proximate order with $\lim_{t\to\infty}d_{\M}(t)\in(0,\infty)$.
 \item There exists $\lim_{p\to\infty} \log\big(m_p/M_p^{1/p}\big)\in(0,\infty)$.
 \item $\m$ is \textit{regularly varying}  with a positive index of regular variation, i.e., there exists $\o>0$ (called the \textit{index of regular variation} of $\m$) such that
\begin{equation*}
\lim_{p\to\oo} \frac{ m_{\lfloor \lambda p\rfloor}}{m_p}=\lambda^\o
\end{equation*}
 for every $\lambda>0$.
 \item There exists $\o>0$ such that for every natural number $\ell\geq 2$,
 $$\lim_{p\to\oo} \frac{m_{\ell p}}{m_p}=\ell^\o.$$
\end{enumerate}
In case any of these statements holds, the value of the limit mentioned in (b), that of the index mentioned in (c), and that of the constant $\o$ in (d) is $\o(\M)$,
and the limit in (a) is $1/\o(\M)$.
\end{theo}

A less restrictive condition on the sequence $\M$, namely the admissibility of a proximate order, is indeed sufficient for our purposes. This condition, defined below, is evidently stable under equivalence of sequences (unlike the condition of $d_{\M}$ being a nonzero proximate order, see~\cite[Example 3.12]{JimenezSanzSchindl}), what is desirable if we take into account that $\M-$summability and $\L-$summability amount to each other whenever $\M\approx\L$.

\begin{theo}[\cite{JimenezSanzSchindl},\, Th. 4.14; \cite{JimenezPhD}, Th.\ 2.2.19]\label{theo.charact.admit.p.o}
Let $\M$ be a weight sequence. The following conditions are equivalent:
\begin{enumerate}
 \item[(a)] $\M$ \emph{admits a nonzero proximate order}, i.e., there exist a nonzero proximate order $\ro(t)$
 and constants $C$ and $D$ such that
\begin{equation}\label{eqAdmitsProximateOrder}
C\leq \log(t)(\ro(t)-d_{\M}(t)) \leq D, \qquad t\textrm{ large enough}.
\end{equation}
 \item[(b)] There exist a weight sequence $\L$ equivalent to $\M$ and such that $d_{\L}(t)$ is a nonzero proximate order.
\item[(c)] There exist $\o\in(0,\oo)$ and bounded sequences  of real numbers $(b_p)_{p\in\N}$, $(\eta_p)_{p\in\N}$  such that $(\eta_p)_{p\in\N}$ converges to $\o$ and we can write
  \begin{equation*}
  m_p=\exp  \left(b_{p+1} + \sum^{p+1}_{j=1}\frac{\eta_j}{j}\right),\quad p\in\N_0.
 \end{equation*}
\end{enumerate}
In case the previous holds, $\lim_{t\to\oo} d_{\M} (t) =\lim_{t\to\oo} d_{\L} (t) = 1/\o=1/\o(\M)$.
\end{theo}

All the weight sequences admitting a nonzero proximate order are strongly regular, what shows that some of the hypotheses in~\cite{SanzFlat,lastramaleksanz15} were redundant.

\begin{cor}[\cite{JimenezSanzSchindl}, Cor. 3.10 and Remark 4.15]\label{coroConseqRegularVariation}
Let $\M$ be a weight sequence admitting a nonzero proximate order.
Then,  $\M$ is strongly regular and
\begin{equation}\label{limit.logmp.over.logp.omega}
\lim_{p\to\infty} \frac{\log(m_p)}{\log(p)}=\o(\M).
\end{equation}
\end{cor}

Although not every strongly regular sequence admits a nonzero proximate order, as shown in~\cite[Example 4.16]{JimenezSanzSchindl}, admissibility holds true for every strongly regular sequence appearing in applications.

%
%
%

The next lemma is valid under weaker conditions (see \cite[Lemmma 3.18]{Schindl16} by G. Schindl), but this version is enough for our purpose.

\begin{lemma}\label{lemma.M.scaling}
Let $\M$ be a weight sequence admitting a  nonzero proximate order, then for any $A>0$ there exist $t_A,E,F>0$ such that
$$\o_{\M}(Et)<A\o_{\M}(t)<\o_{\M}(Ft), \qquad t>t_A, $$
and for any $B>0$ there exist constants $t_B,G,H>0$
such that
$$G\o_{\M}(t)<\o_{\M}(Bt)<H\o_{\M}(t), \qquad t>t_B. $$
\end{lemma}
\begin{proof}
Since $\M$ admits a nonzero proximate order $\ro(t)$,  if we write $V(t)=t^{\ro(t)}$, we have that there exist $t_0>0$ and constants $C$ and $D$ such that
\begin{equation}\label{ineq.V.and.M}
e^C \o_{\M}(t)\leq V(t)\leq e^D \o_{\M}(t), \qquad t\geq t_0.
\end{equation}
We fix $E,F>0$ such that $E^{1/\o(\M)}<A e^{C-D}$ and $F^{1/\o(\M)}>A e^{D-C}$. By Theorem~\ref{theo.limsup.dM.omega},
we know that $\ro(t)$ is a proximate order with $\lim_{t\to\oo} \ro(t)=1/\o(\M)$. Then, by Remark~\ref{rema.prox.order.RV}, V(t)
is a regularly varying function of index $1/\o(\M)$ and we see that
$$\lim_{t\to\oo}\frac{V(Et)}{V(t)}=E^{1/\o(\M)}, \quad \lim_{t\to\oo}\frac{V(Ft)}{V(t)}=F^{1/\o(\M)}.$$
Hence, by the election of $E$ and $F$, we deduce that it exists $t_A>\max(t_0,t_0/E,t_0/F) $ such that $V(Et)<A e^{C-D}V(t)$ and $V(Ft)>A e^{D-C}V(t)$ for every $t>t_A$. Using \eqref{ineq.V.and.M}, we conclude for $t>t_A$ that
$$A\o_{\M}(t)\geq Ae^{-D}V(t)=Ae^{-D+C}e^{-C} V(t)\geq  e^{-C}V(Et)\geq \o_{\M}(Et),$$
$$A\o_{\M}(t)\leq Ae^{-C}V(t)=Ae^{-C+D}e^{-D} V(t)\leq e^{-D}V(Ft)\leq \o_{\M}(Ft).$$
We fix $G,H>0$ such that  $B^{1/\o(\M)}e^{D-C}<H $ and $B^{1/\o(\M)}e^{C-D} >G $ an we proceed in a similar way to prove the second part.
\end{proof}

\subsection{$\M-$summability and moment summability methods}
\label{subsect.M.summability}


Every function $f$ in a quasianalytic Carleman ultraholomorphic class is determined by its asymptotic expansion $\widehat f$. This fact motivates the concept, developed in~\cite{lastramaleksanz15,SanzFlat,SanzAsymptoticAnalysis}, of $\M-$summability of formal (i.e. divergent in general) power series in a direction, so generalizing the by-now classical and powerful tool of $k-$summability of formal Gevrey power series, introduced by J.-P. Ramis~\cite{Ramis78,Ramis80}. Observe that the definition only makes sense whenever injectivity of the Borel map is available, and Watson's Lemma imposes then that $\o(\M)<\infty$, an assumption always implicit whenever $\M-$summability is considered, and automatically satisfied for strongly regular sequences, or its subclass of weight sequences admitting a nonzero proximate order, for which the theory is completely satisfactory.

\begin{defi}\label{def.Msummable.direction}
 Let $d\in\R$ and $\M$ be a weight sequence. We say $\widehat{f}=\sum_{p\geq0} a_p z^p\in \C[[z]]$ is \textit{$\M-$summable in direction $d$} if there exist a sectorial region $G=G(d,\gamma)$, with $\gamma>\o(\M)$, and a function $f\in\widetilde{\mathcal{A}}_{\M}(G)$ such that $f\sim_\M \widehat{f}$.
\end{defi}

According to Theorem~\ref{TheoWatsonlemma}, $f$ is unique with the property stated,
and it will be denoted $f=\mathcal{S}_{\M,d}\widehat{f}$, the \textit{$\M-$sum of $\widehat{f}$ in direction $d$}.

Some basic properties of $\M-$summable series in a direction are straightforward consequences of Watson's Lemma and the elementary properties of $\M-$asymptotics.

\begin{lemma}\label{basic.properties.Msumm}
 Let $\M$ be a weight sequence, the following hold:
 \begin{enumerate}[(i)]
  \item Let $\widehat{f}$ be convergent (i.e., its radius of convergence is not 0). Then for every $d$, the series $\widehat{f}$ is $\M-$summable in direction $d$
  and $\mathcal{S}_{\M,d}\widehat{f}(z)= \mathcal{S}\widehat{f} (z) $ for every $z$ where both sides are defined, where $\mathcal{S}$ maps each convergent power series to its natural sum.
   \item If $\widehat f$ is $\M-$summable in direction $d$ for every $d\in (\a,\b)$ with $\a<\b$, then
  $$\mathcal{S}_{\M,d_1}\widehat{f}(z)=\mathcal{S}_{\M,d_2}\widehat{f}(z), \qquad d_1,d_2\in(\a,\b),$$
  whenever the corresponding domains intersect.
  \item Let $\widehat{f}$ be $\M-$summable in direction $d$, there exists $\ep>0$ such that $\widehat{f}$ is $\M-$summable
  in all directions $\widetilde{d}$ with $|\widetilde{d}-d|<\ep$.
  \item For $\widetilde{d}=d+2\pi$, the $\M-$summability of $\widehat{f}$ in direction $d$ is equivalent to
  the $\M-$summability of $\widehat{f}$ in direction $\widetilde{d}$. Moreover, we have that
  $$\mathcal{S}_{\M,\widetilde{d}}\widehat{f}(z)=\mathcal{S}_{\M,d}\widehat{f}(ze^{-2\pi i}),$$
  where both functions are defined.
 \end{enumerate}
\end{lemma}

In particular (i) in the last Lemma says that the $\M-$summability method is \textit{regular}.


From Lemma~\ref{basic.properties.Msumm}.(iv), when considering $\M-$summability we can identify directions $d$ that
differ by integer multiples of $2\pi$.
By Lemma~\ref{basic.properties.Msumm}.(iii), the set of directions for which a formal power series is not $\M-$summable is closed. Special attention deserves the case in which this set is finite (mod $2\pi$).

\begin{defi}
$\C\{z\}_{\M,d}$ stands for the set of formal power series $\widehat{f}$ which are $\M-$summable in direction $d$.
$\C\{z\}_{\M}$ will be the set of formal power series $\widehat{f}$  which are $\M-$summable in every direction except for a finite set of \emph{singular directions} (mod $2\pi$), denoted by $\hbox{sing}_{\M}(\widehat{f})=\{d_1,\dots,d_m\}$.
\end{defi}

We write $\C\{z\}$ for the set of convergent formal series, that is, those with a positive radius of convergence. Note that $\C\{z\}\en \C\{z\}_{\M,d} \en \C[[z]]_{\M}$. Due to its importance in the proof of the Tauberian results in Section~\ref{sect.tauberian.theorems}, we include the proof of (i) among the following properties of these sets.

\begin{pro}\label{prop.summability.algebras}
Let $\M$ be a weight sequence. Then,
\begin{enumerate}[(i)]
\item If $\widehat f\in\C\{z\}_{\M}$ and $\hbox{sing}_{\M}(\widehat{f})=\emptyset$, then $\widehat{f}$ is convergent.
\item $\C\{z\}_{\M,d}$, $\C\{z\}_{\M}$ are algebras. Moreover, if $\M$ is (dc) they are differential algebras.
\end{enumerate}
\end{pro}

\begin{proof}
 \begin{enumerate}
  \item[(i)] Using Lemma~\ref{basic.properties.Msumm}.(ii), we know that $f(z):=\mathcal{S}_{\M,d}\widehat{f}(z)$ is well-defined (independent from direction $d$). By Lemma~\ref{basic.properties.Msumm}.(iv) we know that $f$ is single-valued, then $f$ can be expanded into a convergent power series about the origin. By the uniqueness of the asymptotic expansion, we
  deduce that $\widehat{f}$ coincides with this convergent power series.%
\end{enumerate}%
\end{proof}

The ideas in the theory of general moment summability methods put forward by W.~Balser in~\cite{Balser2000} are now followed in order to recover $f$ from $\widehat{f}$, what requires the existence of a pair of kernel functions, say $e$ and $E$, with suitable asymptotic and growth properties, in terms of which to define formal and analytic Laplace- and Borel-like transforms.



\begin{defi}~\label{defikernelMsumm}
 Let $\M$ be a strongly regular sequence with $\o(\M)<2$. A pair of complex functions $e$, $E$ are said to be \textit{kernel functions for $\M-$summability} if:
 \begin{enumerate}
  \item[(\textsc{i})] $e$ is holomorphic in $S_{\o(\M)}$.
  \item[(\textsc{ii})] $z^{-1}e(z)$ is locally uniformly integrable at the origin, i.e., there exists $t_0>0$, and for every
  $z_0\in{S_{\o(\M)}}$ there exists a neighborhood $U$ of $z_0$, $U\en S_{\o(\M)}$, such that the integral
  $\int^{t_0}_{0} t^{-1}\sup_{z\in U} |e(t/z)|dt$ is finite.
  \item[(\textsc{iii})] For every $\ep>0$ there exist $c,k>0$ such that
  \begin{equation}\label{eq.bound.e}
|e(z)|\le ch_{\bM}\left(\frac{k}{|z|}\right)=c\,e^{-\o_{\M}(|z|/k)},\qquad z\in S_{\omega(\M)-\varepsilon},
\end{equation}
  where $h_\M$ and $\o_{\M}$ are the functions associated with $\M$ defined in Subsection~\ref{subsectLCsequences}.
\item[(\textsc{iv})] For $x\in\R$, $x>0$, the values of $e(x)$ are positive real.
\item[(\textsc{v})] If we define the \textit{moment function} associated with $e$,
  $$m_{e}(\lambda):=\int^{\oo}_{0} t^{\lambda-1}e(t)dt,\qquad \Re(\lambda)\geq0,$$
  from $(I)-(IV)$ we see that $m_{e}$ is continuous in $\{\Re(\lambda)\geq0\}$, holomorphic in $\{\Re(\lambda)>0\}$, and
  $m_e(x)>0$ for every $x\geq 0$. Then, the function $E$ given by
  $$E(z)= \sum^\oo_{n=0}\frac{z^n}{m_e(n)}, \qquad z\in\C,$$
  is entire, and there exist $C,K>0$ such that
  \begin{equation}\label{eq.bound.E}
|E(z)|\le\displaystyle \frac{C}{h_{\M}(K/|z|)}=Ce^{\o_{\M}(|z|/K)},\quad z\in\C.
\end{equation}
\item[(\textsc{vi})] $z^{-1}E(1/z)$ is locally uniformly integrable at the origin in the sector $S(\pi,2-\o(\M))$, in the sense that there
exists $t_0>0$, and for every $z_0\in S(\pi,2-\o(\M))$ there exist a neighborhood $U$ of $z_0$, $U\en S(\pi,2-\o(\M))$, such
that the integral $\int^{t_0}_{0} t^{-1}\sup_{z\in U} |E(z/t)|dt$ is finite.
  \end{enumerate}
\end{defi}


\begin{rema}\label{rema.ramif.kernels}
 \begin{enumerate}[(i)]
  \item According to Definition~\ref{defikernelMsumm}(\textsc{v}), the knowledge of $e$ is
enough to determine $E$, so in the sequel we will frequently omit the function $E$ in our statements.

\item The case $\omega(\M)\ge 2$, for which the previous  condition (\textsc{vi}) does not make sense, is dealt with by a ramification process described in~\cite[Remark 3.5(iii)]{lastramaleksanz15}.
\end{enumerate}

\end{rema}

\begin{rema}
 \label{rema.def.exist.kernels}
\begin{enumerate}
 \item[(i)]  Note that Definition~\ref{defikernelMsumm} can be given for arbitrary weight sequences with $\o(\M)\in(0,2)$. If we assume that $\M$ is (dc) and there exists an $\M-$summability kernel, one can show, following the ideas in~\cite{lastramaleksanz12,SanzFlat}, that $\M$ satisfies (snq), so this condition is obtained automatically.
However, (mg) seems not to be deduced from the definition of the kernels, while it is essential for many of the arguments to come. In any case, since the existence of such kernels is only guaranteed for a subfamily of strongly regular sequences, those admitting a nonzero proximate order (see (ii) in this remark), the definition in~\cite{lastramaleksanz15} has been kept.

\item[(ii)] The existence of such kernels, proved in~\cite{lastramaleksanz15}
    whenever the function $d_{\M}(t)$ is a nonzero proximate order, 
    can be also obtained whenever $\bM$ admits a nonzero proximate order, i.e.,
there exists a nonzero proximate order $\ro(t)$ satisfying the estimates~\eqref{eqAdmitsProximateOrder} or, equivalently, there exist positive constants $A$ and $B$ such that
  $$A\leq \frac{t^{\ro(t)}}{\o_{\M}(t)} \leq B, \qquad \text{for}\,\,\, t \,\,\,\text{large enough}.$$%
In this situation we know that $\lim_{t\to\oo} d_{\M} (t)=\lim_{t\to\oo} \ro(t)=1/\o(\M)$. 
Then, for every $V\in MF (2\o(\M),\ro(t))$ one may consider the function $e_V$ defined in $S_{\omega(\M)}$ by
$$
e_V(z)= z\,\exp(-V(z)),
$$
and the arguments in the proof of~Theorem~4.8 in~\cite{lastramaleksanz15} can be mimicked to deduce that $e_V$ is a kernel of $\M-$summability.
We only record for the future that, firstly,
for every $\ep\in(0,\o(\M))$ there exist $b,b_1>0$ such that for any $z\in S_{\o(\M)-\ep}$,
\begin{equation}\label{eq.eV.bound.IIB}
 |e_{V}(z)|\leq |z| \exp (-\Re (V(z)))\leq |z| \exp (-b V(|z|))\leq |z| \exp (-b_1\o_{\M}(|z|))\leq |z|
\end{equation}
(observe that $\o_{\M}(|z|)\geq0$), what clearly implies (\textsc{ii}) in Definition~\ref{defikernelMsumm}.
%
Secondly, the function $E_V(z)=\sum_{n=0}^\infty z^n/m_V(n)$, $z \in \C$, 
%
%
is such that for every 
$0<\varepsilon<\pi/2(2-\o(\M))$ we have, uniformly as $|z|\to \infty$ and in Landau's notation,
\begin{equation}\label{eq.EV.bound.VIB}
  E_V(z)=  O\left(\frac{1}{|z|}\right),\qquad   \frac{\pi}{2}\o(\M)+\varepsilon\le |\arg z| \le \pi,
 \end{equation}
what implies that also condition (\textsc{vi})  in Definition~\ref{defikernelMsumm} is fulfilled.

\item[(iii)] In Balser's theory for Gevrey sequences $\bM_{1/k}=(p!^{1/k})_{p\in\N_0}$ (see~\cite[Sect.\ 5.5]{Balser2000}), the classical example of kernels is given by $e_k(z)=kz^k\exp(-z^k)$;  the moment function is $m_e(\lambda)=\Gamma(1+\lambda/k)$, $\Re(\lambda)\ge 0$, where $\Gamma$ is Euler's function, and the Borel kernel $E_k$ is a classical Mittag-Leffler function. Observe that the sequences $\bM_{1/k}$ and $(m_e(p))_{p\in\N_0}$ are equivalent, and this is not a coincidence at all (see Proposition~\ref{prop.Moment.equiv.M}).
\end{enumerate}
\end{rema}

Next, we recall a key result by H.~Komatsu that characterizes the growth of a entire function in terms of that of its Taylor coefficients; it was useful in the proof of Proposition~\ref{prop.Moment.equiv.M} and will be employed afterwards.

\begin{pro}[\cite{komatsu}, Prop.~4.5]\label{propKomatsu}
Let $\M$ be a weight sequence. Given an entire function $F(z)=\sum_{n=0}^\infty a_nz^n$, $z\in\C$, the following statements are equivalent:
\begin{itemize}
\item[(i)] There exist $C,K>0$ such that
$|F(z)|\le\displaystyle Ce^{\o_{\M}(K|z|)}$, $z\in\C$.
\item[(ii)] There exist $c,k>0$ such that for every $n\in\N_0$, $|a_n|\le ck^n/M_n$.
\end{itemize}
\end{pro}

Given a kernel $e$ for $\M-$summability, the associated \emph{sequence of moments} is $\mf_e:=(m_e(p))_{p\in\N_0}$. The following result, important for the development of a satisfactory summability theory, ensures that the classes of functions and formal power series defined respectively by $\M$ and $\mf_e$ coincide. In the proof, the estimates, for the kernels $e$ and $E$ appearing in (\ref{eq.bound.e}) and (\ref{eq.bound.E}), respectively, are crucial.

\begin{pro}[\cite{SanzFlat}, Prop.\ 5.7]\label{prop.Moment.equiv.M}
Let $e$ be a kernel for $\M-$summability,
then
$\M\approx\mathfrak{m}_e$.
\end{pro}

\begin{rema}\label{rema.moment.equiv}
For any kernel $e$ for $\M-$summability, up to multiplication by a constant scaling factor, one may
always suppose that $m_e(0)=1$. Moreover,
$\mathfrak{m}_e=(m_e(p))_{p\in\N_0}$ is (lc) as a consequence of H\"older's inequality. Then, from the equivalence between $\M$ and $\mathfrak{m}_e$ we deduce the strong regularity of $\mf_e$.
\end{rema}


In the rest of this subsection we recall from~\cite{lastramaleksanz15} how Laplace- and Borel-like transforms are defined from the $\M-$summability kernels, summarizing their main properties, and how they allow for the explicit reconstruction of the sum of a summable formal power series in a direction.
The first definition resembles that of functions of exponential growth of order $1/k$.
For convenience, we will say a holomorphic function $f$ in a sector $S$ is {\it continuous at the origin}
if $\lim_{z\to 0,\ z\in T}f(z)$ exists for every $T\ll S$.

\begin{defi}\label{def.Mgrowth}
Let $\bM$ be a weight sequence, and consider
an unbounded sector $S$ in $\mathcal{R}$.
The set $\mathcal{O}^{\bM}(S)$ consists of the holomorphic functions $f$ in $S$, continuous at the origin
and having \emph{$\M-$growth} in $S$, i.e. such that for every unbounded proper subsector $T$ of $S$ there exist $r,c,k>0$ such that for every $z\in T$ with $|z|\ge r$ one has
\begin{equation}\label{eq.Mgrowth}
|f(z)|\le\frac{c}{h_{\bM}(k/|z|)}=ce^{\o_{\M}(|z|/k)}.
\end{equation}
\end{defi}

Since continuity at 0 has been asked for, $f\in\mathcal{O}^{\bM}(S)$ implies that for
every unbounded proper subsector $T$ of $S$ there exist $c,k>0$ such that \eqref{eq.Mgrowth} holds for every $z\in T$.

 Given  a sector $S=S(d,\alpha)$, a kernel $e$ for $\M-$summability and $f\in\mathcal{O}^{\bM}(S)$, for any direction $\tau$ in $S$ we define the operator $T_{e,\tau}$  sending $f$ to its \textit{$e-$Laplace transform in direction $\tau$}, defined as
\begin{equation}\label{equadefitranLapl}
(T_{e,\tau}f)(z):=\int_0^{\infty(\tau)}e(u/z)f(u)\frac{du}{u},\quad
|\arg(z)-\tau|<\o(\M)\pi/2,\ |z|\textrm{ small enough},
\end{equation}
where the integral is taken along the half-line parametrized by $t\in(0,\infty)\mapsto te^{i\tau}$.
We have the following result.

\begin{pro}[\cite{lastramaleksanz15},\, Prop.\ 3.11]\label{prop.hol.lap}
For a sector $S=S(d,\alpha)$ and $f\in\mathcal{O}^{\bM}(S)$, the family $\{T_{e,\tau}f\}_{\tau\textrm{\,in\,}S}$ defines a holomorphic function $T_{e}f$ in a sectorial region $G(d,\a+\o(\M))$.
\end{pro}

We now define the generalized Borel transforms.

\begin{defi}\label{def.e.Borel}
Suppose $\o(\M)<2$, and let $G=G(d,\a)$ be a sectorial region
with $\a>\o(\M)$, and $f:G\to \C$ be holomorphic in $G$ and continuous at 0.
For $\tau\in\R$ such that $|\tau-d|<(\a-\o(\M))\pi/2$ we may consider a path
$\delta_{\o(\M)}(\tau)$ in $G$ like the ones used in the classical Borel transform, consisting of a segment from the origin to a point $z_0$ with $\arg(z_0)=\tau+\o(\M)(\pi+\varepsilon)/2$ (for some
suitably small $\varepsilon\in(0,\pi)$), then the circular arc $|z|=|z_0|$ from $z_0$ to
the point $z_1$ on the ray $\arg(z)=\tau-\o(\M)(\pi+\varepsilon)/2$ (traversed clockwise), and
finally the segment from $z_1$ to the origin.

Given kernels $e,E$ for $\M-$summability, we define the operator $T^{-}_{e,\tau}$ sending $f$ to its \textit{$e-$Borel transform in direction $\tau$}, defined as
$$
(T^{-}_{e,\tau}f)(u):=\frac{-1}{2\pi i}\int_{\delta_{\o(\M)}(\tau)}E(u/z)f(z)\frac{dz}{z},\quad
u\in S(\tau,\varepsilon_0), \quad \varepsilon_0\textrm{ small enough}.
$$
\end{defi}

\begin{pro}[\cite{lastramaleksanz15},\, Prop.\ 3.12]\label{prop.hol.Bor}
For $G=G(d,\a)$ and $f:G\to \C$ as above, the family $\{T^{-}_{e,\tau}f\}_{\tau}$, where $\tau$ is a real number such that $|\tau-d|<(\a-\o(\M))\pi/2$, defines a holomorphic function $T^{-}_{e}f$ in the sector $S=S(d,\a-\o(\M))$. Moreover, $T^{-}_{e}f$ is of $\M-$growth in $S$.
\end{pro}

In case $\o(\M)\ge 2$, the treatment is similar to that
in~\cite[p.\ 90]{Balser2000} (see also~\cite[p.~1186]{lastramaleksanz15}).

One can compute the $e-$transforms of a monomial.

\begin{pro}[\cite{lastramaleksanz15}, p. 1187]\label{prop.trans.monom}
Given $\lambda\in\C$ with $\Re(\lambda)\ge0$, the function $f_{\lambda}(z)=z^{\lambda}$  belongs to $\mathcal{O}^{\bM}(S)$ and
$T_ef_{\lambda}(z)=m_e(\lambda)z^{\lambda}$,
$T^{-}_ef_{\lambda}(u)=u^{\lambda}/m_e(\lambda)$.
\end{pro}

This justifies the definition of
%
%
%
the formal $e-$Laplace and $e-$Borel transforms $\widehat{T}_{e},\widehat{T}_e^{-}:\C[[z]]\to\C[[z]]$, given respectively by
$$\widehat{T}_{e}\big(\sum_{p=0}^{\infty}a_{p}z^{p}\big):= \sum_{p=0}^{\infty}m_e(p)a_{p}z^{p},
\quad \widehat{T}_e^{-}\big(\sum_{p=0}^{\infty}a_{p}z^{p}\big):= \sum_{p=0}^{\infty}\frac{a_{p}}{m_e(p)}z^{p}.$$
%

The next result lets us know how these analytic and formal transforms interact with general asymptotic expansions.
Given two sequences of positive real numbers $\M=(M_p)_{p\in\N_0}$ and $\M'=(M'_p)_{p\in\N_0}$,
we consider the sequences $\M\cdot\M':=(M_pM'_p)_{n\in\N_0}$ and $\M'/\M:=(M'_p/M_p)_{p\in\N_0}$.

\begin{theo}[\cite{lastramaleksanz15},\, Th. 3.16]\label{teorrelacdesartransfBL}
Suppose $\M$ is a sequence and $e$ is a kernel of $\M-$summability. For any sequence $\M'$ of positive real numbers the following hold:
\begin{itemize}
\item[(i)] If $f\in\mathcal{O}^{\bM}(S(d,\a))$
and $f\sim_{\M'}\widehat{f}$, then $T_ef\sim_{\M\cdot\M'}\widehat{T}_e\widehat{f}$ in a
sectorial region $G(d,\a+\o(\M))$.
\item[(ii)] If $f\sim_{\M'}\widehat{f}$ in a sectorial region $G(d,\a)$ with $\a>\o(\M)$, then
    $T^{-}_ef\sim_{\M'/\M}{\widehat{T}}^{-}_e\widehat{f}$ in the sector $S(d,\a-\o(\M))$.
\end{itemize}
\end{theo}

Note that if both sequences are weight sequences, $\M\cdot\M'$ is again a weight sequence, but $\M'/\M$ might not be.
This last theorem motivates the study of the quotient and the product sequence achieved in Subsection~\ref{subsect.product.quotient.seq}.



Let $e$ be a kernel of $\M-$summability. Since $\mathfrak{m}_e$ is strongly regular and equivalent to $\M$ (see Proposition~\ref{prop.Moment.equiv.M} and Remark~\ref{rema.moment.equiv}), one has
$\C[[z]]_{\M}=\C[[z]]_{\mathfrak{m}_e}$
and $\mathcal{O}^{\mathfrak{m}_e}(S)=\mathcal{O}^{\M}(S)$ (see  \eqref{eq:equiv.seq.associated.ommega}). This justifies the following definition:

We say $\widehat{f}=\sum_{p\ge 0} a_p z^p$ is \textit{$e-$summable in direction $d\in\R$} if:
 \begin{itemize}
 \item[(i)] $\widehat{f}\in\C[[z]]_{\mathfrak{m}_e}$, so  $g:=
\widehat{T}_e^{-}\widehat{f}= \displaystyle \sum_{p\ge 0}\frac{a_p}{m_e(p)}z^p$ converges in a disc and
\item[(ii)] $g$ admits analytic continuation in a sector $S=S(d,\varepsilon)$ for some $\varepsilon>0$, and $g\in\mathcal{O}^{\mathfrak{m}_e}(S)$.
\end{itemize}

The next result states the equivalence between $\M-$summability and $e-$summability in a direction, and provides a way to recover the $\M-$sum in a direction of a summable power series by means of the formal and analytic transforms previously introduced.

\begin{theo}[\cite{lastramaleksanz15},\, Th. 3.18]\label{th.Msummable.equiv.esummable}
Suppose given a strongly regular sequence $\M$ admitting a kernel of $\M-$summability, a direction $d\in\R$ and a formal power series $\widehat f=\sum_{p\ge 0}a_p z^p$. The following are equivalent:
\begin{itemize}
\item[(i)] $\widehat f$ is $\M-$summable in direction $d$.
\item[(ii)] For every kernel $e$ of $\M-$summability, $\widehat{f}$ is $e-$summable in direction $d$.
\item[(iii)] For some kernel $e$ of $\M-$summability, $\widehat{f}$ is $e-$summable in direction $d$.
\end{itemize}
In case any of the previous holds, for any kernel $e$ of $\M-$summability we have (after analytic continuation)
\begin{equation*}
 \mathcal{S}_{\M,d}\widehat{f}=T_e(\widehat{T}_e^{-}\widehat{f}).
\end{equation*}
\end{theo}

 In case $\M=\M_{1/k}$, the summability methods described are just the classical $k-$summability and $e-$summability (in a direction) for kernels $e$ of order $k>0$, as defined by W. Balser.

\section{Tauberian theorems}\label{sect.tauberian.theorems}

Classical Tauberian theorems (\cite[Th. 2.2.4.2]{Malgrange95} and \cite[Th. 37]{Balser2000}) serve to compare the processes of
$k-$summability for various $k$; these theorems are strongly related to multisummability. In this section,
we will see what can be said about the connection between the algebras $\C\{z\}_{\M}$ and $\C\{z\}_{\L}$ for two given sequences $\M$ and $\L$ admitting a nonzero proximate order. For this purpose in the first subsection we will thoroughly examine the comparability notion for sequences introduced in Subsection~\ref{subsectLCsequences}. Secondly, we will analyze the properties of the quotient and the product sequences of $\M$ and $\L$. Finally, we will formulate our main results, generalizing the Gevrey case if
$\o(\M)<\o(\L)$, and showing that such a generalization is not possible (for our definition of summability) if the sequences are comparable and $\o(\M)=\o(\L)$.
The results in this section are stated for a couple of sequences but can be easily extended
for a finite set of sequences $\M_1,\M_2, \dots ,\M_k$.

\subsection{Comparison of sequences}\label{subsect.comp.sequence}

The example at the end of this subsection will show that, in general, two weight sequences $\M$ and $\L$ need not be comparable, not even  sequences whose sequences of quotients are regularly varying (which implies the admissibility of a nonzero proximate order, see Theorem~\ref{theorem.charact.prox.order.nonzero}). Hence, it will be natural to impose some comparability condition between $\M$ and $\L$ in the forthcoming subsections.\para

Since equivalent sequences define the same classes, we are particularly interested in comparable but not equivalent sequences, i.e., $\L\precsim \M$ and $\L\not\approx\M$, which is true if and only if
$$
\inf_{p\in\N} \left(\frac{M_p}{L_p}\right)^{1/p}>0 \qquad\text{and} \quad\sup_{p\in\N} \left(\frac{M_p}{L_p}\right)^{1/p}=\oo,
$$
 or, equivalently, if
\begin{equation}\label{eq.ComparableNotEquivSeq}
  \liminf_{p\in\N} \left(\frac{M_p}{L_p}\right)^{1/p}>0 \qquad\text{and} \quad\limsup_{p\in\N} \left(\frac{M_p}{L_p}\right)^{1/p}=\oo.
  \end{equation}
In other words, we want to avoid noncomparable sequences, that is, those for which
\begin{equation}\label{eq.noncomparability.cond.seq}
 \liminf_{p\to\oo} \left(\frac{M_p}{L_p}\right)^{1/p}=0 \qquad\text{and} \quad\ \limsup_{p\to\oo} \left(\frac{M_p}{L_p}\right)^{1/p}=\oo.
\end{equation}

For the construction of our example of noncomparable sequences, we need to characterize \eqref{eq.noncomparability.cond.seq} in terms of the corresponding associated functions (see~\eqref{equadefiMdet}).
Although it is valid for strongly regular sequences $\M$, the characterization of noncomparability in terms of the associated function is stated below in the most regular case, i.e., for a weight sequence $\M$ admitting a nonzero proximate order, as these are the ones used to develop the summability theory.

\begin{pro}\label{prop.noncomp.charct.M.L}
Let $\M$ and $\L$ be two weight sequences, and suppose that $\M$ admits a  nonzero proximate order.
We have that
  $$\liminf_{p\to\oo} \left(\frac{M_p}{L_p}\right)^{1/p}=0 \qquad\text{if and only if} \quad \liminf_{t\to\oo} \frac{\o_{\L}(t)}{\o_{\M}(t)}=0, $$
and
  $$\limsup_{p\to\oo} \left(\frac{M_p}{L_p}\right)^{1/p}=\oo \qquad\text{if and only if} \quad\ \limsup_{t\to\oo} \frac{\o_{\L}(t)}{\o_{\M}(t)}=\oo.$$
\end{pro}

\begin{proof}
If we suppose that $\liminf_{t\to\oo} \o_{\L}(t)/\o_{\M}(t)>0$,
then there exists $A>0$ such that $\o_{\L}(t)\geq A \o_{\M}(t)$ for every $t\geq t_0$. By \eqref{eq.WeightSeqFromOmegaM} we have that
$$\liminf_{p\to\oo} \left(\frac{M_p}{L_p}\right)^{1/p} =
 \liminf_{p\to\oo} \left(\frac{M_p}{\sup_{t>0} (t^p / e^{\o_{\L}(t)})} \right)^{1/p} \geq
 \liminf_{p\to\oo} \left(\frac{M_p}{\sup_{t>0} (t^p / e^{A\o_{\M}(t)})} \right)^{1/p}.  $$
By Lemma~\ref{lemma.M.scaling}, there exists $E>0$ such that $A\o_{\M}(t)>\o_{\L}(Et)$ for $t$ large enough. It is easy to check that the supremum for $t>0$ of the function $f_{p,\M}(t)=t^p e^{-A\o_{\M}(t)}$ is attained in $[m_{\lfloor p/A \rfloor},\oo)$, so for $p$ large enough we have that
$$\sup_{t>0} (t^p / e^{A\o_{\M}(t)}) <\sup_{t>0} (t^p / e^{\o_{\L}(Et)}),
$$
and  we deduce from~\eqref{eq.WeightSeqFromOmegaM} that
$$\liminf_{p\to\oo} \left(\frac{M_p}{L_p}\right)^{1/p} > E >0.  $$

 Now,  if
 we suppose that $\liminf_{p\to\oo} \left(M_p/L_p\right)^{1/p}>0$,
 then there exists $B>0$ such that $M_p\geq B^p L_p$ for every $p\in\N$. Consequently,  $\o_{\L}(t)\geq \o_{\M}(Bt)$ for every $t>m_0$ and, using Lemma~\ref{lemma.M.scaling}, we have that
 $$\liminf_{t\to\oo} \frac{\o_{\L}(t)}{\o_{\M}(t)}\geq \liminf_{t\to\oo} \frac{\o_{\M}(Bt)}{\o_{\M}(t)}\geq G>0.  $$
 The same arguments lead to the other equivalence.
\end{proof}

We will use a construction of sequences from proximate orders
in order to build noncomparable sequences.

\begin{defi}\label{defi.seq.U}
Let $\ro(t)$ be a nonzero proximate order,  $\ga>0$ and $V \in MF (\ga, \ro(t) )$. We define its \textit{associated sequence} $\V=(V_p)_{p\in\N_0}$ by $V_0:=1$ and $V_p:= v_0v_1\cdots v_{p-1}$ for all $p\in\N$, where
$$v_0:=V(1), \quad v_p:=V(p), \ \ p\in\N.$$
\end{defi}

By Theorem~\ref{Th.Maergoiz.analytic.PO}.(I), (III) and (VI), we see that $\V$ is a weight sequence and that $\bv=(v_p)_{p\in\N_0}$ is regularly varying of positive index $\ro:=\lim_{t\to\oo} \ro(t)$. In particular, by Theorem~\ref{theorem.charact.prox.order.nonzero}, $d_{\V}$ is a nonzero proximate order and, by Corollary~\ref{coroConseqRegularVariation}, $\V$ is strongly regular.

\begin{exam}\label{exam.noncomparability}
We consider the functions
$$\ro_1(t)=1, \quad t\in(0,\oo),$$
$$\ro_2(t)=1+\frac{\sin(\log_2(t))}{\log_2(t)}, \quad t\in (e,\oo), \quad \log_2(t):=\log(\log(t)),$$
which are easily checked to be
nonzero proximate orders which tend to 1 at infinity.
We define $V_1(t)=t^{\ro_1(t)}$ and $V_2(t)=t^{\ro_2(t)}$, and we observe that, for the sequences
$$r_n=\exp(\exp (\pi/2+2\pi n)),\qquad s_n=\exp(\exp (3\pi/2+2\pi n)), \qquad n\in\N,$$
one has
\begin{equation}\label{eq.V1V2noncomparable}
\lim_{n\to\oo}\frac{V_2(s_n)}{V_1(s_n)} =0,\qquad \lim_{n\to\oo}\frac{V_2(r_n)}{V_1(r_n)} =\oo .
\end{equation}

We fix $\ga>0$, $\widetilde{V}_1\in MF (\ga, \ro_1(t) )$ and $\widetilde{V}_2\in MF (\ga, \ro_2(t))$ and we consider the sequences $\U_1$ and $\U_2$ (see Definition~\ref{defi.seq.U}) defined from the corresponding inverse functions $\widetilde{U}_1(t)$ and $\widetilde{U}_2(t)$ (see Theorem~\ref{th.inv.analytic.PO}). These sequences $\U_1$ and $\U_2$ are regularly varying of index $1$, so $\o(\U_1)=\o(\U_2)=1$, and the functions $d_{\U_j}(t)$ are nonzero proximate orders.
According to Theorem~\ref{Th.Maergoiz.analytic.PO}.(VI) and by the information contained in~\cite[Prop.~4.13 and Th.~4.14]{JimenezSanzSchindl}, for $j=1,2$ we see that
$$0<\liminf_{t\to\oo} \frac{\o_{\U_j}(t)}{V_j(t)} \leq \limsup_{t\to\oo} \frac{\o_{\U_j}(t)}{V_j(t)}<\oo.$$
Taking into account~\eqref{eq.V1V2noncomparable}, we deduce that
 $$\liminf_{t\to\oo} \frac{\o_{\U_2}(t)}{\o_{\U_1}(t)}=0 \qquad\text{and} \quad\ \limsup_{p\to\oo}\frac{\o_{\U_2}(t)}{\o_{\U_1}(t)}=\oo.$$
Finally, by Proposition~\ref{prop.noncomp.charct.M.L}, we conclude that $\U_1$ and $\U_2$ are noncomparable.
\end{exam}


\subsection{Product and quotient of sequences}\label{subsect.product.quotient.seq}

In the study of multisummability in this general context there naturally appear the product sequence $\M \cdot \L:=(M_pL_p)_{p\in\N_0}$ and the quotient sequence $\M/\L:=(M_p/L_p)_{p\in\N_0}$ of two sequences $\M$ and $\L$.
In this subsection some elementary properties of these sequences will be obtained, and the connection with the comparability notion in the previous subsection will be established.
Since many of these properties are stated in terms of the sequence of quotients, note that the corresponding ones for $\M \cdot \L$ and $\M/\L$ are $\m\cdot\l=(m_p\ell_p)_{p\in\N_0}$ and
$\m/\l=(m_p/\ell_p)_{p\in\N_0}$, respectively.

\begin{pro}\label{prop.product.seq}
Suppose given two weight sequences $\M$ and $\L$, each one admitting a nonzero proximate order. Then,
$\M \cdot \L$ is a weight sequence, it admits a nonzero proximate order and $\o(\M\cdot\L)=\o(\M)+\o(\L)$.
\end{pro}

\begin{proof} This is immediate using Theorems~\ref{theorem.charact.prox.order.nonzero} and~\ref{theo.charact.admit.p.o} and the stability of (lc), regular variation and the index of regular variation for the product.
\end{proof}


We observe that the product sequence of two sequences also preserves some weaker properties. In particular, if $\M$ and $\L$ are strongly regular sequences then $\M \cdot \L$  is strongly regular, but only $\o(\M)+\o(\L)\leq \o(\M \cdot \L)$ can be guaranteed. An example in which equality fails to hold can be constructed with the techniques described in~\cite[Remark~4.13]{JimenezSanzSchindlIndex}.

If there exists $a>0$ such that the sequences of quotients $\m=(m_p)_{p\in\N_0}$ and $\l=(l_p)_{p\in\N_0}$, respectively associated with $\M=(M_p)_{p\in\N_0}$ and $\L=(L_p)_{p\in\N_0}$, satisfy
$$a^{-1}\ell_p\leq m_p \leq a \ell_p, \qquad p\in\N_0,$$
we write $\m\simeq\l$, and then it is clear that $\M\approx\L$. Of course, if $\m$ and $\l$ are equivalent in the classical sense, that is, $\lim_{p\to\oo} m_p/\ell_p=1$ (denoted $\m\sim\l$), then $\m\simeq\l$.

The main difficulty when dealing with the quotient of two weight sequences is to ensure that it satisfies (lc).
If the sequences considered are regular enough, we will solve this problem by switching $\M/\L$ for an equivalent sequence, thanks to the following result of R.~Bojanic and E.~Seneta.

\begin{theo}[\cite{BojanicSeneta}, \, Th. 4]\label{theo.Boj.Sen.normalized.seq}
Let $\o\in\R$ be given. A sequence of positive real numbers $(c_p)_{p\in\N}$ is regularly varying with index $\o$
if and only if there exists a sequence of positive real numbers $(b_p)_{p\in\N}$  such that $b_p\sim c_p$ and \begin{equation}\label{eq.boj.sen.smooth.seq}
	\frac{b_{p+1}}{b_{p}} = 1+\frac{\o}{p}+o\left(\frac{1}{p}\right)\qquad \text{as $p\to\oo$}.
	\end{equation}
In particular, $(b_p)_{p\in\N}$ is also regularly varying with index $\o$.
\end{theo}

\begin{pro}\label{prop.quotient.seq}
Given two weight sequences $\M$ and $\L$, each one admitting a nonzero proximate order, assume that
$\o(\L)<\o(\M)$.  Then there exists a weight sequence $\A$ equivalent to $\M/\L$ whose sequence of quotients is regularly varying with index $\o(\M)-\o(\L)$.
\end{pro}

\begin{proof} By Theorems~\ref{theorem.charact.prox.order.nonzero} and~\ref{theo.charact.admit.p.o},  we know that there exist weight sequences
$\L'$ and $\M'$ equivalent to $\L$ and $\M$, respectively, whose sequences of quotients $(\ell'_p)_{p\in\N_0}$ and $(m'_p)_{p\in\N_0}$
 are regularly varying with indices $\o(\L)$ and $\o(\M)$.
Applying Theorem~\ref{theo.Boj.Sen.normalized.seq} there exist sequences of positive real numbers $(\ell''_p)_{p\in\N_0}$ and $(m''_p)_{p\in\N_0}$ with
$\lim_{p\to\oo} \ell''_p/\ell'_p=1$ and $\lim_{p\to\oo} m''_p/m'_p=1$,
and satisfying~\eqref{eq.boj.sen.smooth.seq}. It is plain to check that the sequence $\bb=(b_p:=m''_p/\ell''_p)_{p\in\N_0}$ is regularly varying of index $\o(\M)-\o(\L)$. By~\eqref{eq.boj.sen.smooth.seq}, we observe that
$$\frac{b_{p+1}}{b_{p}}=\frac{m''_{p+1}}{m''_{p}}\frac{\ell''_{p}}{\ell''_{p+1}}= \left( 1+\frac{\o(\M)}{p}+o\left(\frac{1}{p}\right)\right)\frac{1}{1+\o(\L)/p+o(1/p) }  .$$
We take $\a\in(\o(\L),\o(\M))$, then there exists $p_0\in\N$ such that
$$\frac{b_{p+1}}{b_{p}}>\left( 1+\frac{\a}{p}\right)\frac{1}{(1+\a/p)} =1, \qquad p\geq p_0.$$
We define the sequence $a_p:=b_{p_0}$ for $p<{p_0}$ and $a_p:=b_p$ for $p\geq p_0$.
The sequence $\ba=(a_p)_{p\in\N_0}$ is nondecreasing and regularly varying of index $\o(\M)-\o(\L)$ and $\ba\simeq\bb$. Consequently, the corresponding sequence $\mathbb{A}$ is a weight sequence with
$\A\approx \mathbb{B} = \M''/\L''$.
Since $\m'' / \l'' \sim \m'/\l'$, we have that $\m''/\l''\simeq \m'/\l'$, so also $\mathbb{A}\approx \M'/\L'$. Finally, using that $\L'$ and $\M'$ are respectively equivalent to $\L$ and $\M$, we conclude that the proposition holds.
\end{proof}

\begin{rema}\label{rema.MoverL.welldefined}
If $\o(\L)<\o(\M)$, in our regards we can always change, using the above proposition, $\M$ for the equivalent sequence $\A \cdot \L$ which is (lc) and admits a nonzero proximate order. So, without loss of generality we  can always assume that $\M/\L$ is (lc) and that its sequence of quotients is regularly varying of positive index.
\end{rema}

Some information about the behavior of these sequences can be obtained even if the conditions on $\M$ and $\L$ are relaxed. In particular, we observe that the indices $\o(\M\cdot \L)$ and $\o(\M/\L)$ can be computed from $\o(\M)$ and $\o(\L)$.

\begin{rema}\label{rema.omega.index.quotient.sequence}
Assume that $\M$ and $\L$ are weight sequences, such that $\L$ satisfies \eqref{limit.logmp.over.logp.omega}, which is guaranteed in case $\L$ admits a nonzero proximate order (see Corollary~\ref{coroConseqRegularVariation}), then
  $$\o(\M\cdot\L)= \liminf_{p\to\oo} \frac{\log(m_p\ell_p)}{\log(p)} =
 \liminf_{p\to\oo} \frac{\log(m_p)}{\log(p)} +\lim_{p\to\oo}\frac{\log(\ell_p)}{\log(p)} =\o(\M)+\o(\L)\in[0,\oo],$$
 $$\o(\M/\L)= \liminf_{p\to\oo} \frac{\log(m_p/\ell_p)}{\log(p)} =
 \liminf_{p\to\oo} \frac{\log(m_p)}{\log(p)} -\lim_{p\to\oo}\frac{\log(\ell_p)}{\log(p)} =\o(\M)-\o(\L),$$
 whenever the last difference is not indeterminate.
\end{rema}

 Finally, the following proposition shows that $\L\precsim\M$ and $\L\not\approx\M$ (comparability but not equivalence) can be characterized in terms of the quotient sequence.

\begin{pro}\label{prop.quotient.sequence.compara.cond}
 Given two sequences $\M$ and $\L$, we have that:
 \begin{enumerate}[(i)]
 \item If the sequence of quotients of $\M/\L$  tends to infinity, then $\lim_{p\to\oo}(M_p/L_p)^{1/p}=\oo$ and $\L\precsim\M$ and $\L\not\approx\M$.
 \item Assume that $\M/\L$ is (lc). Then
  $\L\precsim\M$ and $\L\not\approx\M$ if and only if
 $\lim_{p\to\oo}(M_p/L_p)^{1/p}=\oo$
  if and only if the sequence of quotients of $\M/\L$ tends to infinity.
 \item Assume that $\L$ is a weight sequence satisfying~\eqref{limit.logmp.over.logp.omega},
and that $\o(\L)<\o(\M)$.
Then  the sequence of quotients of $\M/\L$ tends to infinity,  $\L\precsim\M$, $\L\not\approx\M$ and $\lim_{p\to\oo}(M_p/L_p)^{1/p}=\oo$.
 \end{enumerate}

\end{pro}

\begin{proof}
 \begin{enumerate}
 \item[(i)]  If $\m/\l$ tends to infinity, then for any $K>0$ it exists $p_0\in\N$ such that
 $m_p/\ell_p\geq K$ for every $p\geq p_0$. Hence, we deduce that
 $$ \left(\frac{M_p}{L_p}\right)^{1/p} \geq \left(\frac{K^{p-p_0} M_{p_0}}{L_{p_0}}\right)^{1/p},  \qquad p\geq p_0.$$
 Taking limit inferior in both sides we see that $\liminf_{p\to\oo} (M_p/L_p)^{1/p}\geq K$. Since this is true for any $K>0$ we conclude that $\lim_{p\to\oo}(M_p/L_p)^{1/p}=\oo$, so~\eqref{eq.ComparableNotEquivSeq} holds, which implies that $\L\precsim\M$ and $\L\not\approx\M$.

\item[(ii.a)] The `if' part of the first equivalence follows from~\eqref{eq.ComparableNotEquivSeq}. 
Conversely, if $\M/\L$ is (lc) then $((M_p/L_p)^{1/p})_{p\in\N}$ is nondecreasing; from $\L\precsim\M$ and $\L\not\approx\M$ we deduce that $((M_p/L_p)^{1/p})_{p\in\N}$ tends to $\oo$.

For (lc) sequences, the second equivalence was indicated in Subsection~\ref{subsectLCsequences}.

\item[(iii)] We have that
$$\liminf_{p\to\oo} \frac{\log(m_p)}{\log(p)}=\o(\M),\qquad \lim_{p\to\oo} \frac{\log(\ell_p)}{\log(p)}=\o(\L). $$
Since $\o(\L)<\o(\M)$ we can fix $0<\ep<(\o(\M)-\o(\L))/2$, and there exists $p_0\in\N$ such that for every $p\geq p_0$ we get that
$$\frac{m_p}{\ell_p}\geq \frac{p^{\o(\M)-\ep}}{p^{\o(\L)+\ep}} = p^{\o(\M)-\o(\L)-2\ep},$$
then $\lim_{p\to\oo}m_p/\ell_p=\oo$. We conclude by applying (i).
  \end{enumerate}

\end{proof}


In our framework we will usually assume that $\L$ admits a nonzero proximate order, so it satisfies~\eqref{limit.logmp.over.logp.omega}; according to (iii), any sequence $\M$ with $\o(\L)<\o(\M)$ will be comparable with and not equivalent to $\L$.
Consequently, comparability assumptions need only be made when  $\o(\M)=\o(\L)$ (and they \emph{must} be made, as shown by Example~\ref{exam.noncomparability}); in this case, according to the last result, it is natural to assume that $\M/\L$ is a weight sequence, and so
$\L\precsim\M$ and $\L\not\approx\M$.

\subsection{Tauberian theorem for sequences with different indices}

Assuming a basic comparability hypothesis, justified in the previous subsections, and that one of the sequences is regular enough in order to  employ summability techniques,  we are ready to clarify the relation between $\C\{z\}_{\M}$ and $\C\{z\}_{\L}$. First,  we observe that if a formal power series is summable for a sequence $\M$ and its coefficients are bounded in terms of a smaller sequence $\L$, then it is also summable for $\L$ and its sums agree, extending what was proved for $k-$summability by J. Martinet and J.-P. Ramis \cite[Ch. 2, Prop. \ 4.3]{RamisMartinet90}     (see also \cite[Th. 37]{Balser2000} and \cite[Coro. 5.3.15]{Loday16}).

\begin{pro}\label{prop.IntersectQuasianalClasses}
Let $\L$ and $\M$ be weight sequences  such that  $\L$ admits
 a nonzero proximate order  and $\M/\L$ is (lc). If $\o(\L)<\o(\M)$, or if $\o(\L)=\o(\M)$ assuming in addition that $\L\precsim\M$ and $\L\not\approx\M$, then for every $\widehat{f}\in\C\{z\}_{\M}\cap\C[[z]]_{\L}$, we have that
\begin{enumerate}[(i)]
 \item $\hbox{sing}_{\L} (\widehat{f})\en \hbox{sing}_{\M}(\widehat{f})$ and $\widehat{f}\in\C\{z\}_{\M}\cap\C\{z\}_{\L}$.
\item For every $d\not\in\hbox{sing}_{\M}(\widehat{f})$, $(\mathcal{S}_{\L,d}\widehat{f})(z)\sim_{\L} \widehat{f}$ on $G(d,\a)$ with $\a>\o(\M)$,
 \item  $(\mathcal{S}_{\L,d}\widehat{f})(z)=(\mathcal{S}_{\M,d}\widehat{f})(z)$
for every $z$ where both functions are defined.

\end{enumerate}
\end{pro}

\begin{proof}
 By Proposition~\ref{prop.quotient.sequence.compara.cond}, from the hypotheses in both situations ($\o(\L)<\o(\M)$ or $\o(\L)=\o(\M)$) we deduce that $\M/\L$ is a weight sequence, then Watson's Lemma is available.
 \begin{enumerate}[(i)]

 \item Since $\widehat{f}\in\C\{z\}_{\M}$, we have that $\hbox{sing}_{\M}(\widehat{f})$ is finite. Let $d$ be a
 nonsingular direction of $\M-$summability. We write $f(z):=\mathcal{S}_{\M,d}\widehat{f}(z)$. We choose a kernel of
 $\L-$summability, that exists by Remark~\ref{rema.def.exist.kernels}, and we consider $g:=T^{-}_{\L} f $. Since $f$ is defined on a sectorial region $G(d,\a)$ with
 $\a>\o(\M)\geq\o(\L)$, by Proposition~\ref{prop.hol.Bor}, $g$ is holomorphic in a
 sector $S(d,\a-\o(\L))$ with $\L-$growth on this sector. \par
 On the other hand, $\widehat{f}\in\C[[z]]_{\L}$, hence we have that $\widehat{g}:=\widehat{T}^{-}_{\L}\widehat{f}\in\C\{z\}$.
 By Theorem~\ref{teorrelacdesartransfBL}.(ii), we have that $g\sim_{\M/\L} \widehat{g} $ on $S(d,\a-\o(\L))$ and,
 since $\widehat{g}\in\C\{z\}$, $\mathcal{S} \widehat{g}\sim_{\M/\L}\widehat{g} $ in a disc, and then in a sectorial region of
 opening $\pi(\a-\o(\L))$.
 By the Watson's Lemma (Theorem~\ref{TheoWatsonlemma}) we have that
 $\widetilde{\mathcal{B}}:\widetilde{\mathcal{A}}_{\M/\L}(G_\ga)\longrightarrow \C[[z]]_{\M/\L}$
 is injective for every $\gamma>\o(\M/\L)=\o(\M)-\o(\L)$ (see Remark~\ref{rema.omega.index.quotient.sequence}).
 Since $g$ is holomorphic in a sector of opening
 $\pi(\a-\o(\L))>\pi(\o(\M)-\o(\L))$,  $g$ is unique and, consequently, it is an analytic extension
 of $\widehat{g}:=\widehat{T}^{-}_{\L}\widehat{f}$ with $\L-$growth in the sector $S(d,\a-\o(\L))$. Using
 Theorem~\ref{th.Msummable.equiv.esummable},  we see that
 $\widehat{f}$ is $\L-$summable in direction $d$. We conclude that
 $\hbox{sing}_{\L} (\widehat{f})\en \hbox{sing}_{\M}(\widehat{f})$, then $\widehat{f}\in \C\{z\}_{\L}$.
  \item From (i), we know that $\widehat{g}=\widehat{T}_\L^{-}\widehat{f}$ converges in a disc and admits analytic
  continuation $g$ in a sector $S=S(d,\a-\o(\L))$, $g\in\mathcal{O}^{\L}(S)$ and we have that
$\mathcal{S} \widehat{g} \sim_{\M'}\widehat{T}_\L^{-}\widehat{f}$ in $S$ with $\M'=(1)_{n\in\N_0}$. Then, in Theorem~\ref{teorrelacdesartransfBL}.(i),
we obtain that the function $f:=\mathcal{S}_{\L,d}\widehat{f}=T_\L g$ is holomorphic in a sectorial region
$G(d,\a)$ and $f\sim_{\L}\widehat{f}$ there.

\item With the notation in (i), we have that, for every $z$ where these functions are defined,
$$(\mathcal{S}_{\L,d}\widehat{f})(z)= (T_{\L} g )(z)=(T_{\L} T^{-}_{\L} f)(z) = f(z)= (\mathcal{S}_{\M,d}\widehat{f})(z).$$
\end{enumerate}
\end{proof}

As a consequence of the last proposition, a generalization of the classical Tauberian result for $k-$summability of J.-P. Ramis (see \cite[Th. 3.8.1]{RamisSibuya89}, \cite[Th. 37]{Balser2000} and~\cite[Coro. 5.3.16]{Loday16}) can be stated when the indices $\o(\L)$ and $\o(\M)$ are distinct.

\begin{theo}\label{th.tauberian.distinct.index}
Let $\L$ and $\M$ be weight sequences  such that  $\L$ admits
 a nonzero proximate order, $\M/\L$ is (lc) and  $\o(\L)<\o(\M)$.
  If  $\widehat{f}\in\C\{z\}_{\M}\cap \C[[z]]_{\L}$, then $\widehat{f}$ is convergent.
\end{theo}
\begin{proof}
  Using Proposition~\ref{prop.IntersectQuasianalClasses}, we know that $\widehat{f}\in\C\{z\}_{\M}\cap\C\{z\}_{\L}$ and
 $$(\mathcal{S}_{\L,d}\widehat{f})(z)=(\mathcal{S}_{\M,d}\widehat{f})(z),$$
for every $z$ where both functions are defined. Given
$\theta \in \hbox{sing}_{\M}(\widehat{f})=\{\theta_1,\theta_2,\dots,\theta_m\}$, we can take
$d\not\in\hbox{sing}_{\M}(\widehat{f})$ such that $|d-\theta|<\delta:=\pi(\o(\M)-\o(\L)) $. Then
$(\mathcal{S}_{\M,d}\widehat{f})(z)=(\mathcal{S}_{\L,d}\widehat{f})(z)$ is defined in a sectorial region $G$ of opening $\pi\a>\pi \o (\M)$
bisected by direction $d$. We observe that $f=(\mathcal{S}_{\L,d}\widehat{f})(z)$ is a holomorphic function defined in a sectorial
region $\widetilde{G}$ contained in $G$, bisected by direction $\theta$, and of opening
$$\a\pi-|d-\theta|>\a\pi-\delta=\a\pi-\pi\o(\M)+\pi\o(\L)>\pi\o(\L). $$
By Proposition~\ref{prop.IntersectQuasianalClasses}.(ii), we have that $f\sim_{\L}\widehat{f}$ in this region, then
$\widehat{f}\in \C\{z\}_{\L,\theta}$. Since $\hbox{sing}_{\L} (\widehat{f})\en \hbox{sing}_{\M}(\widehat{f})$, we deduce that
$\hbox{sing}_{\L} (\widehat{f}) =\emptyset$ and, by Proposition~\ref{prop.summability.algebras}, we conclude that $\widehat{f}\in\C\{z\}$.
\end{proof}

Regarding the last two results, in the case $\o(\L)<\o(\M)$, if $\M$ also admits a nonzero proximate order,  the logarithmic convexity
of the sequence $\M/\L$ does not need to be assumed since it is automatically guaranteed (see Remark~\ref{rema.MoverL.welldefined}).

Finally, we will show that this theorem is not valid when the indices coincide.

\begin{theo}
Let $\L$ and $\M$ be weight sequences with $\L\precsim\M$ and $\o(\L)=\o(\M)$. Then
 $$\C\{z\}_{\M}\cap \C\{z\}_{\L}=\C\{z\}_{\L}.$$
 Moreover, if $\L$ admits a nonzero proximate order,  $\L\not\approx\M$ and $\M/\L$ is (lc) we have that
 $$\C\{z\}_{\M}\cap \C[[z]]_{\L}=\C\{z\}_{\M}\cap \C\{z\}_{\L}=\C\{z\}_{\L}.$$
\end{theo}

\begin{proof}
If $\widehat{f}$ is $\L-$summable in direction $d$, then it exists $\a>\o(\L)=\o(\M)$ and $f$ holomorphic in $G=G(d,\a)$ such that $f\sim_\L\widehat{f}$ in $G$. Since $\L\precsim\M$, we deduce that $f\sim_\M\widehat{f}$ in $G$, and we conclude that
$\widehat{f}$ is $\M-$summable in direction $d$. Then,
$  \C\{z\}_{\L,d}\en\C\{z\}_{\M,d}$ and, consequently, $  \C\{z\}_{\L}\en\C\{z\}_{\M}$.
The last statement is obtained immediately using Proposition~\ref{prop.IntersectQuasianalClasses}.
\end{proof}

\begin{exam}
 With the notation and computations in Example~\ref{exampleSequences}, we deduce that for any $\a>\a'>0$ and any $\b,\b'\in\R$ we have  $\C\{z\}_{\M_{\a,\b}}\cap \C[[z]]_{\M_{\a',\b'}}=\C\{z\}$, while $\C\{z\}_{\M_{\a,\b}}\cap \C[[z]]_{\M_{\a,\b'}}= \C\{z\}_{\M_{\a,\min(\b,\b')}}$.
\end{exam}

%

\section{Multisummability}
\label{sect.multisummab.analyt}

Whenever the Tauberian Theorem~\ref{th.tauberian.distinct.index} is available it makes sense to give a
definition of multisummability in this context. In the recent book of M. Loday-Richaud~\cite[Ch.\ 7]{Loday16}, several equivalent definitions of multisummability are provided together with a careful study of their peculiarities. In this paper, since the $\M-$summability tools are defined using moment summability methods, the approach of W. Balser~\cite[Ch.\ 10]{Balser2000} has been chosen, that is, the decomposition into sums. Due to a ramification inconvenience, this splitting definition is only compatible with the others if the corresponding indices $\o(\M_j)$ are all smaller than $2$.

\begin{defi}\label{def.multisum}
Let $\M_1$ and $\M_2$ be weight sequences admitting a nonzero proximate order such that $\o(\M_1)<\o(\M_2)<2$, and $d_1,d_2\in\R$ such that $|d_1-d_2|<\pi(\o(\M_2)-\o(\M_1))/2$. A formal power series $\widehat{f}(z) = \sum_{p \geq 0} a_{p} z^{p} \in\C[[z]]$ is said to be
\textit{$(\mathbb{M}_1,\mathbb{M}_2)-$summable in the multidirection $(d_1,d_2)$},  if there exist a formal series
$\widehat{f}_{1}(z)$ which is $\mathbb{M}_1-$summable in $d_1$ with $\mathbb{M}_1-$sum $f_{1}$ and a formal series $\widehat{f}_{2}(z)$ which is $\mathbb{M}_2-$summable in $d_2$ with
$\mathbb{M}_2-$sum $f_{2}$ such that
$$\widehat{f} = \widehat{f}_{1} + \widehat{f}_{2}.$$ \para

Furthermore, the holomorphic function
$f(z) = f_{1}(z) + f_{2}(z)$ defined on $G(d_1,\a_1)$ for some $\a_1 >\o(\M_1)$  is called the
\textit{$(\mathbb{M}_1,\mathbb{M}_2)-$sum of $\widehat{f}$ in the multidirection $(d_1,d_2)$} and we write $f(z)=(\mathcal{S}_{(\M_1,\M_2),(d_1,d_2)}\widehat{f})(z)$ and $\widehat{f}\in\C\{z\}_{(\M_1,\M_2),(d_1,d_2)}$.
\end{defi}

In the conditions of the previous definition, there always exists a sectorial region $G=G(d_1,\a_1)$ with $\a_1>\o(\M_1)$ such that
\begin{equation}\label{eq.non.unique.multisum}
 f\sim_{\M_2} \widehat{f} \qquad \text{on} \quad G.
\end{equation}
Since the region is not wide enough, $f$ is not determined by condition~\eqref{eq.non.unique.multisum}, weaker than multisummability. Nevertheless, the next proposition shows that the multisum is unique and the splittings are essentially unique.

\begin{pro}
 In the conditions of Definition~\ref{def.multisum}, assume that there exist two pairs of formal power series $\widehat{f}_1,\widehat{f}_2$ and $\widehat{g}_{1},\widehat{g}_{2}$ such that
 $$\widehat{f}=\widehat{f}_{1} + \widehat{f}_{2} =\widehat{g}_{1} + \widehat{g}_{2}.$$
 Then there exist $\a_2>\o(\M_2)$ and $\widehat{u}_1\in\C[[z]]$ such that $\widehat{u}_1$ is $\M_1-$summable on a sectorial region $G(d_2,\a_2)$ and
 $$\widehat{g}_1=\widehat{f}_1-\widehat{u}_1,\qquad\text{and}\qquad \widehat{g}_2=\widehat{u}_1+\widehat{f}_2.$$
 Moreover, we have that the $(\mathbb{M}_1,\mathbb{M}_2)-$sum of $\widehat{f}$ is unique, that is,
 $$f_1(z)+f_2(z)=f(z)=g_1(z)+g_2(z),$$
in a sectorial region $G(d_1,\a_1)$ with $\a_1>\o(\M_1)$.
\end{pro}

\begin{proof}
 We define $\widehat{u}_1:=\widehat{f}_{1}-\widehat{g}_{1}$, so $\widehat{u}_1\in\C\{z\}_{\M_1,d_1}$, in particular
 $\widehat{u}_1\in\C[[z]]_{\M_1}$. We observe that
 $\widehat{g}_{2}-\widehat{f}_{2}=\widehat{u}_1$, then $\widehat{u}_1\in\C\{z\}_{\M_2,d_2}$. By Proposition~\ref{prop.IntersectQuasianalClasses}.(ii), there exists $\a_2>\o(\M_2)$ such that $\widehat{u}_1$ is $\M_1-$summable in $G(d_2,\a_2)$. Furthermore, by Lemma~\ref{basic.properties.Msumm}.(ii), $(\mathcal{S}_{\M_1,d_1}\widehat{u}_1)(z)= (\mathcal{S}_{\M_1,d_2}\widehat{u}_1)(z)$ on a sectorial region $G=G(d_1,\a_1)$ with $\a_1>\o(\M_1)$ and, using Proposition~\ref{prop.IntersectQuasianalClasses}.(iii), for all $z\in G$ we conclude that
 \begin{align*}
  f_1(z)-g_1(z)&=(\mathcal{S}_{\M_1,d_1}(\widehat{f}_1-\widehat{g}_1))(z)=(\mathcal{S}_{\M_1,d_2}\widehat{u}_1)(z) \\
  &=(\mathcal{S}_{\M_2,d_2}\widehat{u}_1)(z)=(\mathcal{S}_{\M_2,d_2}(\widehat{g}_2-\widehat{f}_2))(z)= g_2(z)-f_2(z). \qquad
 \end{align*}
\end{proof}

This definition can be recursively extended for a finite set of sequences $\M_1,\M_2,\dots,\M_k$ with $\o(\M_1)<\o(\M_2)<\cdots<\o(\M_k)<2$ (see \cite[Ch.\ 10]{Balser2000} and \cite[Ch.\ 7]{Loday16}).

The rest of this section is devoted to recover the multisum by means of some suitable integral transforms.

\subsection{Moment-kernel duality}
\label{subsect.momentkernelduality}

The main aim of this subsection is to prove that a kernel $e$ of $\M-$summability is uniquely determined by its
sequence of moments $\mathfrak{m}_e$, similarly to the result of W.~Balser for the moment summability methods~\cite[Sect.\ 5.8]{Balser2000}. \para

For bounded functions on sectors, the following auxiliary lemma shows that the  domain of holomorphy of their $e-$Laplace transform is not an arbitrary sectorial region, as indicated in Proposition~\ref{prop.hol.lap}, but an unbounded sector.

\begin{lemma}\label{lemma.holomorphy.e.laplace.transform.bounded.functions}
 Let $e(z)$ be a kernel of $\M-$summability with corresponding Laplace operator $T=T_{e}$. Given a function $f$ defined in a sector $S=S(d,\a)$,
 assume that for every  $0<\b<\a$ there exists a constant $C_\b>0$ such that
 $$|f(u)|\leq C_\b,\qquad u\in S(d,\b).$$
Then $g = T f$ is holomorphic in $S(d,\a+\o(\M))$.
\end{lemma}

\begin{proof}
 We have that $f\in\mathcal{O}^{\bM}(S)$ with $S=S(d,\a)$.
Let $\tau\in\R$ be a direction in $S$, i.e., such that $|\tau-d|<\pi\a/2$.
For every $u,z\in\mathcal{R}$ with $\mathrm{arg}(u)=\tau$ and $|\tau-\mathrm{arg}(z)|<\o(\M)\pi/2$ we have that
$u/z\in S_{\o(\M)}$, so the expression under the integral sign in (\ref{equadefitranLapl}) makes sense.
We fix $a>0$, and write
$$
g(z)=\int_0^{\infty(\tau)}e(u/z)f(u)\frac{ du}{u}=
\int_0^{ae^{i\tau}}e(u/z)f(u)\frac{du}{u}+\int_{ae^{i\tau}}^{\infty(\tau)}e(u/z)f(u)\frac{du}{u}.
$$
Since $f$ is bounded at  the origin following direction $\tau$ and by Definition~\ref{defikernelMsumm}.(\textsc{ii}), it is straightforward
to apply Leibniz's rule for parametric integrals and deduce that
the first integral in the right-hand side defines a holomorphic function in $S(\tau,\o(\M))$. Regarding the second
integral, we take $\b<\a$ such that $|\tau-d|<\b\pi/2$ and we fix $0<\ga<\o(\M)$. We have that
$u/z\in S_{\ga}$, for $\arg(u)=\tau$ and $z$ such that
$|\tau-\mathrm{arg}(z)|<\ga\pi/2$. The property in Definition \ref{defikernelMsumm}.(\textsc{iii}) provides us with
constants $c,k>0$ such that
$$
|e(u/z)|\le c h_{\M}(k|z|/|u|), \qquad \arg(u)=\tau, \quad z\in S(\tau,\ga),
$$
then
$$
\big|\frac{1}{u}e(u/z)f(u)\big|\le \frac{c C_\b }{|u|} h_{\bM}(k|z|/|u|), \qquad \arg(u)=\tau, \quad z\in S(\tau,\ga).
$$
For any $z_0\in S(\tau,\ga)$ we fix a bounded neighborhood $U$ of $z_0$ contained in $S(\tau,\ga)$.
We have that $|z|<r$ for every $z\in U$, and from the monotonicity of $h_{\M}$ we deduce that
\begin{equation*}
\big|\frac{1}{u}e(u/z)f(u)\big|\le \frac{c C_\b}{|u|} h_{\bM}(kr/|u|).
\end{equation*}
By the definition of $h_{\M}$, we have that $h_{\M}(kr/|u|)\le M_1kr/|u|$, so the right-hand
side of the last inequality is an integrable function of $|u|$ in $(a,\infty)$, and again Leibniz's
rule allows us to conclude the desired analyticity for the second integral.\par
Consequently, $g$ is holomorphic in $S(\tau,\ga)$ for every $|\tau-d|<\pi\a$ and every $0<\ga<\o(\M)$.
By analytic continuation on rotating $\tau$,
$g$ is holomorphic in $S(d,\a+\o(\M))$.
\end{proof}

\begin{rema}\label{rema.asymp.laplace.transform.bounded.functions}
Moreover,  if $f(z)\sim_{\M'} \sum_{n=0}^\oo a_n z^n$, by Theorem~\ref{teorrelacdesartransfBL}.(i), we deduce that $g=T_{e} f\sim_{\M\M'} \sum_{n=0}^\oo a_n m_{e}(n) z^n$ on a sectorial region $G(d,\a+\o(\M))$. Since the notion of asymptotic expansion only depends on the behavior of the function on bounded subsectors, we can say that $g\sim_{\M\M'} \sum_{n=0}^\oo a_n m_{e}(n) z^n$ on $S(d,\a+\o(\M))$,
whenever $g$ is holomorphic in $S(d,\a+\o(\M))$.
\end{rema}

As it happens in the Gevrey case, since the moment function $m_{e}(\lambda)$ is the Mellin transform of $e(z)$ (see \cite[Sect. 1.29]{Titchmarsh}), there is a duality between $m_{e}(\lambda)$ and $e$ and  the next lemma shows how one can recover $e(z)$ from its moment sequence $\mf_{e}$, thanks to the inversion formula. However, observe that, as it was mentioned in~\cite{Balser2000}, we shall not be concerned with the harder question of how to characterize such $\mf$ to which a kernel $e(z)$ exists. The following lemma generalizes Lemma 7 in~\cite{Balser2000}.

\begin{lemma}\label{lemma.e.Laplace.geom.series}
 Let a kernel function $e(z)$ of $\M-$summability with corresponding operator $T=T_{e}$ be given. For $f (u):= (1-u)^{-1}$, we have that $g = T f$ is holomorphic in $S(\pi,2+\o(\M))$, is $\M-$asymptotic to
$ \widehat{g}(z)=\sum^\oo_{0}m(n)z^n$ there, and $g(z)\to0$ as $|z|\to\oo$ uniformly for $z\in  S(\pi,2+\ga)$ for every $\ga<\o(\M)$. Moreover,
\begin{equation}\label{eq.uniq.e}
 g(z)- g(ze^{2\pi i} ) = 2\pi i e(1/z),\qquad z\in S_{\o(\M)}.
\end{equation}
\end{lemma}

\begin{proof}
The function $f(u)$ is defined in the sector $S(\pi,2)$ and continuous at the origin. For every $1<\b<2$, we have that
$$|f(u)|\leq \frac{1}{|1-u|} \leq \frac{1}{\sin(\pi-\b\pi/2)},\qquad u\in S(\pi,\b).$$
Hence, by Lemma~\ref{lemma.holomorphy.e.laplace.transform.bounded.functions},
we have that $g$ is holomorphic in $S(\pi,2+\o(\M))$.  Since $f$ is convergent at $0$, we have that $f(z)\sim_{\M'} \sum_{n=0}^\oo z^n$
with $\M'=(1)_{n\in\N_0}$ and, by Remark~\ref{rema.asymp.laplace.transform.bounded.functions}, we deduce that $g\sim_{\M} \sum_{n=0}^\oo m(n) z^n$ on $S(\pi,2+\o(\M))$.

The behavior at infinity can be again read off from the integral representation as follows. We fix a direction $\tau\in(0,2\pi)$, $\tau\not=\pi$, and we consider a direction $\theta \in (-\pi\o(\M)/2, 2\pi+\pi\o(\M)/2)$ such that
$|\tau-\theta|<\pi\ga/2<\pi\o(\M)/2$. For every $z\in S(\tau,\ga)$ with $\arg(z)=\theta$, we have that
$$g(z)=\int_{0}^{\oo} e\left(\frac{r}{|z|} e^{i(\tau-\theta)}\right) \frac{dr}{r(1-re^{i\tau})}=
\int_{0}^{\oo}e(se^{i(\tau-\theta)}) \frac{ds}{s(1-s|z|e^{i\tau})}.$$
We split the integral into two parts and we see that
$$\left|\int_{0}^{1/|z|^{1/2}} e(se^{i(\tau-\theta)}) \frac{ds}{s(1-s|z|e^{i\tau})}\right|\leq
\frac{1}{\inf_{0<s<\oo}\{|1-s|z|e^{i\tau}|\}} \int_{0}^{1/|z|^{1/2}} \frac{|e(se^{i(\tau-\theta)})|}{s} ds .$$
If $\tau\in (0,\pi/2) \cup (3\pi/2,2\pi)$, we have that
 $$\inf_{0<s<\oo}\{|1-s|z|e^{i\tau}|\}=|\sin(\tau)|\not=0,$$
and if $\tau\in [\pi/2,3\pi/2]$ ($\tau\not=\pi$),  we observe that
$$\inf_{0<s<\oo}\{|1-s|z|e^{i\tau}|\}= 1 \geq|\sin(\tau)|\not=0.$$
Consequently, we have shown that
\begin{equation}\label{ineq.limit.g.infinity.1}
 \left|\int_{0}^{1/|z|^{1/2}} e(se^{i(\tau-\theta)}) \frac{ds}{s(1-s|z|e^{i\tau})}\right|\leq
\frac{1}{|\sin(\tau)|} \int_{0}^{1/|z|^{1/2}} \sup_{|\phi|<\pi\ga/2}|e(se^{i\phi})| \frac{ds}{s} .
\end{equation}
Using Definition~\ref{defikernelMsumm}.(\textsc{ii}), we see that the integral in the right hand side is convergent. Subsequently, it tends to $0$, uniformly
as $|z|$  goes to infinity in $S(\tau,\ga)$.\para

On the other hand, since the function $e$ is uniformly bounded
in $S_{\ga}$ for every $|\tau-\theta|<\pi\ga/2$, there exists $c>0$ such that
$$\left|\int_{1/|z|^{1/2}}^{\oo} e(se^{i(\tau-\theta)}) \frac{ds}{s(1-s|z|e^{i\tau})}\right|\leq
c \int_{1/|z|^{1/2}}^{\oo} \frac{ ds}{s|1-s|z|e^{i\tau}|} .$$
We have that $|1-s|z|e^{i\tau}|\geq|s|z|-1|=|s|z|^{1/2}|z|^{1/2}-1|$. Since always $s|z|^{1/2}\geq 1$, if $|z|^{1/2}\geq 1$ we see that  $|1-s|z|e^{i\tau}|\geq s|z|-1$. We observe that if $|z|>4$, we have that $1/|z|^{1/2}>2/|z|$, then $s>2/|z|$, and consequently,  $s|z|-1\geq s|z|/2$. Hence, for every $z\in S(\tau,\ga)$ with $|z|>4$ we see that
\begin{equation}\label{ineq.limit.g.infinity.2}
 \left|\int_{1/|z|^{1/2}}^{\oo} e(se^{i(\tau-\theta)}) \frac{ds}{s(1-s|z|e^{i\tau})}\right|\leq
\frac{2c}{|z|} \int_{1/|z|^{1/2}}^{\oo} \frac{ ds}{s^2} = \frac{2c}{|z|^{1/2}} .
\end{equation}
The right hand side of this inequality tends to $0$
as $|z|$  goes to infinity. By~\eqref{ineq.limit.g.infinity.1} and~\eqref{ineq.limit.g.infinity.2}, we see that $g(z)\to0$ as $|z|\to\oo$ uniformly for $z\in  S(\tau,\ga)$. By a compactness argument,  we see that $g(z)\to0$ as $|z|\to\oo$ whenever
$z\in S(\pi,2+\o(\M))$ uniformly for $z\in  S(\pi,2+\ga)$.

Let $\theta\in\R$ be a direction such that $|\theta|<\pi\o(\M)/2$ and $z\in\mathcal{R}$ with $\mathrm{arg}(z)=\theta$.
There exists $\ep\in (0,\pi\o(\M)/2)$ such that:
 \begin{enumerate}
  \item For every $u\in\mathcal{R}$ with $\mathrm{arg}(u)=\ep$ we have that $u/z\in S_{\o(\M)}$.
  \item For every $u\in\mathcal{R}$ with $\mathrm{arg}(u)=2\pi-\ep$ we have that $u/ze^{2\pi i}\in S_{\o(\M)}$.
 \end{enumerate}
Then, since $f$ is single-valued, we have that
\begin{align*}
 g(z)-g(ze^{2\pi i}) = &
\int_0^{\infty(\ep)}\frac{e(u/z)du}{(1-u)u} - \int_0^{\infty(2\pi-\ep)}\frac{e(u/(ze^{2\pi i}))du}{(1-u)u} \\
= & \int_0^{\infty(\ep)}\frac{e(u/z)du}{(1-u)u} - \int_0^{\infty(-\ep)}\frac{e(w/z)dw}{(1-w)w}.
\end{align*}
We denote by $\gamma_r$ the arc of radius $r>1$ from $re^{i\ep}$ to $re^{-i\ep}$ traversed clockwise. We observe that
$$\left|\int_{\gamma_r} \frac{e(u/z)du}{(1-u)u} \right|= \left|\int_{-\ep}^{\ep}\frac{e(re^{i\theta}/z)i d\theta}{(1-re^{i\theta})}\right|
\leq \int_{-\ep}^{\ep}\frac{|e(re^{i\theta}/z)| d\theta}{|r-1|} \leq 2\ep c \frac{h_{\M} (k |z|/r)}{r-1}
\leq \frac{2\ep c }{r-1}. $$
Hence, we deduce that
$$\lim_{r\to\oo} |\int_{\gamma_r} \frac{e(u/z)du}{(1-u)u} | = 0.$$
The residue of $h_z(u) = e(u/z)/(u(1-u))$ at $u=1$ being $-e(1/z)$,
according to the Residue theorem we conclude that
$g(z)-g(ze^{2\pi i}) = 2\pi i e(1/z)$.
\end{proof}

\begin{rema}\label{rema.unique.e}
Let $e$ and $\overline{e}$ be two $\M-$summability kernels whose moment sequence $\mf=(m(n))_{n\in\N_0}$ is the same. By the above lemma, for $f(z)=1/(1-z)$, we have that
 $T_{e}{f}\sim_{\M} \sum_{n=0}^\oo m(n) z^n$ and
$T_{\overline{e}}{f}\sim_{\M} \sum_{n=0}^\oo m(n) z^n$ on
$S(\pi,2+\o( \M ))$.
Observe that, according to Watson's Lemma, Theorem~\ref{TheoWatsonlemma}, $T_{e}{f}=T_{\overline{e}}{f}$.
By \eqref{eq.uniq.e}, we deduce that $e(z)=\overline{e}(z)$.
\end{rema}

\begin{rema}\label{rema.tends.0.at.0}
 Since  $g(z)\to0$ as $|z|\to\oo$ uniformly for $z\in  S(\pi,2+\ga)$ for every $\ga<\o(\M)$, by \eqref{eq.uniq.e} we also deduce that
$e(z)\to 0$, $|z|\to 0$, uniformly for $z\in  S_{\ga}$ for every $\ga<\o(\M)$, which does not follow immediately from Definition~\ref{defikernelMsumm}.
\end{rema}

\subsection{Strong kernels of $\M-$summability}
\label{subsect.strongkernels}

In order to recover the multisum of a formal power series, we need to combine a kernel $e_1$ of $\M_1-$summability with a kernel $e_2$ of $\M_2-$summability. The idea is to define new kernels $e_1\ast e_2$ and $e_1 \triangleleft e_2$ whose sequences of moments are $\mf_{e_1}\cdot\mf_{e_2}$ and $\mf_{e_2}/\mf_{e_1}$, respectively. This construction is based on the one given  in the Gevrey case by W.~Balser~\cite[Sect.\ 5.8]{Balser2000}. Nevertheless, in this general situation, a stronger notion of summability kernel should be considered which will not entail a significant restriction (see Remark~\ref{rema.strong.kernels.norestriction}).

\begin{defi}~\label{NEWdefikernelMsumm}
 Let $\M$ be a strongly regular sequence with $\o(\M)<2$. A pair
 of complex functions $e$, $E$ are said to be \textit{strong kernel functions for
 $\M-$summability} if they satisfy (\textsc{i}) and (\textsc{iii})-(\textsc{v}) in Definition~\ref{defikernelMsumm} of kernel functions for $\M-$summability, and moreover:
 \begin{enumerate}
  \item[(\textsc{ii.b})] There exists $\a>0$ such that for every $\tau\in(0,\o(\M))$, there exist $C_\tau,\ep_\tau>0$ such that
  $$|e(z)|\leq C_\tau |z|^\a, \qquad \text{for all}\,\,\, z\in S_\tau, \qquad \text{with}\,\,\,|z|\leq \ep_\tau.$$
\item[(\textsc{vi.b})] There exists $\b>0$ such that for every $\tau\in(0,2-\o(\M))$, there exist $K_\tau,M_\tau>0$ such that
  $$|E(z)|\leq \frac{K_\tau}{|z|^\b}, \qquad \text{for all}\,\,\, z\in S(\pi,\tau), \qquad \text{with}\,\,\, |z|\geq M_\tau.$$
  \end{enumerate}
\end{defi}

\begin{rema}
In case $\o(\M) \geq2$, condition (\textsc{vi.b}) makes no sense and, in the same way as in~\cite[p.\ 90]{Balser2000} (see Remark~\ref{rema.ramif.kernels}), suitable adaptations should be made. For simplicity, we will omit this situation from now on.
\end{rema}

The next result justifies the terminology in the previous definition.

 \begin{lemma}
   Let $\M$ be a strongly regular sequence with $\o(\M)<2$.
   Let $e$ and $E$ be a pair of complex functions satisfying Definition~\ref{NEWdefikernelMsumm}, then they are kernels for $\M-$summability in the sense of Definition~\ref{defikernelMsumm}.
 \end{lemma}

 \begin{proof}
  We only have to check that $e$ and $E$ satisfy conditions (\textsc{ii}) and (\textsc{vi}) in Definition~\ref{defikernelMsumm}.
  We take $z_0\in S_{\o(\M)}$, we fix $r_0>0$ and $\tau_0\in(0,\o(\M))$ such that $U:=B(z_0,r_0)\en S_{\tau_0}$.
  By condition (\textsc{ii.b}), we have that
  $$|e(z)|\leq C_0 |z|^\a, \qquad z\in S_{\tau_0}, \qquad |z|\leq \ep_0.$$
  For $t\in(0,\ep_0 (|z_0|-r_0))$ and for every $z\in U$ we observe that $t/z\in S_{\tau_0}$ and $|t/z|\leq\ep_0$. Then
  $$\int_{0}^{\ep_0 (|z_0|-r_0)}\sup_{z\in U} |e(t/z)| \frac{dt}{t}
  \leq  \int_{0}^{\ep_0 (|z_0|-r_0)} \frac{C_0 t^{\a-1} dt}{ (|z_0|-r_0)^\a } \leq  \frac{C_0 \ep_0^\a}{\a}. $$
  We fix $T>0$, if $\ep_0 (|z_0|-r_0)\geq T$ condition  (\textsc{ii}) is immediately satisfied. If
  $ \ep_0 (|z_0|-r_0)< T $ we define $D_0:=\{t/z; \, z\in U, \,t\in[\ep_0 (|z_0|-r_0),T]\}\en S_{\tau_0}$, by condition (\textsc{i}), $e$ is continuous on  $ S_{\tau_0}$ and, since $D_0$ is contained on a compact subset of $ S_{\tau_0}$ ,  we have that $\sup_{w\in D_0} |e(w)|=K_0<\oo$. Then
 $$\int_{0}^{T}\sup_{z\in U} |e(t/z)| \frac{dt}{t} \leq \frac{C_0 \ep_0^\a}{\a} + \frac{T K_0}{\ep_0 (|z_0|-r_0)}<\oo.  $$

Analogously, we will verify condition (\textsc{vi}).  We take $z_0\in S(\pi, 2-\o(\M))$, we fix $r_0>0$ and
$\tau_0\in(0,2-\o(\M))$ such that $U:=B(z_0,r_0)\en S (\pi,\tau_0)$.
  By condition (\textsc{vi.b}), we have that
  $$|E(z)|\leq \frac{K_0}{|z|^\b}, \qquad z\in S(\pi,\tau_0), \qquad |z|\geq M_0.$$
  For $0<t\leq(|z_0|-r_0)/M_0$ and for every $z\in U$ we observe that $z/t\in S(\pi,\tau_0)$ and $|z/t|\geq M_0$. Then
  $$\int_{0}^{(|z_0|-r_0)/M_0}\sup_{z\in U} |E(z/t)| \frac{dt}{t}
  \leq  \int_{0}^{ (|z_0|-r_0)/M_0} \frac{K_0  dt}{ (|z_0|-r_0)^\b t^{1-\b}} \leq  \frac{K_0 }{M_0^\b \b}, $$
 and, since $E$ is entire, we conclude as before.
 \end{proof}

\begin{rema}\label{rema.strong.kernels.norestriction}
In general, for a kernel $e$ of $\M-$summability and thanks to Remark~\ref{rema.tends.0.at.0}, one can only guarantee that $e$ tends to $0$ in the regions considered in (\textsc{ii.b}) but it seems not possible to ensure that it has power-like growth, likewise for $E$.
However, either the  classical kernels in the Gevrey theory $e_k(z)=kz^k\exp(-z^k)$  (see~\cite{Balser2000}), or the new ones $e_V(z)=z\exp(-V(z))$, constructed for sequences admitting a nonzero proximate order (see Remark~\ref{rema.def.exist.kernels}) by using the functions $V$ of Maergoiz~\cite{Maergoiz} (see~\cite[Th. 4.8,\, Prop. 4.11]{lastramaleksanz15}), satisfy conditions (\textsc{ii.b}) and (\textsc{vi.b}) (see~\eqref{eq.eV.bound.IIB}, \eqref{eq.EV.bound.VIB}).

Moreover, if one wants to prove the integrability conditions (\textsc{ii}) or (\textsc{vi}) in some concrete example, one usually ends up showing estimates as those in (\textsc{ii.b}) and (\textsc{vi.b}).

This stronger notion of kernel functions needs to be considered to assure that the convolution and the acceleration kernels, defined in the forthcoming subsections from two given  kernels $e_1$ and $e_2$, also satisfy adequate integrability properties which seem not to be preserved in the standard situation.
\end{rema}

\begin{rema}\label{rema.change.alpha.beta}
We observe that once condition (\textsc{ii.b}) or (\textsc{vi.b}) is satisfied for some values $\a$ and $\b$,
it is possible to replace  $\a$ for any $0<\a'<\a$ and $\b$ for any $0<\b'<\b$ and the corresponding conditions hold.
\end{rema}

\subsection{Convolution kernels}
\label{subsect.convolutionkernels}

In this subsection, we consider two strong kernels $e_1$ and $e_2$ satisfying the properties in Definition~\ref{NEWdefikernelMsumm} for two sequences $\M_1$ and $\M_2$ with corresponding operators $T_{e_j}, T^{-}_{e_j}$ and moment functions $m_{e_j}(\lambda)$ for $j=1,2$. We will find a pair of operators $T, T^{-}$ such that $T$ coincides
with $T_{e_1}\circ T_{e_2}$ for a suitable class of functions containing the monomials. Hence, we will deduce that the moment function $m(\lambda)$ associated with $T$ equals $m_{e_1}(\lambda)m_{e_2}(\lambda)$. The kernel that defines the operator $T$  will be obtained as a Mellin convolution of the kernels $e_1$ and $e_2$,  which justifies its name. \para

First, we prove an auxiliary lemma that connects the associated function of two weight sequences $\M_1$ and  $\M_2$ with the associated function of their product sequence $\M_1\cdot \M_2$. This will be essential when dealing with functions in the classes $\mathcal{O}^{\M_1}$, $\mathcal{O}^{\M_2}$ and $\mathcal{O}^{\M_1\cdot\M_2}$.

 \begin{lemma}\label{lemma.propierties.ommega.funct.M1M2}
  Let $\M_j$, $j=1,2$, be weight sequences, for every $s,r>0$ we have that
\begin{align}\label{ineq.fundamental.M12.M1.M2}
e^{\o_{\M_1\cdot\M_2}(r)} \leq e^{\o_{\M_1}(s)}e^{\o_{\M_2}(r/s)}.
  \end{align}
 \end{lemma}
\begin{proof}
We write $\M_1 = (M_{1,p})_{p\in\N_0}$ and $\M_2= (M_{2,p})_{p\in\N_0}$.
For every $s,r>0$, we observe that
$$e^{\o_{\M_1\cdot \M_2}(r)}=\sup_{p\in\N_0} \frac{r^p}{M_{1,p} M_{2,p}} =\sup_{p\in\N_0} \frac{s^p}{M_{1,p} } \frac{(r/s)^p}{ M_{2,p}} \leq \sup_{p\in\N_0} \frac{s^p}{M_{1,p}} \sup_{p\in\N_0}  \frac{(r/s)^p}{ M_{2,p}}  = e^{\o_{\M_1}(s)}e^{\o_{\M_2}(r/s)}.$$
\end{proof}

We say that a weight sequence $\M$ is \textit{normalized} if $m_0=M_1=1$,  which by log-convexity implies that
$m_p\geq1$, $M_p\leq M_{p+1}$ and $M_p\geq1$ for all $p\in\N_0$.
Given two normalized weight sequences $\L$ and $\M$, the inequality
  \begin{align}\label{ineq.natural.M12.M2}
   \min(\o_{\L}(t),\o_{\M}(t))\geq \o_{\L\cdot\M}(t) \qquad  \text{for}\,\,\, t>0
  \end{align}
follows directly from the definition of the associated functions, since $L_{p}\leq L_{p} M_{p}$ and $M_{p}\leq L_{p} M_{p}$  for all $p\in\N_0$.\para

For arbitrary weight sequences, \eqref{ineq.natural.M12.M2} is satisfied for $t$ large enough. However, normalization is not a significant restriction since $m_p\geq1$ for $p$ large and we can modify the first terms of a sequence
in order to get a normalized weight sequence $\M'$ with $\m'\simeq\m$. This assumption simplifies in a considerable way the proofs of the forthcoming results.

The results until the end of the section might be stated for normalized strongly regular sequences for which a strong kernel exists. Nevertheless, the existence of such kernels has only been proved for sequences admitting a nonzero proximate order which, as already pointed out, are the ones appearing in the applications.

The following proposition shows the convergence of the double integral~\eqref{ineq.double.mod.int.T1.T2} that will ensure that the operators $T$ and $T_{e_1}\circ T_{e_2}$ coincide.

\begin{pro}\label{prop.doubleintegral.mod.T1T2}
Let $\M_j$, $j=1,2$, be normalized weight sequences admitting a nonzero proximate order.
We consider strong kernels $e_j$ for $\M_j-$summability, its moment function $m_{e_j}$ and $T_{e_j}$ the corresponding Laplace-like operators. If $f\in\mathcal{O}^{\M_1\cdot\M_2}(S(d,\ga))$, then there exists a sectorial region
$G(d,\ga+\o(\M_1)+\o(\M_2))$ such that for every $z_0\in G(d,\ga+\o(\M_1)+\o(\M_2))$ there exist a neighborhood
$U_0\subseteq G(d,\ga+\o(\M_1)+\o(\M_2))$ of $z_0$ and directions $\theta$ and $\phi$ (depending on $z_0$) such that we have
\begin{equation}\label{ineq.double.mod.int.T1.T2}
 \int^{\oo(\theta+\phi)}_{0} \int^{\oo (\theta)}_{0} \left| e_1 \left(\frac{v}{z}\right) e_2 \left(\frac{u}{v}\right) f(u)\frac{du}{u}\frac{dv}{v}\right| <\oo,
\end{equation}
for every $z\in U_0$. Consequently, $T_{e_1}\circ T_{e_2}(f) (z)  $ is holomorphic in $G(d,\ga+\o(\M_1)+\o(\M_2))$ and
$$T_{e_1}\circ T_{e_2}(f) (z) = \int_{0}^{\oo(\theta)}  f(u) \left(\int_{0}^{\oo(\phi)}  e_1(wu/z) e_2 (1/w)  \frac{dw}{w}\right) \frac{du}{u}. $$
\end{pro}

\begin{proof}
We write $\M_1 = (M_{1,p})_{p\in\N_0}$ and $\M_2= (M_{2,p})_{p\in\N_0}$ and, for simplicity,  $\o_1=\o(\M_1)$ and $\o_2=\o(\M_2)$.
We fix $\psi_0\in(d-(\o_1+\o_2+\ga)\pi/2,d+(\ga+\o_1+\o_2)\pi/2 )$, we choose directions $\tau_1\in(0,\o_1)$, $\tau_2\in(0,\o_2)$, $\tau_3\in (0,\ga)$, $\theta$ and $\phi$  with
$|\theta-d|<\tau_3\pi/2$ and $|\phi|\leq\pi\tau_2/2$ such that
\begin{equation}\label{eq.Dom.def.T1T2}
 |\theta+\phi-\psi_0|< \pi \tau_1/2.
\end{equation}
Then, there exists $\ep>0$,  such that $ [\psi_0-\ep,\psi_0+\ep]\subseteq  (d-(\o_1+\o_2+\ga)\pi/2,d+(\ga+\o_1+\o_2)\pi/2 )$ and~\eqref{eq.Dom.def.T1T2} remains true if we replace
$\psi_0$ by $\psi$  for every $\psi\in(\psi_0-\ep,\psi_0+\ep)$.
By Definition~\ref{NEWdefikernelMsumm}~(\textsc{ii.b}), for $e_1$ and $e_2$ there exist $\a_1,\a_2>0$ (not depending on $\tau_1$ and $\tau_2$), and constants $C_1,C_2>0$, $\ep_1,\ep_2\in(0,1)$ such that
\begin{align}
 |e_1(w)|\leq& C_1 |w|^{\a_1}, \qquad w\in S_{\tau_1}, \qquad |w|\leq \ep_1, \label{Condition.(II.B).e1.T1T2}\\
 |e_2(w)|\leq& C_2 |w|^{\a_2}, \qquad w\in S_{\tau_2}, \qquad |w|\leq \ep_2. \label{Condition.(II.B).e2.T1T2}
\end{align}
Using condition (\textsc{iii}), for $e_1$ and $e_2$, there exist $d_1,d_2,k_1,k_2>0$ such that
\begin{align}
 |e_1(w)| \le  d_1\,e^{-\o_{\M_1}(k_1|w|)},\qquad w\in S_{\tau_1}, \label{Condition.(III).e1.T1T2}\\
 |e_2(w)| \le  d_2\,e^{-\o_{\M_2}(k_2|w|)},\qquad u\in S_{\tau_2}. \label{Condition.(III).e2.T1T2}
\end{align}
Since $f\in\mathcal{O}^{\M_1\cdot\M_2}(S(d,\ga))$, we see that there exist $d_3,k_3>0$ such that
\begin{align}
 |f(w)| \le  d_3\,e^{\o_{\M_1\cdot\M_2}(k_3|w|)},\qquad w\in S(d,\tau_3). \label{Condition.growth.f.T1T2}
\end{align}
Now, we define $k_4:= \max (\ep_2, A_2/k_2)$ where $A_{2}$ is the constant appearing in~\eqref{ineq.def.moderate.growth.ommegaM} for $ \M_2$ and $s=2$. We fix $z_0\in S(d,\ga+\o_1+\o_2)$, with $\arg(z_0)\in (\psi_0-\ep,\psi_0+\ep)$ and $|z_0|<k_1/(k_3 k_4 A_1)$, where $A_1$ is the constant appearing in~\eqref{ineq.def.moderate.growth.ommegaM} for $\M_1$ and $s=1$. We consider $U_0:=B(z_0,\ro_0)$ centered in $z_0$ such that $\overline{U_0}\en S(\psi_0,\ep, k_1/(k_3 k_4 A_1))$. \para

In order to prove~\eqref{ineq.double.mod.int.T1.T2}, parametrizing the integral and using Tonelli's Theorem,  it is enough to show that
$$
 \int^{\oo}_{0} \left| e_1 \left(\frac{s e^{i(\theta+\phi)}}{z}\right) \right|\left( \int^{\oo }_{0} \left| e_2 \left(\frac{r}{se^{i\phi}}\right)\right| |f(re^{i\theta})|\frac{dr}{r}\right)\frac{ds}{s} <\oo,
$$
for every $z\in U_0$.
We fix $\a<\min(\a_1,1)$ and $$s_0=\min(1,k_2/(k_3 A_{1,2}),\ep_1(|z_0|-\ro_0)),$$
where $A_{1,2}$ is the constant appearing in~\eqref{ineq.def.moderate.growth.ommegaM} for $\M_1\cdot \M_2$ and $s=2$. For all $s<s_0$, we observe that
$$I(s):=\int^{\oo }_{0} \left| e_2 \left(\frac{r}{se^{i\phi}}\right)\right| |f(re^{i\theta})|\frac{dr}{r} =
\left(\int^{\ep_2 s}_{0} + \int^{\ep_2 s^{\a}}_{\ep_2 s}  + \int^\oo_{\ep_2 s^{\a}} \right) \left| e_2 \left(\frac{r}{se^{i\phi}}\right)\right| |f(re^{i\theta})|\frac{dr}{r}.
$$
We split the integral into three parts $I_j(s)$ for $j=1,2,3$ defined below. Since $r/(se^{i\phi})\in S_{\tau_2} $ and $|r/(se^{i\phi})|\leq \ep_2$ for all $r\in(0,\ep_2 s)$, by~\eqref{Condition.(II.B).e2.T1T2} and \eqref{Condition.growth.f.T1T2}, we have that
$$I_1(s):=\int^{\ep_2 s}_{0} \left| e_2 \left(\frac{r}{se^{i\phi}}\right)\right| |f(re^{i\theta})|\frac{dr}{r}\leq
\frac{C_2 d_3}{s^{\a_2}}\int^{\ep_2 s}_{0} r^{\a_2}  e^{\o_{\M_1\cdot \M_2}(k_3 r)}\frac{dr}{r}.$$
Using that ${\o_{\M_1\M_2}(k_3 r)}$ is nondecreasing we see that
\begin{align}\label{I1.s.T1T2}
 I_1(s)\leq \frac{C_2 d_3 \ep_2^{\a_2}}{\a_2} \exp(\o_{\M_1\cdot \M_2}(k_3 \ep_2 s)).
\end{align}
By~\eqref{Condition.(III).e2.T1T2} and
\eqref{Condition.growth.f.T1T2}, we see that
$$I_2(s):= \int^{\ep_2 s^{\a}}_{\ep_2 s} \left| e_2 \left(\frac{r}{se^{i\phi}}\right)\right| |f(re^{i\theta})|\frac{dr}{r}\leq d_2 d_3\int^{\ep_2 s^{\a}}_{\ep_2 s}   e^{-\o_{ \M_2}(k_2 r/s )}  e^{\o_{\M_1\cdot \M_2}(k_3 r)}\frac{dr}{r}.$$
Using again that ${\o_{\M_1\cdot \M_2}(k_3 r)}$ is nondecreasing and that
\begin{equation}\label{ineq.def.M2}
 \exp (-\o_{ \M_2}(k_2 r/s ))= h_{\M_2}(s/(k_2r))=\inf_{p\in\N_0} M_{2,p} \frac{s^p}{(k_2r)^p}\leq M_{2,1} \frac{s}{k_2r},
\end{equation}
we get that
\begin{align}\label{I2.s.T1T2}
 I_2(s)\leq \frac{d_2 d_3 M_{2,1} }{k_2} s \exp(\o_{\M_1\cdot \M_2}(k_3 \ep_2 s^\a))\int^{\ep_2 s^{\a}}_{\ep_2 s}   \frac{dr}{r^2}\leq \frac{d_2 d_3 M_{2,1}  }{k_2\ep_2} \exp(\o_{\M_1\cdot \M_2}(k_3 \ep_2 s^\a) ).
\end{align}
By~\eqref{Condition.(III).e2.T1T2} and
\eqref{Condition.growth.f.T1T2} again, we can see that
$$I_3(s):= \int^\oo_{\ep_2 s^{\a}}  \left| e_2 \left(\frac{r}{se^{i\phi}}\right)\right| |f(re^{i\theta})|\frac{dr}{r}\leq d_2 d_3\int^\oo_{\ep_2 s^{\a}}  e^{-\o_{ \M_2}(k_2 r/s )}  e^{\o_{\M_1\cdot \M_2}(k_3 r)}\frac{dr}{r}.$$
Using~\eqref{ineq.natural.M12.M2}  and Lemma~\ref{lemma.MG.associated.function} for  $\o_{\M_1\cdot \M_2}$, we obtain that
$$I_3(s)\leq d_2 d_3 \int^\oo_{\ep_2 s^{\a}}   e^{-\o_{\M_1\cdot \M_2}(k_2 r/s )}  e^{\o_{\M_1\cdot \M_2}(k_3 r)}\frac{dr}{r}\leq d_2 d_3
\int^\oo_{\ep_2 s^{\a}}   e^{-2\o_{\M_1\cdot \M_2}((k_2 r)/(sA_{1,2}) )+\o_{\M_1\cdot \M_2}(k_3 r)}\frac{dr}{r}.$$
Since $\o_{\M_1\cdot \M_2}(t)$ is nondecreasing and $s<s_0\leq k_2/(k_3A_{1,2})$, we have that
$$  I_3(s)\leq d_2 d_3 \int^\oo_{\ep_2 s^{\a}}     e^{-\o_{\M_1\cdot \M_2}(k_3 r)}\frac{dr}{r}. $$
Finally, by the definition of $\o_{\M_1\cdot \M_2}(t)$ we obtain
\begin{align}\label{I3.s.T1T2}
 I_3(s)\leq d_2 d_3 \int^\oo_{\ep_2 s^{\a}}   \frac{M_{1,1}M_{2,1}}{ k_3 r^2}  dr = \frac{d_2 d_3 M_{1,1}M_{2,1}}{ k_3 \ep_2 s^{\a}} .
\end{align}
Consequently, by \eqref{I1.s.T1T2}, \eqref{I2.s.T1T2} and~\eqref{I3.s.T1T2}, for all $s<s_0$ we see that
\begin{align}\label{I.s.small.T1T2}
 I(s)&\leq \frac{C_2 d_3 \ep_2^{\a_2}}{\a_2} \exp(\o_{\M_1\cdot \M_2} (k_3 \ep_2 s_0)) +  \frac{d_2 d_3 M_{2,1}  }{k_2\ep_2} \exp(\o_{\M_1\cdot \M_2}(k_3 \ep_2 s_0^\a) ) + \frac{d_2 d_3 M_{1,1}M_{2,1}}{ k_3 \ep_2 s^{\a}}\no \\
 &= a_1+\frac{a_2}{s^\a} .
\end{align}
Now, for every $s\geq s_0$ we split the integral into two parts $\widetilde{I}_j(s)$ for $j=1,2$. Since $r/(se^{i\phi})\in S_{\tau_2} $ and $|r/(se^{i\phi})|\leq \ep_2$ for all $r\in(0,\ep_2 s)$, by~\eqref{Condition.(II.B).e2.T1T2} and \eqref{Condition.growth.f.T1T2}, as before we get that
$$\widetilde{I}_1(s):=\int_{0}^{\ep_2 s} \left| e_2 \left(\frac{r}{se^{i\phi}}\right)\right| |f(re^{i\theta})| \frac{dr}{r} \leq
\frac{C_2 d_3}{s^{\a_2}}\int^{\ep_2 s}_{0} r^{\a_2}  e^{\o_{\M_1\cdot \M_2} (k_3 r)}\frac{dr}{r}.$$
Using~\eqref{ineq.natural.M12.M2} and the monotonicity of $\o_{\M_1}(t)$,  we deduce that
\begin{align}\label{I1.s.big.T1T2}
 \widetilde{I}_1(s)\leq \frac{C_2 d_3 \ep_2^{\a_2}}{\a_2} \exp(\o_{\M_1}(k_3 \ep_2 s)).
\end{align}
By~\eqref{Condition.(III).e2.T1T2} and
\eqref{Condition.growth.f.T1T2} again, we can see that
$$\widetilde{I}_2(s):=\int_{\ep_2 s}^\oo \left| e_2 \left(\frac{r}{se^{i\phi}}\right)\right| |f(re^{i\theta})| \frac{dr}{r}  \leq d_2 d_3\int^\oo_{\ep_2 s}  e^{-\o_{\M_2}(k_2 r/s )}  e^{\o_{\M_1\cdot \M_2}(k_3 r)}\frac{dr}{r}.$$
Applying Lemma~\ref{lemma.MG.associated.function} to $\o_{\M_2}$ and by~\eqref{ineq.fundamental.M12.M1.M2} we can show that
$$\widetilde{I}_2(s) \leq d_2 d_3\int^\oo_{\ep_2 s}  e^{-2 \o_{\M_2}(k_2 r/(sA_2) )} e^{\o_{\M_1}(k_3A_2s/k_2)} e^{\o_{\M_2}(k_2 r/(sA_2) )} \frac{dr}{ r}.$$
Using the definition of $\o_{\M_2}(t)$, as in~\eqref{ineq.def.M2}, we deduce that
$$\widetilde{I}_2(s) \leq d_2 d_3 e^{\o_{\M_1}(k_3A_2s/k_2)} \int^\oo_{\ep_2 s}  e^{- \o_{\M_2}(k_2 r/(sA_2) )}  \frac{dr}{ r} \leq
 \frac{d_2 d_3 A_2  M_{2,1} }{k_2\ep_2} e^{\o_{\M_1}(k_3A_2s/k_2)}  .$$
Together with \eqref{I1.s.big.T1T2}, for all $s\geq s_0$ we get that
$$
 I(s) \leq \frac{C_2 d_3 \ep_2^{\a_2}}{\a_2} \exp(\o_{\M_1}(k_3 \ep_2 s)) +\frac{d_2 d_3 A_2  M_{2,1} }{k_2\ep_2} \exp (\o_{\M_1}(k_3A_2s/k_2)) .$$
Since $k_4=\max (\ep_2, A_2/k_2)$ and $\o_{\M_1}(t)$ is nondecreasing, for every $s\geq s_0$ we have shown that
\begin{align}\label{I.s.big.T1T2}
 I(s) \leq b_1 \exp(\o_{\M_1}(k_3 k_4 s)) .
\end{align}

Consequently, for any $z\in U_0$ and any $s\in (0,s_0)$ we have that
 $s e^{i(\theta+\phi)}/z\in S_{\tau_1} $ and $s/|z|\leq s_0/|z| \leq  \ep_1$ and, by~\eqref{Condition.(II.B).e1.T1T2} and
\eqref{I.s.small.T1T2}, we deduce that
$$J_1(z):=\int_{0}^{s_0} \left| e_1 \left(\frac{s e^{i(\theta+\phi)}}{z}\right) \right| I(s) \frac{ds}{s} \leq
   \frac{C_1}{|z|^{\a_1}} \int_{0}^{s_0} (a_1 s^{\a_1} + a_2 s^{\a_1-\a}) \frac{ds}{s}. $$
Since $\a<\a_1$, for every $z\in U_0$ we see that
\begin{equation}\label{Int.global.s.small}
 \int_{0}^{s_0} \left| e_1 \left(\frac{s e^{i(\theta+\phi)}}{z}\right) \right| I(s) \frac{ds}{s} \leq
   \frac{C_1 a_1 s_0^{\a_1}}{\a_1 (|z_0|-\ro_0) ^{\a_1}} +  \frac{C_1 a_2 s_0^{\a_1-\a}}{(\a_1-\a) (|z_0|-\ro_0) ^{\a_1}}.
\end{equation}

In the same way, applying~\eqref{Condition.(III).e1.T1T2}  and~\eqref{I.s.big.T1T2}, we get that
\begin{align*}
 J_2(z):= \int_{s_0}^\oo \left| e_1 \left(\frac{s e^{i(\theta+\phi)}}{z}\right) \right| I(s) \frac{ds}{s}\leq&
 d_1b_1 \int_{s_0}^\oo\exp(-\o_{\M_1}(sk_1/|z|)+\o_{\M_1}(k_3k_4 s)) \frac{ds}{s} .
\end{align*}
 By Lemma~\ref{lemma.MG.associated.function} for $\o_{\M_1}$ and since $|z|<k_1/(k_3 k_4 A_1)$ for every $z\in U_0$, we deduce that
$$J_2(z)\leq d_1b_1 \int_{s_0}^\oo\exp(-2\o_{\M_1}(sk_1/(|z|A_1))+\o_{\M_1}(k_3 k_4 s)) \frac{ds}{s} \leq  d_1b_1 \int_{s_0}^\oo\exp(-\o_{\M_1}(k_3 k_4 s)) \frac{ds}{s}, $$
Using the definition of $\o_{\M_1}(t)$, as in~\eqref{ineq.def.M2},  we get that
\begin{equation}\label{Int.global.s.big}
 J_2(z)\leq\frac{M_{1,1}}{k_3 k_4 s_0}.
\end{equation}
From~\eqref{Int.global.s.small} and~\eqref{Int.global.s.big}, we see that \eqref{ineq.double.mod.int.T1.T2} holds.

Hence, by the Leibniz's rule  we deduce that the double integral
\begin{equation}\label{int.double.without.mod.T1.T2}
 \int^{\oo(\theta+\phi)}_{0} \int^{\oo (\theta)}_{0}  e_1 \left(\frac{v}{z}\right) e_2 \left(\frac{u}{v}\right) f(u)\frac{du}{u}\frac{dv}{v}
\end{equation}
defines a holomorphic function in $U_0$. One can easily show that the value of such function does not depend on the directions $\phi$ and $\theta$, so it defines a holomorphic function in a sectorial region $G(d,\ga+\o_1+\o_2)$ and we observe that
$$T_{e_1}\circ T_{e_2} (f) (z)= \int_{0}^{\oo(\theta+\phi)} e_1(v/z) ( T_{e_2}f)(v)  \frac{dv}{v} = \int^{\oo(\theta+\phi)}_{0} \int^{\oo (\theta)}_{0}  e_1 \left(\frac{v}{z}\right) e_2 \left(\frac{u}{v}\right) f(u)\frac{du}{u}\frac{dv}{v},$$
where $\theta$ and $\phi$  (depending on $z$) are chosen as in the beginning of the proof.
Finally, since the double integral in~\eqref{int.double.without.mod.T1.T2} converges for any $z\in G(d,\ga+\o_1+\o_2) $, applying Fubini's Theorem and making the change of variables  $v=w u$, we see that
 $$T_{e_1}\circ T_{e_2} (f) (z) = \int_{0}^{\oo(\theta)}  f(u) \left(\int_{0}^{\oo(\phi)}  e_1(wu/z) e_2 (1/w)  \frac{dw}{w}\right) \frac{du}{u}.$$

 \end{proof}

In the proof of the last proposition, we have shown  that for any $f\in\mathcal{O}^{\M_1\cdot\M_2}(S(d,\ga))$ we have $T_{e_2}(f)\in\mathcal{O}^{\M_1}(S(d,\ga+\o(\M_2)))$  (see \eqref{I.s.small.T1T2}+\eqref{I.s.big.T1T2}), which extends the classical Gevrey result.

Finally, we construct the convolution kernel of two strong kernels and we prove that it is also a strong kernel of $\M_1\cdot\M_2-$summability.

\begin{pro}\label{prop.convolution.kernel}
Let $\M_j$, $j=1,2$, be normalized weight sequences admitting a nonzero proximate order. Assume $\o(\M_1)+\o(\M_2)<2$
and consider strong kernels $e_j$ of $\M_j-$summability, its moment function $m_{e_j}$ and $T_{e_j},T_{e_j}^{-}$ the corresponding Laplace or Borel operators.

\begin{enumerate}
 \item We define \textit{the convolution of $e_1$ and $e_2$}, denoted $e_1\ast e_2$, by
 $$e_1\ast e_2(z):=T_{e_1}(e_2(1/u))(1/z).$$
 Then,  $e_1\ast e_2$ is a  strong kernel  of $\M_1\cdot\M_2-$summability whose moment
function is $m(\lambda)=m_{e_1}(\lambda)m_{e_2}(\lambda)$. Moreover if $E_1$ and $E$ are the kernels associated by
Definition~\ref{NEWdefikernelMsumm}.(\textsc{v}) with $e_1$ and $e_1\ast e_2$, respectively,  we have that
$$E(z)=T_{e_2}^{-} E_1(z).$$

\item  The function $e_1\ast e_2$ is the unique moment summability kernel with moment sequence $(m(p)=m_{e_1}(p)m_{e_2}(p))_{p\in\N_0}$.

\item Let $T_{e_1}\ast T_{e_2}$ denote the Laplace-like integral operator associated with $e_1\ast e_2$. If $S$ is
an unbounded sector and $f\in\mathcal{O}^{\M_1\cdot\M_2}(S)$, then
$$(T_{e_1}\ast T_{e_2})f=T_{e_1}\circ T_{e_2}(f).$$

\item We consider $f(u)= (1-u)^{-1}$. We define $g(z):= ((T_{e_1} \circ T_{e_2} )f)(z)$ and
$$e(1/z):=\frac{g(z)-g(ze^{2\pi i})}{2\pi i}.$$
Then, $e$ is well defined in $S_{\o(\M_1)+\o(\M_2)}$ and $e(z)=e_1\ast e_2 (z)$.

\end{enumerate}
\end{pro}


\begin{proof} For simplicity we write $\o_1=\o(\M_1)$ and $\o_2=\o(\M_2)$. We observe that $\o(\M_1 \cdot \M_2) = \o_1+\o_2$
(see Remark~\ref{rema.omega.index.quotient.sequence}).

\begin{enumerate}

\item  Let us show that the function $e_1\ast e_2 (z)=T_{e_1}(e_2(1/u))(1/z)$ has the necessary properties for a strong kernel
function of $\M_1\cdot\M_2-$summability, as listed in Definition~\ref{NEWdefikernelMsumm}.
Since the function $e_2(1/u)$ is holomorphic in $S_{\o_2}$ and bounded on every sector $S_\b$ with $0<\b<\o_2$, by
Lemma~\ref{lemma.holomorphy.e.laplace.transform.bounded.functions}, we have that $T_{e_1} (e_2(1/u)) (z) $ is holomorphic in $S_{\o_1+\o_2}$ which proves (\textsc{i}). \para


Regarding the integrability condition (\textsc{ii.b}), we fix $\tau\in(0,\o_1+\o_2)$ and  we take $\tau_1\in (0,\o_1)$ and
$\tau_2\in(0,\o_2)$ such that $\tau<\tau_1+\tau_2$. By Definition~\ref{NEWdefikernelMsumm}~(\textsc{ii.b}) for $e_1$ and $e_2$ we know that there exist $\a_1,\a_2>0$ (not depending on $\tau_1$ and $\tau_2$), and constants $C_1,C_2>0$, $\ep_1,\ep_2\in(0,1)$ such that
\begin{align}
 |e_1(z)|\leq& C_1 |z|^{\a_1}, \qquad z\in S_{\tau_1}, \qquad |z|\leq \ep_1, \label{Condition.(II.B).e1}\\
 |e_2(z)|\leq& C_2 |z|^{\a_2}, \qquad z\in S_{\tau_2}, \qquad |z|\leq \ep_2. \label{Condition.(II.B).e2}
\end{align}
We fix $z\in S_\tau$ with $|z|\leq (\ep_1\ep_2)^2$ and we choose $|\theta|<\pi\tau_2/2$, such that $ ze^{i\theta}\in S_{\tau_1}$. We have that
\begin{align*}
 |e_1\ast e_2(z) | =& \left|\int_{0}^{\oo(\theta)} e_1 (uz) e_2\left(\frac{1}{u}\right) \frac{du}{u}  \right| \\
 \leq& \int_{0}^{\ep_1/|z|^{1/2}} \left|e_1 (re^{i\theta} z) e_2\left(\frac{1}{re^{i\theta}}\right)\right| \frac{dr}{r}
 +\int_{\ep_1/|z|^{1/2}}^{\oo} \left|e_1 (re^{i\theta} z) e_2\left(\frac{1}{re^{i\theta}}\right) \right|\frac{dr}{r}.
\end{align*}
 If $r\leq \ep_1/|z|^{1/2}$, we have that $|re^{i\theta} z|\leq  \ep_1 |z|^{1/2} \leq \ep_1^2 \ep_2 \leq \ep_1$, and if
$r\geq \ep_1/|z|^{1/2}$, we see that $|1/(re^{i\theta})| \leq |z|^{1/2}/\ep_1 \leq \ep_2 $. Applying  \eqref{Condition.(II.B).e1} and \eqref{Condition.(II.B).e2}, we obtain
\begin{equation*}
 |e_1\ast e_2(z) |
 \leq C_1 |z|^{\a_1} \int_{0}^{\ep_1/|z|^{1/2}} \left| e_2\left(\frac{1}{re^{i\theta}}\right)\right| \frac{dr}{r^{1-\a_1}}
 + C_2 \int_{\ep_1/|z|^{1/2}}^{\oo} \left|e_1 (re^{i\theta} z) \right|\frac{dr}{r^{1+\a_2}}.
\end{equation*}
By condition (\textsc{iii}) for $e_1$ and $e_2$, we know that there exist constants $D_1,D_2$ such that
$|e_1(w)|\leq D_1$ for every $w\in S_{\tau_1}$ and $|e_2(w)|\leq D_2$ for every $w\in S_{\tau_2}$. We deduce that
\begin{equation*}
 |e_1\ast e_2(z) |
 \leq \frac{C_1  D_2 \ep_1^{\a_1}}{\a_1} |z|^{\a_1/2}
 + \frac{C_2 D_1}{\ep_1^{\a_2}\a_2} |z|^{\a_2/2} .
\end{equation*}
Consequently, condition (\textsc{ii.b}) is satisfied with $\a=\min(\a_1/2,\a_2/2)$. \para

By condition (\textsc{iii}) for $e_2$, for every $\ep>0$, there exist $c,k>0$ such that
$$|e_2(1/u)| \le  c\,e^{-\o_{\M_2}(k/|u|)},\qquad u\in S_{\o_2-\ep}.$$
Then, by Proposition~\ref{prop.thilliez.flat.function.M}, we have that $e_2(1/u)\sim_{\M_2} \widehat{0}$ in $S_{\o_2}$.
By Remark~\ref{rema.asymp.laplace.transform.bounded.functions}, we see that $e_1 \ast e_2 (1/z)=T_{e_1}(e_2(1/u))(z)\sim_{\M_1\M_2} \widehat{0}$ in  $S_{\o_1+\o_2}$ which implies, again by Proposition~\ref{prop.thilliez.flat.function.M},  that for every $\ep>0$ there exist $c,k,r>0$ such that
\begin{equation}\label{ineq.condition.III.e1.ast.e2}
|e_1 \ast e_2(z)|\le c\,e^{-\o_{\M_1\cdot\M_2 }(|z|/k)},\qquad z\in S_{\o_1+\o_2-\varepsilon},\qquad |z|>r.
\end{equation}
 By condition (\textsc{ii.b}) for $e_1 \ast e_2$, we know that   $|e_1 \ast e_2(z)|\leq C$ for $z\in S_{\o_1+\o_2-\ep}$ with $|z|\leq \de $. Since $e_1 \ast e_2(z)$ is continuous, \eqref{ineq.condition.III.e1.ast.e2} holds for every $z\in S_{\o_1+\o_2-\varepsilon}$ and  we conclude that condition (\textsc{iii}) is satisfied.\para

Condition (\textsc{iv}) holds immediately because for $x>0$ we have that
$$e_1\ast e_2(x)=\int_{0}^\oo e_{1}(xy) e_2(1/y)\frac{dy}{y}.$$
Since $e_1$ and $e_2$ are positive real over the positive real axis, we deduce $e_1\ast e_2(x)$ also is.
Let us show that
\begin{equation}\label{eq.moment.equality}
m_{e_1\ast e_2} (\lambda)=m_{e_1}(\lambda)m_{e_2}(\lambda).
\end{equation}
We have that
$$m_{e_1\ast e_2}(\lambda)=\int_{0}^\oo x^{\lambda-1} (e_1\ast e_2) (x) dx = \int_{0}^\oo\int_{0}^\oo x^{\lambda-1}  e_{1}(xy) e_2(1/y)\frac{dy}{y} dx.$$
We make the change of variables $t=xy$ and $s=1/y$ and we get
$$m_{e_1\ast e_2}(\lambda)= \int_{0}^\oo\int_{0}^\oo (st)^{\lambda-1}  e_{1}(t) e_2(s) dt ds = m_{e_1}(\lambda)m_{e_2}(\lambda). $$
Consequently, using property (\textsc{v}) of $e_1$ and $e_2$, we deduce that $m_{e_1\ast e_2}(\lambda)$ is continuous in $\{\Re(\lambda)\geq0\}$, holomorphic in $\{\Re(\lambda)>0\}$ and  $m_e(x)>0$ for every $x\geq 0$.\para

We define the function $E$ by
$$E(z):= \sum^\oo_{n=0} \frac{z^n}{m_{e_1\ast e_2}(n)}, \qquad z\in\C.$$
If we compute the radius of convergence of this series, using \eqref{eq.moment.equality},  we see that
$$r=\liminf_{n\to\oo} \sqrt[n]{m_{e_1\ast e_2}(n)} = \liminf_{n\to\oo} \sqrt[n]{m_{e_1}(n)m_{e_2}(n)} =\oo, $$
hence $E$ is entire. We see that $(m_{e_1\ast e_2}(p))_{p\in\N_0}$, again by~\eqref{eq.moment.equality}, is equivalent to the sequence $\M_1\cdot\M_2$, then, by Proposition~\ref{propKomatsu}, there exist $C,K>0$ such that
$|E(z)|\le Ce^{\o_{\M_1\cdot\M_2}(K|z|)}$, $z\in\C$, then (\textsc{v}) holds.

Finally, regarding condition (\textsc{vi.b}),  we need first to show that
\begin{equation}\label{eq.E.Borel2.E1}
E(u)=(T_{e_2}^{-} E_1)(u), \qquad u\in \C^{*}.
\end{equation}
We fix $u\in \C^{*}$, write $\tau=\arg(u)$ and consider a path $\delta_{\o_2}(\tau)$ (see Definition~\ref{def.e.Borel}).
Since $E_1$ is entire and $\delta_{\o_2}(\tau)$ compact, we have that
$$|\sum_{n=0}^N z^n/m_{e_1}(n)|\leq E_1(|z|) \leq C_{\tau},\qquad z\in \delta_{\o_2}(\tau), \qquad N\in\N_0,$$
and by condition (\textsc{vi.b}) for $E_2$, $|E_2(u/z)z^{-1}|$ is integrable on $\de_{\o_2}(\tau)$, so we can apply Dominated Convergence Theorem. Therefore,
we can exchange integral and sum and, by Proposition~\ref{prop.trans.monom}, we see that
$$(T_{e_2}^{-} E_1)(u)=\sum_{n=0}^\oo (T_{e_2}^{-}(z^n/m_{e_1}(n)))(u)= \sum_{n=0}^\oo \frac{u^n}{m_{e_1\ast e_2}(n)}=E(u).$$

We fix $\tau\in(0,2-(\o_1+\o_2))$,  we take $\tau_1\in (0, 2-\o_1)$ and
$\tau_2\in(0,2-\o_2)$ such that $\tau+2\in(0,\tau_1+\tau_2-(2-\o_2-\tau_2))$ and we can choose $\ep\in(2-\o_2-\tau_2,\tau_1+\tau_2-2-\tau)$. By (\textsc{vi.b}) for $E_1$ and $E_2$, we know that there exist $\b_1,\b_2>0$ (not depending on $\tau_1$ and $\tau_2$), and constants $K_1,K_2>0$, $M_1,M_2\geq1$ such that
\begin{align}
 |E_1(z)|\leq& \frac{K_1}{ |z|^{\b_1}}, \qquad z\in S(\pi,\tau_1), \qquad |z|\geq M_1, \label{Condition.(VI.B).E1}\\
 |E_2(z)|\leq& \frac{K_2}{ |z|^{\b_2}}, \qquad z\in S(\pi,\tau_2), \qquad |z|\geq M_2. \label{Condition.(VI.B).E2}
\end{align}
We fix $u\in S(\pi,\tau)$ with $|u|\geq M_1 M_2$. We write $\phi=\arg(u)\in(0,2\pi)$ and we may consider a path
$\delta_{\o_2}(\phi)$ (see Definition~\ref{def.e.Borel}). We can write  $\delta_{\o_2}(\phi)=\delta_1+\delta_2+\delta_3$ ($\delta_1$ and $\delta_3$ are segments in directions $\theta_1=\phi+(\pi/2)(\o_2+\ep)$ and $\theta_3=\phi-(\pi/2)(\o_2+\ep)$, respectively, and $\delta_2$ is a circular arc with radius $R=|u|/M_2$, that can be chosen in this way because $E_1$ is entire). Then,  using~\eqref{eq.E.Borel2.E1}, we see that
$$
E(u)=\frac{-1}{2\pi i}\int_{\delta_{\o_2}(\phi)} E_{2} \left(\frac{u}{z}\right) E_{1} \left(z\right)\frac{dz}{z}.
$$
We have that
\begin{align}
 |E(u) | \leq& \frac{1}{2\pi} \left( \int_{0}^{|u|/M_2} \left|E_{2} \left(\frac{u}{re^{i\theta_1}}\right) E_{1} \left(re^{i\theta_1}\right)\right|\frac{dr}{r} +\int_{\theta_3}^{\theta_1} \left|E_{2} \left(\frac{uM_2}{|u|e^{i\theta}}\right) E_{1} \left(\frac{|u|e^{i\theta}}{M_2}\right)\right| d\theta \right.\no\\
+& \left. \int_{0}^{|u|/M_2} \left|E_{2} \left(\frac{u}{re^{i\theta_3}}\right) E_{1} \left(re^{i\theta_3}\right)\right| \frac{dr}{r} \right).\label{ineq.E.rays.arc.VIB.condition}
\end{align}

First, we study the second integral in~\eqref{ineq.E.rays.arc.VIB.condition}.
By condition (\textsc{v}) for $E_2$, we know that there exists a constant $H_2$ such that
$|E_2(z)|\leq H_2$ for every $z\in D(0,M_2+1)$ and we deduce that
$$\int_{\theta_3}^{\theta_1} \left|E_{2} \left(M_2e^{i(\phi-\theta)}\right) E_{1} \left(\frac{|u|e^{i\theta}}{M_2}\right)\right| d\theta \leq
H_2 \int_{\theta_3}^{\theta_1} \left| E_{1} \left(\frac{|u|e^{i\theta}}{M_2}\right)\right| d\theta.$$
Using the upper bounds of $\ep$ we see that  $2-\tau-\o_2-\ep>2-\tau_1$ and $2+\tau+\o_2+\ep<2+\tau_1$, then $[\theta_3,\theta_1]\en ((\pi/2)(2-\tau-\o_2-\ep) , (\pi/2) (2+\tau+\o_2+\ep))$ and we get that
\begin{equation}\label{arg.of.theta.in.S.tau1}
|u|e^{i\theta}/M_2\in S(\pi,\tau_1), \qquad \text{if}\,\,\,\theta\in [\theta_3,\theta_1].
\end{equation}
Since $|u|\geq M_1 M_2$ we have that $|ue^{i\theta}|/M_2\geq M_1$ and by~\eqref{Condition.(VI.B).E1}, we deduce that
\begin{equation}\label{ineq.E.arc}
H_2\int_{\theta_3}^{\theta_1} \left| E_{1} \left(\frac{|u|e^{i\theta}}{M_2}\right)\right| d\theta \leq
\frac{H_2 K_1 M^{\b_1}_2}{|u|^{\b_1}} \pi\tau_1, \qquad |u|\geq M_1 M_2.
\end{equation}

We study now the first and the last integral in~\eqref{ineq.E.rays.arc.VIB.condition}.  If $r\in(0,|u|/M_2)$, we observe that
  $|u/re^{i\theta_j}|\geq M_2$.   Since $\o_2+\tau_2<2$ and $\tau_1-2-\tau< -\o_1-\tau<0$, using the bounds for $\ep$, we have that
	$$\o_2+\ep \in (2-\tau_2,\tau_2+\o_2+\tau_1-2-\tau  )\en (2-\tau_2, 2 ) $$
	and we deduce that
	 $\arg(u/re^{i\theta_3})= (\pi/2)(\o_2+\ep)\in ((\pi/2)(2-\tau_2), \pi)$ and	 $\arg(u/re^{i\theta_1}) \in (-\pi, -(\pi/2)(2-\tau_2))$.
Then  for $j=1,3$ we  see that $u/re^{i\theta_j}\in S(\pi,\tau_2)$ and, by~\eqref{Condition.(VI.B).E2},  we show that
\begin{equation*}
\int_{0}^{|u|/M_2} \left|E_{2} \left(\frac{u}{re^{i\theta_j}}\right) E_{1} \left(re^{i\theta_j}\right)\right|\frac{dr}{r}\leq
\frac{K_2}{|u|^{\b_2}}\int_{0}^{|u|/M_2} r^{\b_2} \left| E_{1} \left(re^{i\theta_j}\right)\right|\frac{dr}{r}.
\end{equation*}
Since $|u|\geq M_1 M_2$, we can write $(0,|u|/M_2)$ as the disjoint union of the intervals $(0,M_1)$ and $[M_1,|u|/M_2)$.
By condition (\textsc{v}) for $E_1$, we know that there exists a constant $H_1$ such that
$|E_1(z)|\leq H_1$ for every $z\in D(0,M_1+1)$, then
$$\frac{K_2}{|u|^{\b_2}}\int_{0}^{|u|/M_2} r^{\b_2} \left| E_{1} \left(re^{i\theta_j}\right)\right|\frac{dr}{r}\leq
\frac{K_2H_1 M_1^{\b_2}}{|u|^{\b_2}\b_2}+ \frac{K_2}{|u|^{\b_2}}\int_{M_1}^{|u|/M_2} r^{\b_2} \left| E_{1} \left(re^{i\theta_j}\right)\right|\frac{dr}{r}.$$
If $r\geq M_1$, by~\eqref{arg.of.theta.in.S.tau1}, we can use~\eqref{Condition.(VI.B).E1} and we obtain
 $$
\frac{K_2H_1 M_1^{\b_2}}{|u|^{\b_2}\b_2}+ \frac{K_2}{|u|^{\b_2}}\int_{M_1}^{|u|/M_2} r^{\b_2} \left| E_{1} \left(re^{i\theta_j}\right)\right|\frac{dr}{r}\leq
\frac{K_2H_1 M_1^{\b_2}}{|u|^{\b_2}\b_2}+ \frac{K_1K_2}{|u|^{\b_2}}\int_{M_1}^{|u|/M_2} r^{\b_2-\b_1} \frac{dr}{r}.$$
According to Remark~\ref{rema.change.alpha.beta},  we may assume that $\b_1<\b_2$, and we conclude that
$$
\frac{K_2H_1 M_1^{\b_2}}{|u|^{\b_2}\b_2}+ \frac{K_1K_2}{|u|^{\b_2}}\int_{M_1}^{|u|/M_2} r^{\b_2-\b_1} \frac{dr}{r}\leq
\frac{K_2H_1 M_1^{\b_2}}{|u|^{\b_2}\b_2}+ \frac{K_1K_2}{|u|^{\b_1} (\b_2-\b_1) M_2^{\b_2-\b_1}}.$$
Then for $j=1,3$ and  for every $u\in S(\pi,\tau)$ with $|u|\geq M_1 M_2$ we have that
\begin{equation}\label{ineq.E.rays}
\int_{0}^{|u|/M_2} \left|E_{2} \left(\frac{u}{re^{i\theta_j}}\right) E_{1} \left(re^{i\theta_j}\right)\right|\frac{dr}{r}\leq
\frac{K_2H_1 M_1^{\b_2}}{|u|^{\b_2}\b_2}+ \frac{K_1K_2}{|u|^{\b_1} (\b_2-\b_1) M_2^{\b_2-\b_1}}.
\end{equation}
Consequently, by~\eqref{ineq.E.arc} and~\eqref{ineq.E.rays}, condition (\textsc{vi.b}) is satisfied with $\b=\min(\b_1,\b_2)$, for every $u\in S(\pi,\tau)$ with $|u|\geq M_1 M_2$.

\item  Uniqueness follows from Remark~\ref{rema.unique.e}.

\item We take $f\in\mathcal{O}^{\M_1\cdot\M_2}(S(d,\a))$. Since $e_1\ast e_2$ is a kernel of $\M_1 \dot \M_2-$summability,  by Proposition~\ref{prop.hol.lap} we know that $T_{e_1}\ast T_{e_2} (f)$ is holomorphic in
a sectorial region $G(d,\a+\o_1+\o_2)$. We fix  $z\in G(d,\a+\o_1+\o_2)$, there exists $\theta$ with $|\theta-d|<\pi\a/2$ such that if $\arg(u)=\theta$, then $u/z\in S_{\o_1+\o_2}$ and there exists $\phi$ with $|\phi|<\pi\o_2/2$ such that if
$\arg(w)=\phi$, then $wu/z\in S_{\o_1}$. We have that
\begin{align*}
 T_{e_1}\ast T_{e_2} (f) (z) &= \int_{0}^{\oo(\theta)} e_1 \ast e_2 (u/z) f(u) \frac{du}{u}\\
 &= \int_{0}^{\oo(\theta)}  f(u) \left(\int_{0}^{\oo(\phi)}  e_1(wu/z) e_2 (1/w)  \frac{dw}{w}\right) \frac{du}{u}.
\end{align*}
We conclude, using Proposition~\ref{prop.doubleintegral.mod.T1T2}, that this last expression is equal to $T_{e_1}\circ T_{e_2} (f) (z)$.

\item We consider $f(u)= (1-u)^{-1}$ and we define $g(z):= ((T_{e_1} \circ T_{e_2} )f)(z)$. By Lemma~\ref{lemma.e.Laplace.geom.series},  we know that $T_{e_2} f$ is holomorphic in $S(\pi,2+\o_2)$  and $T_{e_2}f(z)\to 0$ as $z\to\oo$ uniformly on every sector $S(\pi,2+\ga)$ with $\ga<\o_2$. Moreover,  $T_{e_2} f\sim_{\M_2} \sum_{n=0}^\oo m_2(n) z^n$  on $S(\pi,2+\o_2)$.  Then we deduce that $T_{e_2}f$ is bounded on every sector $S(\pi,2+\ga)$ with $\ga<\o_2$.
We can apply the $T_{e_1}$ transform  to $T_{e_2}f$ and, by Lemma~\ref{lemma.holomorphy.e.laplace.transform.bounded.functions}, we have that $g=T_{e_1}T_{e_2}f$ is holomorphic in $S=S(\pi,2+\o_1+\o_2)$.
Then, the function
\begin{equation*}
e(1/z):=\frac{g(z)-g(ze^{2\pi i})}{2\pi i},
\end{equation*}
is holomorphic in  $S_{\o_1+\o_2}$. Since $f\in \mathcal{O}^{\M_1\cdot\M_2}(S_{\o_1+\o_2})$, then, by statement 3,
$(T_{e_1}\ast T_{e_2})(f)= (T_{e_1}\circ T_{e_2})(f)$ and, by Lemmma~\ref{lemma.e.Laplace.geom.series},  we deduce that
\begin{equation*}
e(1/z)=\frac{g(z)-g(ze^{2\pi i})}{2\pi i} = \frac{ (T_{e_1}\ast T_{e_2})(f)(z)-((T_{e_1}\ast T_{e_2})f)(ze^{2\pi i})}{2\pi i}= e_1\ast e_2(1/z).
\end{equation*}

\end{enumerate}
\end{proof}

\begin{rema}\label{rema.non.stab.asymp.product}
We know that $\M_1\cdot\M_2$ is a weight sequence admitting a nonzero proximate order (see Proposition~\ref{prop.product.seq}). Consequently, we can construct a kernel $e$ of $\M_1\cdot\M_2-$summability (see~Remark~\ref{rema.def.exist.kernels}) which is indeed strong (Remark~\ref{rema.strong.kernels.norestriction}).  However, we do not have any control on the corresponding moment sequence of $e$ apart from being equivalent to $\M_1\cdot\M_2$. Last proposition guarantees that the moment sequence associated with $e_1\ast e_2$ is $(m_{e_1}(n)m_{e_2}(n))_{n\in\N_0}$, which is important because it ensures good behavior of formal and analytic Borel-Laplace operators with respect to asymptotics.

Note that $\mathcal{O}^{\M_1\cdot\M_2}(S)\subset \mathcal{O}^{\M_2}(S)$.
Consequently, $T_{e_1}\circ T_{e_2}$  extends the operator $T_{e_1} \ast T_{e_2}$. The opposite situation occurs for the acceleration operator that will be presented in the next subsection, whose main advantage is that it extends the composition operator.
\end{rema}

\subsection{Acceleration kernels}
\label{subsect.accelerationkernels}

In the same conditions as in the previous subsection, assuming in addition that $\o(\M_1)<\o(\M_2)$, we will construct a pair of operators $T, T^{-}$ such that $T$ extends $T^{-}_{e_1}\circ T_{e_2}$. This new operator will be called the acceleration operator from $e_2$ to $e_1$ because it will send a function $f\in \mathcal{O}^{\M_2}(S)$ into a function with greater growth $Tf\in\mathcal{O}^{\M_1}(\widetilde{S})$. According to this property, the kernel associated with $T$ will be called acceleration kernel and its corresponding sequence of moments will be $m_{e_2}(\lambda)/m_{e_1}(\lambda)$.\para

Our first result is the analogous version of Proposition~\ref{prop.doubleintegral.mod.T1T2} that guarantees that
the operators $T$ and $T^{-}_{e_1}\circ T_{e_2}$ coincide for a large enough class of functions.

 \begin{pro}\label{prop.doubleintegral.mod.B1L2}
Let $\M_j$, $j=1,2$, be weight sequences admitting a nonzero proximate order, $e_j$ be  strong kernels of $\M_j-$summability, and $T_{e_j}$, $T^{-}_{e_j}$ be the corresponding integral operators. Assume that $\o(\M_1)<\o(\M_2)<2$.\para

If $f\in\mathcal{O}^{\M_2}(S(d,\ga))$, then for every $z_0\in S(d,\ga+\o(\M_2)-\o(\M_1))$ there exist a neighborhood
$U_0\subseteq S(d,\ga+\o(\M_2)-\o(\M_1))$ of $z_0$ and a direction $\phi$ with $|d-\phi|<\pi\ga/2$ (depending on $z_0$) such that we have
\begin{equation}\label{ineq.double.mod.int.B1L2}
\int_{\delta_{\o_1}(\arg(z_0))}  \int_{0}^{\oo(\phi)} \left|E_1(z/v) e_2 (u/v) f(u) \frac{du}{u} \frac{dv}{v} \right|<\oo
\end{equation}
for every $z\in U_0$, where $\delta_{\o_1}(\arg(z_0))$ is a path as considered in Definition~\ref{def.e.Borel}. Moreover, the function $$F(z):=\int_{\delta_{\o_1}(\arg(z))}  \int_{0}^{\oo(\phi)} E_1(z/v) e_2 (u/v) f(u) \frac{du}{u} \frac{dv}{v}$$ is holomorphic in $S(d,\ga+\o (\M_2)-\o(\M_1))$
and
$$T^{-}_{e_1}\circ T_{e_2} (f) (z)= \int^{\oo (\phi)}_{0} f(u)  \left(\frac{-1}{2\pi i} \int_{\delta_{\o_1}(\arg(z)-\phi)}E_1(z/wu)  e_2 \left(1/w\right) \frac{dw}{w}\right)\frac{du}{u}. $$
\end{pro}

\begin{proof}
 For simplicity we write $\o_1=\o(\M_1)$ and $\o_2=\o(\M_2)$, $\M_1 = (M_{1,p})_{p\in\N_0}$ and $\M_2= (M_{2,p})_{p\in\N_0}$. We observe that $\o( \M_2/\M_1) = \o_2-\o_1$
(see Remark~\ref{rema.omega.index.quotient.sequence}).

We fix $z_0\in S(d,\ga+\o_2-\o_1)$, since $\o_2-\o_1>0$ we can consider directions  $\tau_1\in(0,2-\o_1)$, $\tau_2\in(0,\o_2)$, $\tau_3\in (0,\ga)$ with $2-\tau_1-\o_1 <\tau_2+\tau_1-2+\tau_3-\ga$. Then, we can choose
$$\ep\in (2-\tau_1-\o_1,\tau_2+\tau_1-2+\tau_3-\ga)\en(0,\o_2-\o_1),$$ and $\phi>0$ with $|\phi-d|<\tau_3\pi/2$ such that
\begin{equation}\label{ineq.z0.phi}
 |\arg(z_0)-\phi|< (\tau_2+\tau_1-2-\ep+\tau_3-\ga)\pi/2.
\end{equation}
 We observe that we can take  $\ro_0>0$ small enough such that $\overline{B(z_0,\ro_0)}\en S(d,\ga+\o_2-\o_1) $ and the inequality~\eqref{ineq.z0.phi} remains valid if we replace $\arg(z_0)$ by $\arg(z)$ for every $z\in B(z_0,\ro_0)$. We write $\theta_1=\arg(z_0)+(\o_1+\ep)\pi/2 $ and $\theta_3=\arg(z_0)-(\o_1+\ep)\pi/2 $, since $\tau_2<2$ we observe that the value of $\ep$ guarantees that $2-\o_1-\ep>0$ and $ \o_1+\tau_1-2+\ep>0 $. Then, by suitably reducing the radius $\ro_0$, we also have that
\begin{equation}\label{eq.election.de}
  |\arg(z)-\arg(z_0)|<\de:=\min((2-\o_1-\ep)\pi/2 ,(\o_1+\tau_1-2+\ep)\pi/2)/2,
\end{equation}
for every $z\in U_0=B(z_0,\ro_1)$ with  $\ro_1\leq\ro_0$, which implies
\begin{equation}\label{ineq.z.theta13}
-\pi <\arg(z)-\theta_1 < (\tau_1-2)\pi/2, \qquad  (2-\tau_1)\pi/2 < \arg(z)-\theta_3 <\pi  .
\end{equation}
Moreover, from~\eqref{ineq.z0.phi} we deduce that
\begin{equation}\label{ineq.phi.theta}
 |\theta-\phi|\leq |\theta-\arg(z_0)|+ |\arg(z_0)-\phi|< (\o_1 +\tau_1-2+\tau_2+\tau_3-\ga)\pi/2 \leq \tau_2\pi/2
\end{equation}
for every $\theta\in[\theta_3,\theta_1]$. \para

By Definition~\ref{NEWdefikernelMsumm}.(\textsc{vi.b}) for $E_1$ and  (\textsc{ii.b}) for $e_2$ we know that there exist $\b_1,\a_2>0$ (not depending on $\tau_1$ and $\tau_2$), and constants $K_1,C_2>0$, $M\geq1$, $\ep_2\in(0,1)$ such that
\begin{align}
 |E_1(z)|\leq& \frac{K_1}{ |z|^{\b_1}}, \qquad z\in S(\pi,{\tau_1}), \qquad |z|\geq M, \label{ineq.Condition.(VI.B).B1L2}\\
 |e_2(z)|\leq& C_2 |z|^{\a_2}, \qquad z\in S_{\tau_2}, \qquad |z|\leq \ep_2. \label{ineq.Condition.(II.B).e2.B1L2}
\end{align}
By condition (\textsc{iii}) for $e_2$, there exist $d_2,k_2>0$ such that
\begin{align*}
 |e_2(w)| \le  d_2\,e^{-\o_{\M_2}(k_2|w|)},\qquad u\in S_{\tau_2}, 
\end{align*}
and since $f\in\mathcal{O}^{\M_2}(S(d,\ga))$,
we see that there exist $d_3,k_3>0$ such that
\begin{align}
 |f(w)| \le  d_3\,e^{\o_{\M_2}(k_3|w|)},\qquad w\in S(d,\tau_3). \label{Condition.growth.f.B1L2}
\end{align}

We fix $s_0=\min(1,k_2/(k_3 A_{2}),(|z_0-\ro_1|/M))$ and $\b<\min(\b_1,1)$,  where $A_{2}$ is the constant appearing in~\eqref{ineq.def.moderate.growth.ommegaM} for $\M_2$ and $s=2$. We consider a path $\delta_{\o_1}(\arg(z_0))$ (see Definition~\ref{def.e.Borel}). We can write  $\delta_{\o_1}(\arg(z_0))=\delta_1+\delta_2+\delta_3$ ($\delta_1$ and $\delta_3$ are segments in directions $\theta_1=\arg(z_0)+(\o_1+\ep)\pi/2$ and $\theta_3=\arg(z_0)-(\o_1+\ep)\pi/2$, respectively, and $\delta_2$ is a circular arc with radius $R=s_0$).\para

In order to prove~\eqref{ineq.double.mod.int.B1L2}, parametrizing the integral and using Tonelli's Theorem,  it is enough to show that
$$
 J_i(z)= \int_{0}^{s_0} |E_1(z/se^{i\theta_i})| \int_{0}^{\oo} | e_2 (re^{i\phi-\theta_i}/s) f(re^{i\phi})| \frac{dr}{r} \frac{ds}{s}<\oo, \qquad i=1,3, $$
 $$J_2(z)= \int^{\theta_1}_{\theta_3} |E_1(z/s_0e^{i\theta})|  \int_{0}^{\oo} | e_2 (re^{i\phi-\theta}/s_0) f(re^{i\phi})| \frac{dr}{r} d\theta <\oo,$$
for every $z\in U_0$.
For $s\leq s_0$ and $\theta\in[\theta_3,\theta_1]$ we consider
$$I(s,\theta):= \int_{0}^{\oo} | e_2 (re^{i\phi-\theta}/s) f(re^{i\phi})| \frac{dr}{r}. $$
By~\eqref{ineq.phi.theta} we show that $re^{i\phi-i\theta}/s\in S_{\tau_2}$ for all $\theta\in[\theta_3,\theta_1]$ and every $s\leq s_0$. Then, splitting the interval into three parts $(0,\ep_2s)$, $(\ep_2 s, \ep_2s^\b)$, $(\ep_2s^\b,\oo)$ as in the proof of Proposition~\ref{prop.doubleintegral.mod.T1T2} and using~\eqref{ineq.Condition.(II.B).e2.B1L2}, \eqref{Condition.growth.f.B1L2}, Lemma~\ref{lemma.MG.associated.function} and that $\o_{\M_2}(t)$ is nondecreasing, we get that
\begin{equation}\label{ineq.I.theta.B1L2}
 I(s,\theta)\leq \frac{C_2 d_3\ep_2^{\a_2}}{\a_2}\exp(\o_{\M_2}(k_3\ep_2s_0))+
 \frac{d_2 d_3 M_{2,1}}{k_2 \ep_2}\exp(\o_{\M_2}(k_3\ep_2s_0^\b)) + \frac{d_2 d_3 M_{2,1}}{k_3\ep_2s^\b}  =: a_1+\frac{a_2}{s^\b}.
\end{equation}
Applying \eqref{ineq.z.theta13}, we see that $z/se^{i\theta_i}\in S(\pi,\tau_1)$ for $i=1,3$ and every $z\in U_0$ and since $|z/s|\geq |z|/s_0\geq M$, we can apply~\eqref{ineq.Condition.(VI.B).B1L2} and~\eqref{ineq.I.theta.B1L2} and we see  that
\begin{equation}\label{ineq.Ji.B1L2}
  J_i(z) \leq K_1 \int_{0}^{s_0} \frac{s^{\b_1}}{|z|^{\b_1}} \left(a_1+\frac{a_2}{s^\b}\right) \frac{ds}{s} \leq
  \frac{K_1 a_1 s_0^{\b_1}}{\b_1(|z_0|-\ro_1)^{\b_1}} + \frac{K_1 a_2 s_0^{\b_1-\b}}{ (\b_1-\b)(|z_0|-\ro_1)^{\b_1}},
\end{equation}
for $i=1,3$ and every $z\in U_0$.
Using that $E_1$ is entire we have that $|E_1(w)|\leq H_1$ for every $w\in B(0,(|z_0|+\ro_1+1)/s_0)$,
and~\eqref{ineq.I.theta.B1L2} shows that
\begin{equation}\label{ineq.J2.B1L2}
  J_2(z) \leq \left(a_1+\frac{a_2}{s_0^\b}\right)  (\theta_1-\theta_3) H_1
\end{equation}
for every $z\in U_0$. Using~\eqref{ineq.Ji.B1L2} and~\eqref{ineq.J2.B1L2} we see that~\eqref{ineq.double.mod.int.B1L2} holds.
Hence, by the Leibniz's rule  the double integral
\begin{equation*}
\int_{\delta_{\o_1}(\arg(z_0))}  \int_{0}^{\oo(\phi)} E_1(z/v) e_2 (u/v) f(u) \frac{du}{u} \frac{dv}{v}
\end{equation*}
is a holomorphic function in the neighborhood  $U_0$ of $z_0$ for every $z_0\in S(d,\ga+\o_2-\o_1)$. If $z\in U_0 \cap U_1$, with $U_1$ the corresponding neighborhood of $z_1$, the choice of $\de>0$ in \eqref{eq.election.de} guarantees that $\lim_{s\to0} |E_1(z/se^{i\theta})| I (s,\theta)=0$, uniformly for $\theta$ between $\theta_i$ and $\theta'_i=\arg(z_1)\pm (\o_1+\ep)\pi/2$ for $i=1,3$. This fact ensures that we can apply Cauchy's theorem to deform the path of
integration from $\delta_{\o_1}(\arg(z_0))$ to $\delta_{\o_1}(\arg(z_1))$, and we deduce that
\begin{equation}\label{int.double.without.B1L2}
\int_{\delta_{\o_1}(\arg(z))}  \int_{0}^{\oo(\phi)} E_1(z/v) e_2 (u/v) f(u) \frac{du}{u} \frac{dv}{v}
\end{equation}
defines a holomorphic in the sector $S(d,\ga+\o_2-\o_1)$. We observe that
\begin{align*}
T^{-}_{e_1}\circ T_{e_2} (f) (z)=& \frac{-1}{2\pi i}\int_{\delta_{\o_1}(\arg(z))} E_1(z/v) ( T_{e_2}f)(v)  \frac{dv}{v} \\= &\frac{-1}{2\pi i} \int_{\delta_{\o_1}(\arg(z))} \int^{\oo (\phi)}_{0}  E_1(z/v)  e_2 \left(\frac{u}{v}\right) f(u)\frac{du}{u}\frac{dv}{v},
\end{align*}
where  $\phi$ and $\delta_{\o_1}(\arg(z))$  are chosen as in the beginning of the proof.
We write $\eta=\arg(z)-\phi$ and we make the change of variables  $v=wu$. Then the path  $\delta_{\o_1}(\arg(z))$ is transformed into the path $\de_{\o_1} (\eta)$. We can write  $\de_{\o_1}(\eta)=\gamma_1+\gamma_2+\gamma_3$ ($\gamma_1$ and $\gamma_3$ are segments in directions $\theta''_1=\eta+(\pi/2)(\o_1+\ep)$ and $\theta''_3=\eta-(\pi/2)(\o_1+\ep)$, respectively, and $\gamma_2$ is a circular arc with radius $R=s_0/|u|$, the path
$\de_{\o_1}(\eta)$ stays inside $S_{\o_2}$. We have that
$$T^{-}_{e_1}\circ T_{e_2} (f) (z)= \frac{-1}{2\pi i}\int_{\de_{\o_1}(\eta)}  \int^{\oo (\phi)}_{0}  E_1(z/wu)  e_2 \left(1/w\right) f(u)\frac{du}{u}\frac{dw}{w}.$$
Finally, since the double integral in~\eqref{int.double.without.B1L2} converges for any $z\in S(d,\ga+\o_2-\o_1) $, we can apply Fubini's theorem and we can interchange the integration order, then we see that
 $$T^{-}_{e_1}\circ T_{e_2} (f) (z)= \int^{\oo (\phi)}_{0} f(u) \left( \frac{-1}{2\pi i}  \int_{\ga_{\o_1}(\arg(z)-\phi)}E_1(z/wu)  e_2 \left(1/w\right) \frac{dw}{w}\right)\frac{du}{u}.$$
\end{proof}

We are ready to prove the main result, essential for the construction of the multisum.

\begin{pro}\label{prop.acceleration.kernel}
Let $\M_j$, $e_j$, $m_{e_j}$, and $T_{e_j},T_{e_j}^{-}$, $j=1,2$, be as in Proposition~\ref{prop.convolution.kernel}. Assume that $\o(\M_1)<\o(\M_2)<2$.

\begin{enumerate}
 \item We define \textit{the acceleration from $e_2$ to $e_1$}, denoted $e_1\triangleleft  e_2$, by
 $$(e_1\triangleleft e_2)(z)=T_{e_1}^{-}(e_2(1/u))(1/z).$$
 Then,  $e_1\triangleleft e_2$ is a   strong kernel of $\M_2/\M_1-$ summability whose moment
function is $m(\lambda)=m_{e_2}(\lambda)/m_{e_1}(\lambda)$.
Moreover, if $E_2$ and $E$ are the functions associated by
Definition~\ref{NEWdefikernelMsumm}.(\textsc{v}) with $e_2$ and $e_1\triangleleft e_2$, respectively,  we have that
$$E(u)=T_{e_1} (E_2(u)).$$

\item  The function $e_1\triangleleft e_2$ is the unique moment summability kernel with moment sequence $(m(p)=m_{e_2}(p)/m_{e_1}(p))_{p\in\N_0}$.

\item Let $A_{e_1,e_2}$ denote the Laplace-like integral operator associated with
$e_1\triangleleft e_2$. If $S$ is an unbounded sector and $f\in\mathcal{O}^{\M_2}(S)$, then
$$A_{e_1,e_2} f=T^{-}_{e_1}\circ T_{e_2}(f).$$

\item  We define $g(z):= ((T^{-}_1 \circ T_{e_2} )f)(z)$  with $f(u)= (1-u)^{-1}$ and
$$e(1/z):=\frac{g(z)-g(ze^{2\pi i})}{2\pi i}.$$
Then, $e$ is well defined in $S_{\o(\M_2) -\o(\M_1)}$ and $e(z)=e_1\triangleleft e_2 (z)$.

\end{enumerate}
\end{pro}

\begin{proof}

 For simplicity we write $\o_1=\o(\M_1)$ and $\o_2=\o(\M_2)$. We observe that $\o( \M_2/\M_1) = \o_2-\o_1$
(see Remark~\ref{rema.omega.index.quotient.sequence}).

\begin{enumerate}

\item  Let us show that the function $(e_1\triangleleft e_2)(z)=T_{e_1}^{-}(e_2(1/u))(1/z)$ has the necessary properties for a strong kernel function of $\M_2/\M_1-$summability, as listed in Definition~\ref{NEWdefikernelMsumm}.\para

Since the function $e_2(1/u)$ is holomorphic in $S_{\o_2}$, continuous at the origin and $\o_2>\o_1$,
by Proposition~\ref{prop.hol.Bor} we have that   $T_{e_1}^{-}(e_2(1/u)) (z) $ is holomorphic in $S_{\o_2-\o_1}$ which proves requirement (\textsc{i}).\para

Regarding the integrability condition (\textsc{ii.b}), we fix $\tau\in(0,\o_2-\o_1)$ and we take
$\tau_1\in(0,2-\o_1)$ and $\tau_2\in(0,\o_2)$ such that $2+\tau\in (0,\tau_1+\tau_2-(2-\o_1-\tau_1))$ and we can choose
$\ep\in(2-\o_1-\tau_1,\tau_1+\tau_2-2-\tau)$.
By Definition~\ref{NEWdefikernelMsumm}.(\textsc{vi.b}) for $E_1$ and  (\textsc{ii.b}) for $e_2$ we know that there exist $\b_1,\a_2>0$ (not depending on $\tau_1$ and $\tau_2$), and contants $K_1,C_2>0$, $M_1\geq1$, $\ep_2\in(0,1)$ such that
\begin{align}
 |E_1(z)|\leq& \frac{K_1}{ |z|^{\b_1}}, \qquad z\in S(\pi,{\tau_1}), \qquad |z|\geq M_1, \label{ineq.Condition.(VI.B).E1}\\
 |e_2(z)|\leq& C_2 |z|^{\a_2}, \qquad z\in S_{\tau_2}, \qquad |z|\leq \ep_2. \label{ineq.Condition.(II.B).e2}
\end{align}
We fix $z\in S_\tau$ with $|z|\leq\ep^2_2/ M^2_1$, then $|z|\leq 1$. We write $\phi=\arg(z)\in(-\pi\tau/2,\pi\tau/2)$ and we may consider a path
$\delta_{\o_1}(-\phi)$ (see Definition~\ref{def.e.Borel}). We can write  $\delta_{\o_1}(-\phi)=\delta_1+\delta_2+\delta_3$ ($\delta_1$ and $\delta_3$ are segments in directions $\theta_1=-\phi+(\pi/2)(\o_1+\ep)$ and $\theta_3=-\phi-(\pi/2)(\o_1+\ep)$, respectively, and $\delta_2$ is a circular arc with radius $R=1/(|z|M_1)$, that can be chosen in this way because $e_2$ is holomorphic in $S_{\o_2}$, that is unbounded, and the path
$\delta_{\o_1}(-\phi)$ stays inside $S_{\tau_2}$). Then,  by definition, we see that
$$
e_1\triangleleft e_2 (z)=\frac{-1}{2\pi i}\int_{\delta_{\o_1}(-\phi)} E_{1} \left(\frac{1}{uz}\right) e_{2} \left(1/u\right)\frac{du}{u}.
$$
We have that
\begin{align}
 |e_1\triangleleft e_2 (z) | \leq& \frac{1}{2\pi} \left( \int_{0}^{1/(|z|M_1)} \left|E_{1} \left(\frac{1}{zre^{i\theta_1}}\right) e_{2} \left(\frac{1}{re^{i\theta_1}}\right)\right|\frac{dr}{r}\right.\no\\ +& \int_{\theta_3}^{\theta_1} \left|E_{1} \left(\frac{|z|M_1}{ze^{i\theta}}\right) e_{2} \left(\frac{|z|M_1}{e^{i\theta}}\right)\right| d\theta\no \\
+& \left. \int_{0}^{1/(|z|M_1)} \left|E_{1} \left(\frac{1}{zre^{i\theta_3}}\right)  e_{2} \left(\frac{1}{re^{i\theta_3}}\right)\right| \frac{dr}{r} \right).\label{ineq.e1.tring.e2.rays.arc.VIB.condition}
\end{align}

First, we study the second integral in~\eqref{ineq.e1.tring.e2.rays.arc.VIB.condition}.
By condition (\textsc{v}) for $E_1$, we know that there exists a constant $H_1$ such that
$|E_1(w)|\leq H_1$ for every $w\in D(0,M_1+1)$.
Using the bounds for $\phi$, we see that
$$ [\theta_3,\theta_1]\en ((\pi/2)(-\tau-\o_1-\ep) , (\pi/2) (\tau+\o_1+\ep)).$$
Employing the upper bound for $\ep$, we obtain that  $\tau+\o_1+\ep<\tau_2-(2-\o_1-\tau_1)$, then
we deduce that
\begin{equation}\label{arg.of.theta.in.S.tau2}
|z|M_1e^{-i\theta}\in S_{\tau_2}, \qquad \theta\in [\theta_3,\theta_1].
\end{equation}
Since $|z|\leq |z|^{1/2}\leq \ep_2/M_1$, we have that $|ze^{-i\theta}M_1|\leq \ep_2$ and, by~\eqref{ineq.Condition.(II.B).e2}, we conclude that
\begin{equation}\label{ineq.e1.tring.e2.arc}
\int_{\theta_3}^{\theta_1} \left|E_{1} \left(\frac{|z|M_1}{ze^{i\theta}}\right) e_{2} \left(\frac{|z|M_1}{e^{i\theta}}\right)\right| d\theta \leq
H_1 C_2 |z|^{\a_2} M^{\a_2}_1 \pi\tau_2.
\end{equation}

We study now the first and the last integrals in~\eqref{ineq.e1.tring.e2.rays.arc.VIB.condition}.  If $r\in(0,1/(|z|M_1))$, we observe that $|1/(zre^{i\theta_j})|\geq M_1$.   Since $\o_1+\tau_1<2$ and $\tau_2<2$, using the bounds for $\ep$, we have that
$$\o_1+\ep \in (2-\tau_1,\tau_2+\o_1+\tau_1-2-\tau  )\en (2-\tau_1, 2 ) $$
and we deduce that $\arg(1/(zre^{i\theta_3}))= (\pi/2)(\o_1+\ep)\in ((\pi/2)(2-\tau_1), \pi)$ and also that	 $\arg(1/(zre^{i\theta_1})) \in (-\pi, -(\pi/2)(2-\tau_1))$.
Then  for $j=1,3$ we  see that $1/(zre^{i\theta_j})\in S(\pi,\tau_1)$ and, by~\eqref{ineq.Condition.(VI.B).E1},  we show that
\begin{equation*}
\int_{0}^{(|z|M_1)^{-1}}  \left|E_{1} \left(\frac{1}{zre^{i\theta_j}}\right) e_{2} \left(\frac{1}{re^{i\theta_j}}\right)\right| \frac{dr}{r}\leq
K_1 |z|^{\b_1}\int_{0}^{(|z|M_1)^{-1}}r^{\b_1} \left| e_{2} \left(\frac{1}{re^{i\theta_j}}\right)\right| \frac{dr}{r}.
\end{equation*}
Since $|z|\leq 1$, we can split $(0,1/(|z|M_1))$ as the union of the intervals $(0,1/(|z|^{1/2}M_1))$ and $[1/(|z|^{1/2}M_1),1/(|z|M_1))$.
By condition (\textsc{iii}) for $e_2$, we know that there exists a constant $D_2$ such that
$|e_2(w)|\leq D_2$ for every $w\in S_{\tau_2}$.
If $r\geq1/ (|z|^{1/2}M_1)$, then $1/r\leq \ep_2 $, by~\eqref{arg.of.theta.in.S.tau2}, we can apply~\eqref{ineq.Condition.(II.B).e2} and we obtain
 $$ \int_{0}^{(|z|M_1)^{-1}}  \left|E_{1} \left(\frac{1}{zre^{i\theta_j}}\right) e_{2} \left(\frac{1}{re^{i\theta_j}}\right)\right| \frac{dr}{r}\leq
\frac{K_1 |z|^{\b_1/2} D_2 }{M_1^{\b_1}\b_1}+ K_1 C_2 |z|^{\b_1}\int_{(|z|^{1/2}M_1)^{-1}}^{(|z|M_1)^{-1}} r^{\b_1-\a_2} \frac{dr}{r}.$$
According to Remark~\ref{rema.change.alpha.beta},  we may assume that $\a_2<\b_1$.
Then for $j=1,3$ and  for every $u\in S_\tau$ with $|z|\leq\ep^2_2/ M^2_1$ we have that
\begin{equation}\label{ineq.e1.tring.e2.rays}
\int_{0}^{(|z|M_1)^{-1}}  \left|E_{1} \left(\frac{1}{zre^{i\theta_j}}\right) e_{2} \left(\frac{1}{re^{i\theta_j}}\right)\right| \frac{dr}{r} \leq
\frac{K_1  D_2|z|^{\b_1/2}  }{M_1^{\b_1}\b_1}+ \frac{K_1C_2 |z|^{\a_2}}{ (\b_1-\a_2) M_1^{\b_1-\a_2}}.
\end{equation}
Consequently, by~\eqref{ineq.e1.tring.e2.arc} and~\eqref{ineq.e1.tring.e2.rays}, condition (\textsc{ii.b}) is satisfied
 with $\a=\min(\b_1/2,\a_2)$.\para

By condition (\textsc{iii}) for $e_2$, for every $\ep>0$ there exist $c,k>0$ such that
$$|e_2(1/u)| \le  c\,e^{-\o_{\M_2}(k/|u|)},\qquad u\in S_{\o_2-\ep}.$$
Then, by Proposition~\ref{prop.thilliez.flat.function.M}, we have that $e_2(1/u)\sim_{\M_2} \widehat{0}$ in $S_{\o_2}$.

By Theorem~\ref{teorrelacdesartransfBL}, we see that $e_1 \triangleleft e_2(1/z)=T^{-}_{e_1}(e_2(1/u))(z)\sim_{\M_2/\M_1} \widehat{0}$ in  $S_{\o_2-\o_1}$ which implies, again by Proposition~\ref{prop.thilliez.flat.function.M}, that for every $\ep>0$ there exist $c,k,r>0$ such that
\begin{equation}\label{ineq.condition.III.e1.triang.e2}
|e_1 \triangleleft e_2 (z)|\le c\,e^{-\o_{\M_2/\M_1}(|z|/k)},\qquad z\in S_{\o_2-\o_1-\varepsilon},\qquad |z|>r.
\end{equation}
By condition (\textsc{ii.b}) for $e_1 \triangleleft e_2$, we know that   $|e_1 \triangleleft e_2(z)|\leq C$ for $z\in S_{\o_2-\o_1-\ep}$ with $|z|\leq \de $. Since $e_1 \triangleleft e_2(z)$ is continuous, \eqref{ineq.condition.III.e1.triang.e2} holds for every $z\in S_{\o_2-\o_1-\ep}$ and  we conclude that condition (\textsc{iii}) is satisfied.\para

For $x>0$ we have that
\begin{align*}
\overline{e_1\triangleleft e_2(x)}=&
\frac{1}{2\pi i} \left( \int_{0}^{R} \overline{ E_{1} \left(\frac{1}{xre^{i\theta_1}}\right) e_{2} \left(\frac{1}{re^{i\theta_1}}\right) } \frac{dr}{r} \right. \\
 -&\int_{\theta_1}^{\theta_3} \overline{ E_{1} \left(\frac{1}{xRe^{i\theta}}\right) e_{2} \left(\frac{1}{Re^{i\theta}}\right)} i d\theta
+\left.  \int_{R}^{0} \overline{ E_{1} \left(\frac{1}{xre^{i\theta_3}}\right) e_{2} \left(\frac{1}{re^{i\theta_3}}\right)}  \frac{dr}{r} \right)
\end{align*}
with $\theta_1=(\pi/2)(\o_1+\ep)$ and $\theta_3=-(\pi/2)(\o_1+\ep)$.
Since $E_1$ and $e_2$ are positive real over the positive real axis and holomorphic in $S_{\o_2}$, we deduce that
\begin{align*}
\overline{e_1\triangleleft e_2(x)}=&
\frac{1}{2\pi i} \left( \int_{0}^{R} E_{1} \left(\frac{1}{xre^{-i\theta_1}}\right) e_{2} \left(\frac{1}{re^{-i\theta_1}}\right)  \frac{dr}{r} \right. \\
 -&\int_{\theta_1}^{\theta_3}  E_{1} \left(\frac{1}{xRe^{-i\theta}}\right) e_{2} \left(\frac{1}{Re^{-i\theta}}\right) i d\theta
+\left.  \int_{R}^{0} E_{1} \left(\frac{1}{xre^{-i\theta_3}}\right) e_{2} \left(\frac{1}{re^{-i\theta_3}}\right)  \frac{dr}{r} \right).
\end{align*}
Since $\theta_1=-\theta_3$, we observe that $\overline{e_1\triangleleft e_2(x)}=e_1\triangleleft e_2(x)$, then (\textsc{iv}) holds.\para

Let us show that
\begin{equation}\label{eq.moment.equality.quotient}
m_{e_1\triangleleft e_2} (\lambda)=m_{e_2}(\lambda)/m_{e_1}(\lambda)
\end{equation}
for $\Re(\lambda)\geq 0$. We have that
\begin{align*}
 m_{e_1\triangleleft e_2}(\lambda)=&\int_{0}^\oo x^{\lambda-1} (e_1\triangleleft e_2) (x) dx  \\
= & \frac{-1}{2\pi i} \int_{0}^\oo x^{\lambda-1}  \int_{\delta_{\o_1}(0)} E_{1} \left(\frac{1}{ux}\right) e_{2} \left(1/u\right)\frac{du}{u} dx.
\end{align*}
We make the change of variables $t=1/(xu)$. Then the path $\delta_{\o_1}(0)$ (with radius $R=1/x$, that can be chosen in this way because $e_2$ is holomorphic in $S_{\o_2}$) stays inside $S_{\o_2}$ and it is transformed into $\De_{\o_1}(0)$, with  $\De_{\o_1}(0)=\De_3+\De_2+\De_1$ ($\De_3$ is the line in direction $\theta_3=-(\pi/2)(\o_1+\ep)$ from infinity to $e^{i\theta_3}$, $\De_2$ is a circular arc with radius $R=1$,  and $\De_1$ is the line from $e^{i\theta_1}$ in direction $\theta_1=(\pi/2)(\o_1+\ep)$ to infinity)
and we get
\begin{align*}
 m_{e_1\triangleleft e_2}(\lambda)=
 & \frac{1}{2\pi i} \int_{0}^\oo x^{\lambda}  \int_{\De_{\o_1}(0)} E_{1} \left(t\right) e_{2} \left(xt\right)\frac{dt}{t} \frac{dx}{x}.
\end{align*}
We make the change of variables $xt=s$, we see that
\begin{align*}
 m_{e_1\triangleleft e_2}(\lambda)=
 & \frac{1}{2\pi i} \int_{0}^\oo s^{\lambda} e_{2} \left(s\right) \int_{\De_{\o_1}(0)} E_{1} \left(t\right) \frac{dt}{t^{\lambda+1}} \frac{ds}{s}.
\end{align*}
Using condition (\textsc{vi.b}) for $E_1$ and Cauchy's theorem to deform the path of
integration and replace $\De_{\o_1}(0)$ by the disc $D(0,1)$, we obtain
\begin{align*}
 m_{e_1\triangleleft e_2}(\lambda)=
 &  \int_{0}^\oo s^{\lambda} e_{2} \left(s\right) \left( \frac{1}{2\pi i}\int_{D(0,1)} E_{1} \left(t\right) \frac{dt}{t^{\lambda+1}} \right)\frac{ds}{s}.
\end{align*}
By Cauchy's formula for $E_1$, we see that
\begin{align*}
 m_{e_1\triangleleft e_2}(\lambda)= \int_{0}^\oo s^{\lambda} e_{2} \left(s\right) \frac{1}{m_{e_1}(\lambda)}\frac{ds}{s}.
\end{align*}
Finally, by condition (\textsc{v}) for $e_2$ we show that \eqref{eq.moment.equality.quotient} is satisfied.


We deduce that $m_{e_1\triangleleft e_2}(\lambda)$ is continuous in $\{\Re(\lambda)\geq0\}$, holomorphic in $\{\Re(\lambda)>0\}$ and  $m_{e_1\triangleleft e_2}(x)>0$ for every $x\geq 0$.
We define the function $E_{\triangleleft}$ by
$$E_{\triangleleft}(z):= \sum^\oo_{p=0} \frac{z^n}{m_{e_1\triangleleft e_2}(p)}, \qquad z\in\C.$$
If we compute the radius of convergence of this series, using \eqref{eq.moment.equality.quotient} and that $\o_1<\o_2$
(see Proposition~\ref{prop.quotient.sequence.compara.cond}.(iii))  we show that
$$r=\liminf_{p\to\oo} \sqrt[n]{m_{e_1\triangleleft e_2}(p)} = \liminf_{p\to\oo} \sqrt[p]{\frac{m_{e_2}(p)}{m_{e_1}(p)}} =\oo, $$
hence $E_{\triangleleft}$ is entire. We see that $(m_{e_1\triangleleft e_2}(p))_{p\in\N_0}$, again by~\eqref{eq.moment.equality.quotient}, is equivalent to the sequence $\M_2/\M_1$. Then, by Proposition~\ref{propKomatsu}, we see that there exist $C,K>0$ such that
$|E_{\triangleleft}(z)|\le C\exp(\o_{\M_2/\M_1}(K|z|))$, for all $z\in\C$, then (\textsc{v}) holds.

Finally, regarding condition (\textsc{vi.b}),  we need first to show that
\begin{equation}\label{eq.E.Laplace1.E2}
E_{\triangleleft}(u)=(T_{e_1} E_2)(u), \qquad u\in \C^{*}.
\end{equation}
We prove this equality for $u\in (0,\oo)$, and we conclude using the identity principle since $E_{\triangleleft}(u)$ is entire and, by Proposition~\ref{prop.hol.lap}, $(T_{e_1} E_2)(u)$  is holomorphic in a sectorial region $G(0,2+\o_1)$.

We fix $u\in(0,\oo)$,  we have that
$$(T_{e_1} E_2)(u) =\int^\oo_{0} e_{1} \left(\frac{z}{u}\right) E_2(z) \frac{dz}{z}.$$
Since $e_1$ and $E_2$ are positive over $(0,\oo)$, then we can exchange integral and sum and applying~\eqref{eq.moment.equality.quotient} we see that
\begin{align*}
 (T_{e_1} E_2)(u)=& \int^\oo_{0} e_{1} \left(\frac{z}{u}\right) E_2(z) \frac{dz}{z}
= \int^\oo_{0} e_{1} \left(\frac{z}{u}\right) \sum_{n=0}^\oo \frac{z^n}{m_{e_2}(n)} \frac{dz}{z} \\
=& \sum_{n=0}^\oo (T_{e_1} (z^n/m_{e_2}(n)))(u)= \sum_{n=0}^\oo \frac{u^n}{m_{e_1\triangleleft e_2}(n)}=E_{\triangleleft}(u).
\end{align*}

Now, we fix $\tau\in(0,2-\o_2+\o_1)$, and we take $\tau_1\in(0,\o_1)$ and $\tau_2\in(0,2-\o_2)$ such that
$\tau_2<\tau<\tau_1+\tau_2$.
By Definition~\ref{NEWdefikernelMsumm}.(\textsc{ii.b}) for $e_1$ and  (\textsc{vi.b}) for $E_2$ we know that there exist $\a_1,\b_2>0$ (not depending on $\tau_1$ and $\tau_2$), and constants $C_1,K_2>0$, $\ep_1\in(0,1)$, $M_2\geq1$ such that
\begin{align}
 |e_1(z)|\leq& C_1 |z|^{\a_1}, \qquad z\in S_{\tau_1}, \qquad |z|\leq \ep_1, \label{ineq.Condition.(II.B).e1}\\
 |E_2(z)|\leq&\frac{K_2}{ |z|^{\b_2}} , \qquad z\in S(\pi,{\tau_2}), \qquad |z|\geq M_2. \label{ineq.Condition.(VI.B).E2}
\end{align}
We fix $u\in S(\pi,\tau)$ with $|u|\geq M_2/\ep_1$. If $u\in S(\pi,\tau_2)$, we define $\theta_u:=\arg(u)$ so we have that
\begin{equation}\label{eq.def.domains.T1E2}
 \frac{e^{i\theta_u}}{u} \in S_{\tau_1}, \qquad e^{i\theta_u}\in S(\pi,\tau_2).
\end{equation}
If $\arg(u)\in ((\pi/2)(2-\tau), (\pi/2)(2-\tau_2)]$, we define $\theta_u:=\arg(u)+\ep_u$ with $$\ep_u\in ((\pi/2) (2-\tau_2)-\arg(u),\min((\pi/2) (2+\tau_2)-\arg(u) ,(\pi/2)\tau_1)).$$
This interval is not empty since $\arg(u)(2/\pi)>2-\tau>2-\tau_2-\tau_1$, then $(\pi/2) (2-\tau_2)-\arg(u)<\tau_1\pi/2$.
We observe that $\arg(e^{i\theta_u}/u)= \ep_u \in [0, (\pi/2)\tau_1)$
and we also have that $ e^{i\theta_u}\in S(\pi,\tau_2)$ and we deduce~\eqref{eq.def.domains.T1E2}.

Analogously, if $\arg(u)\in [(\pi/2)(2+\tau_2), (\pi/2)(2+\tau))$ we choose $\theta_u:=\arg(u)-\ep_u$ with
 $$\ep_u\in (-(\pi/2) (2+\tau_2)+\arg(u),\min(-(\pi/2) (2-\tau_2)+\arg(u) ,(\pi/2)\tau_1)),$$
and we also obtain~\eqref{eq.def.domains.T1E2} for this choice of $\theta_u$.
By \eqref{eq.E.Laplace1.E2}, since $|u|\geq M_2/\ep_1$  we have that
\begin{align*}
 |E_{\triangleleft}(u) | =& \left|\int_{0}^{\oo(\theta_u)} e_{1} \left(\frac{z}{u}\right) E_2(z) \frac{dz}{z}  \right|
  \leq \int_{0}^{M_2}\left| e_{1} \left(\frac{re^{i\theta_u}}{u}\right) E_2(re^{i\theta_u})\right| \frac{dr}{r}\\
  +& \int_{M_2}^{|u|\ep_1}\left| e_{1} \left(\frac{re^{i\theta_u}}{u}\right) E_2(re^{i\theta_u})\right| \frac{dr}{r} +
  \int_{|u|\ep_1}^{\oo} \left| e_{1} \left(\frac{re^{i\theta_u}}{u}\right) E_2(re^{i\theta_u})\right| \frac{dr}{r}.
\end{align*}
 If $r\leq |u|\ep_1$, we have that $|re^{i\theta_u}/u|\leq \ep_1$, by~\eqref{eq.def.domains.T1E2} we can apply \eqref{ineq.Condition.(II.B).e1} and we obtain
\begin{align*}
 |E_{\triangleleft}(u) |
  \leq&\frac{C_1}{|u|^{\a_1}}\left( \int_{0}^{M_2} r^{\a_1} \left| E_2(re^{i\theta_u})\right| \frac{dr}{r}
  + \int_{M_2}^{|u|\ep_1} r^{\a_1} \left| E_2(re^{i\theta_u})\right| \frac{dr}{r} \right)\\
  +&\int_{|u|\ep_1}^{\oo} \left| e_{1} \left(\frac{re^{i\theta_u}}{u}\right) E_2(re^{i\theta_u})\right| \frac{dr}{r}.
\end{align*}
By condition (\textsc{iii}) for $e_1$, and (\textsc{v}) for $E_2$  we know that there exist constants $D_1,H_2$ such that
$|e_1(w)|\leq D_1$ for every $w\in S_{\tau_1}$ and $|E_2(w)|\leq H_2$ for every $w\in D(0,M_2+1)$. We deduce that
\begin{align*}
 |E_{\triangleleft}(u) |
  \leq&\frac{C_1 H_2 M_2^{\a_1}}{\a_1|u|^{\a_1}}
  + \frac{C_1 }{|u|^{\a_1}}  \int_{M_2}^{|u|\ep_1} r^{\a_1} \left| E_2(re^{i\theta_u})\right| \frac{dr}{r} + D_1 \int_{|u|\ep_1}^{\oo} \left|  E_2(re^{i\theta_u})\right| \frac{dr}{r}.
\end{align*}
If $r\geq M_2$, by~\eqref{eq.def.domains.T1E2} we can apply~\eqref{ineq.Condition.(VI.B).E2} and we have
\begin{align*}
 |E_{\triangleleft}(u) |
  \leq&\frac{C_1 H_2 M_2^{\a_1}}{\a_1|u|^{\a_1}}
  + \frac{C_1 K_2 }{|u|^{\a_1}}  \int_{M_2}^{|u|\ep_1} r^{\a_1-\b_2}  \frac{dr}{r} +
   \frac{D_1 K_2}{\b_2|u|^{\b_2} \ep_1^{\b_2}}.
\end{align*}
According to Remark~\ref{rema.change.alpha.beta},  we may assume that $\a_1>\b_2$, then
\begin{align*}
 |E_{\triangleleft}(u) |
  \leq&\frac{C_1 H_2 M_2^{\a_1}}{\a_1|u|^{\a_1}}
  + \frac{C_1 K_2 \ep_1^{\a_1-\b_2}}{(\a_1-\b_2)|u|^{\b_2}}   +
    \frac{D_1K_2}{\b_2|u|^{\b_2} \ep_1^{\b_2}}.
\end{align*}
Consequently, condition (\textsc{vi.b}) is satisfied with $\b=\min(\a_1,\b_2)$.

\item  Uniqueness follows from Remark~\ref{rema.unique.e}.

\item We take $f\in\mathcal{O}^{\M_2}(S(d,\a))\en \mathcal{O}^{\M_2/\M_1}(S(d,\a))$. Since $e_1\triangleleft  e_2$ is a kernel of $\M_2/\M_1-$ summability,  by Proposition~\ref{prop.hol.lap} we know that $A_{e_1,e_2}(f)$ is holomorphic in
a sectorial region $G(d,\a+\o_2-\o_1)$. We fix  $z\in G(d,\a+\o_2-\o_1)$,  there exists $\phi$ with $|d-\phi|<\pi\a/2$,
such that if $\arg(u)=\phi$, then $u/z\in S_{\o_2-\o_1}$.  We write $\eta=\arg(z/u)=\arg(z)-\phi$  and we consider a path $\delta_{\o_1}(\eta)$ chosen as in Proposition~\ref{prop.doubleintegral.mod.B1L2}, what is possible because $e_2$ is holomorphic in $S_{\o_2}$ and the path $\delta_{\o_1}(\eta)$ stays inside $S_{\o_2}$. We have that
$$A_{e_1,e_2} (f) (z) =
\int_{0}^{\oo(\phi)}  f(u) \left(\frac{-1}{2\pi i}\int_{\delta_{\o_1} (\eta)}   E_1\left(\frac{z}{wu}\right) e_2 (1/w)  \frac{dw}{w}\right) \frac{du}{u}.$$
We conclude, using Proposition~\ref{prop.doubleintegral.mod.B1L2}, that this last expression equals $T^{-}_{e_1}\circ T_{e_2} (f) (z)$.

\item We consider $f(u)= (1-u)^{-1}$ and we define $g(z):= ((T^{-}_{e_1} \circ T_{e_2} )f)(z)$. By Lemma~\ref{lemma.e.Laplace.geom.series},  we know that $T_{e_2} f$ is holomorphic in $S(\pi,2+\o_2)$. Moreover,  $T_{e_2} f\sim_{\M_2} \sum_{n=0}^\oo m_2(n) z^n$  on $S(\pi,2+\o_2)$, then it is continuous at the origin.
We can apply  the $T^{-}_{e_1}$ transform  to $T_{e_2} f$ and, by Proposition~\ref{prop.hol.Bor}, we have that $g$ is holomorphic in $S=S(\pi,2+\o_2-\o_1)$. Then, the function
\begin{equation*}
e(1/z):=\frac{g(z)-g(ze^{2\pi i})}{2\pi i},
\end{equation*}
is holomorphic in  $S_{\o_2-\o_1}$. Since $f\in \mathcal{O}^{\M_2}(S_{\o_2-\o_1})$,  by statement 3 we have
$A_{e_1,e_2}(f)= (T^{-}_{e_1}\circ T_{e_2})(f)$ and, by Lemma~\ref{lemma.e.Laplace.geom.series},  we deduce that
\begin{equation*}
e(1/z)=\frac{g(z)-g(ze^{2\pi i})}{2\pi i} = \frac{ A_{e_1,e_2}(f)(z)-A_{e_1,e_2}(f)(ze^{2\pi i})}{2\pi i}= e_1\triangleleft e_2(1/z).
\end{equation*}

\end{enumerate}

\end{proof}

\begin{rema} Under the hypotheses of Proposition~\ref{prop.acceleration.kernel}, we know that
$\M_2/\M_1$ is equivalent to a weight sequence admitting nonzero proximate order (see Proposition~\ref{prop.quotient.seq}). As indicated in Remark~\ref{rema.non.stab.asymp.product}, a strong kernel of $\M_2/\M_1-$summability can be constructed, but it might not behave well with respect to asymptotic relations. The kernel $e_1\triangleleft e_2$ will be suitable in this regard.

Note that $\mathcal{O}^{\M_2}(S)\subseteq \mathcal{O}^{\M_2/\M_1}(S)$.
Consequently, $A_{e_1,e_2}$  extends the operator $T^{-1}_{e_1} \circ T_{e_2}$. Moreover, by Propositions~\ref{prop.acceleration.kernel}.3 and Proposition~\ref{prop.hol.Bor}, for every $f\in\mathcal{O}^{\M_2}(S(d,\ga))$ we deduce  $A_{e_1,e_2}f\in\mathcal{O}^{\M_1}(S(d,\ga+\o(\M_2)-\o(\M_1)))$, which justifies the name of the operator.
\end{rema}

\begin{rema}
 From the uniqueness of the convolution and the acceleration kernels we deduce some basic properties:
 \begin{center}
  $e_1 \ast e_2=e_2\ast e_1, \qquad e_1 \ast (e_1\triangleleft e_2)= e_2, \qquad e_1 \triangleleft  (e_1\ast e_2)= e_2, \qquad e_2 \triangleleft  (e_1\ast e_2)= e_1.$
 \end{center}
\end{rema}

\subsection{Multisummability through acceleration}
\label{subsect.multisum.through.accel}

In order to describe the procedure to recover the multisum of a formal power series presented below, we need to analyze the behavior of asymptotics under the operator $A_{e_1,e_2}$ defined in Proposition~\ref{prop.acceleration.kernel} and to extend what was known for the Gevrey case (see \cite[Th.\ 55 and 56]{Balser2000}).

\begin{theo}\label{th.asymp.acceleration}
Let $\M_j$, $e_j$, $m_{e_j}$,  and $T_{e_j},T_{e_j}^{-}$, $j=1,2$, be as in Proposition~\ref{prop.convolution.kernel}.
Assume  $\o(\M_1)<\o(\M_2)<2$ . Let $A_{e_1,e_2}$ denote the Laplace-like integral operator associated with $e_1\triangleleft e_2$ (see Proposition~\ref{prop.acceleration.kernel}) and $\M'$ be any sequence of positive real numbers. Then,
\begin{enumerate}[(i)]
 \item If $f\in\mathcal{O}^{\M_2/\M_1}(S(d,\a))$
and $f\sim_{\M'}\widehat{f}$, then $A_{e_1,e_2}f\sim_{\M'\cdot(\M_2/\M_1)}\widehat{A}_{e_1,e_2}\widehat{f}$ in a
sectorial region $G(d,\a+\o(\M_2/\M_1))$, where
$$
\widehat{A}_{e_1,e_2}\left(\sum_{p=0}^{\infty}a_pz^p\right):= \sum_{p=0}^{\infty}\frac{a_pm_{e_2}(p)}{m_{e_1}(p)}z^p.
$$
\item If, moreover, $f\in\mathcal{O}^{\M_2}(S(d,\a))$, then $A_{e_1,e_2}f\in\mathcal{O}^{\M_1}(S(d,\a+\o(\M_2/\M_1)))$ and
$$
T_{e_1}(A_{e_1,e_2}f)=T_{e_2}f.
$$
\end{enumerate}

\end{theo}
\begin{proof}

\begin{enumerate}[(i)]
 \item By Proposition~\ref{prop.acceleration.kernel}, $e_1\triangleleft e_2$  is a strong kernel of $\M_2/\M_1-$ summability. Then, the conclusion follows applying Theorem~\ref{teorrelacdesartransfBL}.

 \item By Proposition~\ref{prop.acceleration.kernel}.3, we know that
 $$A_{e_1,e_2}f=(T^{-}_{e_1} \circ T_{e_2}) f.$$
 By Proposition~\ref{prop.hol.lap}, $T_{e_2}f$ is holomorphic in a sectorial region $G(d,\a+\o(\M_2))$. Since $\o(\M_2)>\o(\M_1)$, by Proposition~\ref{prop.hol.Bor}
 $(T^{-}_{e_1} \circ T_{e_2}) f$ is holomorphic in $S=S(d,\a+\o(\M_2)-\o(\M_1))$ and it is of $\M_1-$growth in $S$.
 We observe that $\o( \M_2/\M_1) = \o(\M_2)-\o(\M_1)$ (see Remark~\ref{rema.omega.index.quotient.sequence}), then $A_{e_1,e_2}f= (T^{-}_{e_1} \circ T_{e_2}) f\in\mathcal{O}^{\M_1}(S)$. We can apply $T_{e_1}$ to $A_{e_1,e_2}f$, and we get
$$
T_{e_1}(A_{e_1,e_2}f)= T_{e_1} T^{-}_{e_1} T_{e_2} f =T_{e_2}f.
$$
\end{enumerate}
\end{proof}

In a natural way, we define
$
\widehat{A}^{-}_{e_1,e_2}\left(\sum_{p=0}^{\infty}a_pz^p\right):= \sum_{p=0}^{\infty} (a_pm_{e_1}(p)/m_{e_2}(p))z^p.
$ With the tools presented in the previous subsections and in the conditions of Proposition~\ref{prop.acceleration.kernel}, we are ready for giving a definition of multisummability in a multidirection with respect to the multikernel $(e_1,e_2)$.

\begin{defi}\label{def.multi.sumable.kernel}
In the conditions of Proposition~\ref{prop.acceleration.kernel}, we say that the series $\widehat{f}=\sum_{p\ge 0} a_p z^p$ is \textit{$(e_1,e_2)-$summable in the multidirection $(d_1,d_2)$} with  $|d_1-d_2|<\pi(\o(\M_2)-\o(\M_1))/2$ and $d_1,d_2\in\R$  if:
 \begin{itemize}
 \item[(i)]  $\widehat{g}:=\widehat{T}_{e_1}^{-}\widehat{f}= \displaystyle \sum_{p\ge 0}\frac{a_p}{m_{e_1}(p)}z^p$ is $\M_2/\M_1-$summable in direction $d_2$.
\item[(ii)] The sum $\mathcal{S}_{\M_2/\M_1, d_2} \widehat{g}$ admits analytic continuation $g$ in a sector $S=S(d_1,\ep)$ for some $\ep>0$, and $g\in\mathcal{O}^{\M_1}(S)$.
\end{itemize}
In this situation we can define the corresponding multisum by:
$$\mathcal{S}_{(e_1,e_2),(d_1,d_2)}\widehat{f} := T_{e_1}\circ A_{e_1,e_{2}}\circ \widehat{A}^{-}_{e_1,e_2} \circ  \widehat{T}^{-}_{e_1}\widehat{f}. $$
\end{defi}

The next result states the equivalence between $(\M_1,\M_2)-$multisummability and $(e_1,e_2)-$ multisummability in a multidirection, and provides a way to recover the multisum by means of the formal and analytic acceleration operators previously introduced
(see \cite[Ch.\ 10]{Balser2000} for the Gevrey case).

\begin{theo}\label{th.multisum.construct}
Given  $\M_1,\M_2$  weight sequences admitting a nonzero proximate order with $\o(\M_1)<\o(\M_2)<2$, directions $d_1,d_2\in\R$ with $|d_1-d_2|<\pi(\o(\M_2)-\o(\M_1))/2$ and a formal power series $\widehat f$, the following are equivalent:
\begin{itemize}
\item[(i)] $\widehat{f}\in\C\{z\}_{(\M_1,\M_2),(d_1,d_2)}$.
\item[(ii)] For every pair of strong kernels, $e_1$ of $\M_1-$summability and $e_2$ of $\M_2-$summability, $\widehat f$ is $(e_1,e_2)-$multisummable in multidirection $(d_1,d_2)$.
\item[(iii)] For some pair of strong kernels, $e_1$ of $\M_1-$summability and $e_2$ of $\M_2-$summability, $\widehat f$ is $(e_1,e_2)-$multisummable in multidirection $(d_1,d_2)$.
\end{itemize}
In case any of the previous holds, we deduce that the $(\M_1,\M_2)-$sum of $\widehat{f}$ on the multidirection $(d_1,d_2)$ is given by
$$
\mathcal{S}_{(\M_1,\M_2),(d_1,d_2)}\widehat{f} =T_{e_1}\circ A_{e_1,e_{2}}\circ \widehat{T}^{-}_{e_2}\widehat{f}.
$$
for any pair of kernels $e_1,e_2$.
\end{theo}

\begin{proof}
 For simplicity we write $\o_1=\o(\M_1)$ and $\o_2=\o(\M_2)$. We observe that $\o( \M_2/\M_1) = \o_2-\o_1$
(see Remark~\ref{rema.omega.index.quotient.sequence}).

 (i)$\implies$(ii) With the notation in Definition~\ref{def.multisum}, we write $\widehat{f}=\widehat{f_1}+\widehat{f_2}$. We put $\widehat{g}:=\widehat{T}_{e_1}^{-} \widehat{f} $ and
 we observe that $\widehat{A}_{e_1,e_2}^{-} \widehat{g}= \widehat{A}_{e_1,e_2}^{-} \widehat{T}_{e_1}^{-} \widehat{f}= \widehat{T}_{e_2}^{-} \widehat{f}$. By Theorem~\ref{th.Msummable.equiv.esummable}, we know $\widehat{h}_2:=\widehat{T}_{e_2}^{-} \widehat{f_2}$
 converges in a disc, admits analytic continuation $h_2$ in a sector $S_2=S(d_2,\varepsilon_2)$ for some $\varepsilon_2>0$, and $g_2\in\mathcal{O}^{\M_2}(S_2)$.
 Since $\widehat{f}_1\in\C[[z]]_{\M_1}$, we see that $\widehat{h}_1=\widehat{T}_{e_2}^{-} \widehat{f_1}$ defines an entire function $h_1$ and, by Proposition~\ref{propKomatsu}, we have that $h_1$ is of $\M_2/\M_1-$growth on $S_2$.\para

 Hence, $\widehat{h}:=\widehat{T}_{e_2}^{-} \widehat{f}$ converges in a disc, admits analytic continuation $h$ in the sector $S_2$ where $h$ is of $\M_2/\M_1-$growth there because $\mathcal{O}^{\M_2}(S_2)\en \mathcal{O}^{\M_2/\M_1}(S_2)$. By Theorem~\ref{th.Msummable.equiv.esummable}, this means that the formal power series $\widehat{A}_{e_1,e_2}  \widehat{T}_{e_2}^{-} \widehat{f}= \widehat{T}_{e_1}^{-} \widehat{f} = \widehat{g}$ is $\M_2/\M_1-$summable  in direction $d_2$, so Definition~\ref{def.multi.sumable.kernel}.(i) is valid.\para

On the other hand, we observe that
$$\mathcal{S}_{\M_2/\M_1, d_2}\widehat{g}= A_{e_1,e_2} \widehat{A}_{e_1,e_2}^{-}\widehat{g} =  A_{e_1,e_2} \widehat{A}_{e_1,e_2}^{-} (\widehat{g}_1 +\widehat{g}_2)$$
where $\widehat{g}_1:=\widehat{T}_{e_1}^{-} \widehat{f}_1$ and $\widehat{g}_2:=\widehat{T}_{e_1}^{-} \widehat{f}_2$. Since $\widehat{f}_1\in\C\{z\}_{\M_1,d_1}$, we have that $\widehat{g_1}=\widehat{T}_{e_1}^{-} \widehat{f}_1$  converges in a disc, admits analytic continuation $g_1$ in the sector $S_1=(d_1,\ep_1)$ for some $\varepsilon_1\in(0,\ep_2)$ and $g_1\in\mathcal{O}^{\M_1}(S_1)$. Moreover, thanks to the convergence, $A_{e_1,e_2}\widehat{A}_{e_1,e_2}^{-} \widehat{g}_1 =  \mathcal{S} \widehat{g}_1$. Regarding $\widehat{g}_2$, we observe that
$$A_{e_1,e_2} \widehat{A}_{e_1,e_2}^{-} \widehat{g}_2= A_{e_1,e_2} \widehat{A}_{e_1,e_2}^{-} \widehat{T}_{e_1}^{-} \widehat{f}_2= A_{e_1,e_2} \widehat{T}_{e_2}^{-} \widehat{f}_2.$$
Since $\widehat{h}_2=\widehat{T}_{e_2}^{-} \widehat{f_2}$ converges in a disc and admits analytic continuation $h_2\in\mathcal{O}^{\M_2}(S_2)$, by Theorem~\ref{th.asymp.acceleration}, we see that
$A_{e_1,e_2} h_2=T^{-}_{e_1} T_{e_2}  h_2$. Furthermore,  $T^{-}_{e_1} T_{e_2}  h_2$ is holomorphic in the sector $S(d_2,\o_2-\o_1+\ep_2)$, which contains the sector $S_1$ because $|d_1-d_2|<\pi(\o_2-\o_1)/2$, and $T^{-}_{e_1} T_{e_2}  h_2 \in \mathcal{O}^{\M_1}(S_1)$, so
$\mathcal{S}_{\M_2/\M_1, d_2}\widehat{g}$ can be written as the sum of two functions $A_{e_1,e_2}\widehat{h}_2$ and $\mathcal{S} \widehat{g}_1$ whose analytic continuations  in $S_1$, $T^{-}_{e_1} T_{e_2}  h_2$ and $g_1$, have $\M_1-$growth there, that is, Definition~\ref{def.multi.sumable.kernel}.(ii) holds.

\par\noindent
(ii)$\implies$(iii) Trivial.\par
\noindent
(iii)$\implies$(i) By Definition~\ref{def.multi.sumable.kernel}.(i),   $g=\mathcal{S}_{\M_2/\M_1, d_2} \widehat{g}$ is holomorphic in a sectorial region $G=G(d_2, \a)$ with $\a>\o_2-\o_1$ and $g\sim_{\M_2/\M_1} \widehat{g}$ in $G$. Let $T$ be a subsector of $G$, bisected by direction $d_2$ and of opening $\pi\b$ with $\b\in(\o_2-\o_1, 2)$, such that $\overline{T}\en G$ and let $\ga$ denote the positively oriented boundary of $\overline{T}$. Decomposing $\ga=\ga_1+\ga_2$ where $\ga_1$ is the circular part and $\ga_2$ is the radial part, we define
$$g_{j}(z):=\frac{1}{2\pi i} \int_{\ga_j} \frac{g(w)}{w-z} dw,\qquad \text{for all} \,\,\, z\in T, \quad j=1,2.$$
Since $g$ is continuous at the origin, by Cauchy's Formula, we can write $g=g_1+g_2$. By Leibniz's rule we see that $g_1$ is holomorphic at the origin. Hence, $g_2=g-g_1 \sim_{\M_2/\M_1} \widehat{g}_2 $ where $\widehat{g}_2:=\widehat{g}-\widehat{g_1}$ and $\widehat{g}_1$ is the Taylor series of $g_1$ at the origin.\para

We define $\widehat{f}_1:=\widehat{T}_{e_1} \widehat{g}_1$ and we immediately observe that $\widehat{f}_1 \in \C[[z]]_{\M_1}$. By (ii) in Definition~\ref{def.multi.sumable.kernel}, $g$ admits analytic continuation in a sector $S_1=S(d_1,\varepsilon)$ for some $\varepsilon>0$, and this analytic continuation has $\M_1-$growth there. Again by the Leibniz's rule, we can see that $g_2$ is holomorphic in $S(d_2,\b)$ and we can prove that tends to $0$ as $|z|\to\oo$ therein, so $g_2\in \mathcal{O}^{\M_1}(S(d_2,\b))$.  Since $|d_1-d_2|<\pi(\o_2-\o_1)/2$, we may assume, by suitably reducing $\ep$, that $S_1\en S(d_2,\b)$. Hence, $g_1=g-g_2$ has an analytic continuation to $S_1$ and has $\M_1-$growth there, this means by Theorem~\ref{th.Msummable.equiv.esummable}  that $\widehat{f}_1$ is $\M_1-$summable in direction $d_1$. \para

Now, we consider $\widehat{f}_2:=\widehat{T}_{e_1}\widehat{g}_2$, we can apply Theorem~\ref{teorrelacdesartransfBL}.(i) to $g_2$ and we deduce that $T_{e_1} g_2\sim_{\M_2} \widehat{f}_2$
on a sectorial region $G(d_2,\b+\o_1)$. Since $\b+\o_1>\o_2$, this means that $\widehat{f}_2$ is $\M_2-$summable in direction $d_2$. Consequently, we can write $\widehat{f}=\widehat{T}_{e_1} \widehat{g}=\widehat{f_1}+\widehat{f_2}$, so $\widehat{f}\in\C\{z\}_{(\M_1,\M_2),(d_1,d_2)}$.\para

\noindent
In case any of the previous equivalent conditions holds, we have seen that
by Theorem~\ref{th.asymp.acceleration},
$$f_2=T_{e_2} \widehat{T}^{-}_{e_2} \widehat{f}_2=T_{e_1} A_{e_1,e_2}  \widehat{T}^{-}_{e_2} \widehat{f}_2,$$
and, thanks to the convergence of $\widehat{T}^{-}_{e_1} \widehat{f}_1$, we have shown that
$$f_1= T_{e_1} \widehat{T}^{-}_{e_1} \widehat{f}_1= T_{e_1} A_{e_1,e_2}\widehat{A}_{e_1,e_2}^{-} \widehat{T}^{-}_{e_1} \widehat{f}_1= T_{e_1} A_{e_1,e_2}  \widehat{T}^{-}_{e_2} \widehat{f}_1.$$
Hence, we conclude that $\mathcal{S}_{(\M_1,\M_2),(d_1,d_2)}\widehat{f}=\mathcal{S}_{(e_1,e_2),(d_1,d_2)}\widehat{f}$.
\end{proof}





\section{Cohomological approach to multisummable power series}
\label{sect.cohomologicalapproach}

Classical multisummability theory can be also stated in a cohomological form. In this general context, this approach is also possible and one can provide a version of the relative Watson’s Lemma (see \cite[Th.\ 7.2.1]{Loday16} for the Gevrey case), which is the cohomological equivalent of the Tauberian Theorem~\ref{th.tauberian.distinct.index}. Apart from the Watson's Lemma and the Borel-Ritt Theorem,
a Ramis-Sibuya-like result given by A. Lastra, S. Malek and J. Sanz in~\cite[Lemma 3]{lastramaleksanz16} is necessary for the proof (for a reference on the classical version of Ramis–Sibuya theorem, the reader may consult~\cite[Th. XI-2-3]{HsiehSibuya1999}).
In this section,
we will follow the discussion
in \cite{Balser1994} and \cite{MR}.

\subsection{Relative Watson's lemma}
\label{subsect.relativeWatson}

We take a strongly regular sequence $\M$ with $0<\o(\M)<2$.
Let
$\uC=\R/2\pi\Z$
be the unit circle and let
$I=(d-\ga\pi/2,d+\ga\pi/2)\ \text{mod}\ 2\pi\Z$
be an open arc of $\uC$.
We set
$\ds
\widetilde{\sA}_{\M}(I)
:=\varinjlim
\widetilde{\cA}_{\M}(G),
$
where $G$ runs over the sectorial regions with
bisecting direction $d\in\R$ and opening $\ga\pi$.
Let $\sA_{\M}$ denote the sheaf on $\uC$
associated with $\widetilde{\sA}_{\M}$.
We naturally have a morphism
$\ds
\cT_{\M}:\sA_{\M}\lto\fS_{\M}
$
and set
$\ds
\sA_{\M}^{<0}:={\rm Ker}\cT_{\M}.
$
By the Borel-Ritt theorem for $\M$-asymptotics
(cf. \cite{thilliez03} and \cite{lastramaleksanz12}),
we have an exact sequence
\begin{equation}\label{es1}
\xymatrix{
0
\ar[r]
&\sA_{\M}^{<0}
\ar[r]
&\sA_{\M}
\ar[r]
&\fS_{\M}
\ar[r]
&0.
}
\end{equation}
%
%
%
%
In what follows,
we consider \v{C}ech cohomology of sheaves.
We first prove the following
\begin{pro}\label{prp:CH}
Let $I$ be an open arc of $\uC$.
Then,
we have
$$
{\rm Im}
\big[
\iota^1:
H^1\big(I;\sA_{\M}^{<0}\,\big)
\lto
H^1\big(I;\sA_{\M})\big]
=0.
$$
\end{pro}
\begin{proof}
By considering a good covering of $I$,
it is reduced to the following
(see the proof of \cite[Prop.~7.24]{vdPS}
for the details):
\begin{eq-text}\label{ref1}
Let
$
I_1
=(a_1,b_1)\ \text{mod}\ 2\pi\Z
$
and
$
I_2
=(a_2,b_2)\ \text{mod}\ 2\pi\Z
$
be open arcs of $\uC$
satisfying
$a_1<a_2<b_1<b_2$
and $b_2-a_1<2\pi$.
Then,
for every $\fy\in\Ga\big(I_1\cap I_2\,;\sA_{\M}^{<0}\,\big)$,
there exist $f_1\in\Ga\big(I_1\,;\sA_{\M}\,\big)$
and $f_2\in\Ga\big(I_2\,;\sA_{\M}\,\big)$
such that $\fy=f_2-f_1$.
\end{eq-text}
The statement \eqref{ref1}
follows from
\cite[Lemma 3]{lastramaleksanz16}.
It concludes the proof.
\end{proof}

Let $J^d_{\M}$ be a closed arc of $\uC$ defined by
\begin{equation}\label{eq.OptimalInterval}
J^d_{\M}
:=
\Big[d-\frac{\o(\M)\pi}{2},d+\frac{\o(\M)\pi}{2}\Big]
\ \text{mod}\ 2\pi\Z.
\end{equation}
By Definition~\ref{def.Msummable.direction},
we find that
\begin{equation*}
\cS_{\M,d}=
{\rm Im}\big[
\cT_{\M}:
\Ga\big(J^d_{\M}\,;\sA_{\M}\,\big)
\lto\fS_{\M}
\big].
\end{equation*}
By Watson's lemma for $\M$-asymptotics
(Theorem~\ref{TheoWatsonlemma}),
we see that
\begin{equation}\label{Wat}
\Ga\big(J^d_{\M}\,;\sA_{\M}^{<0}\,\big)=0
\quad
\text{
and
}
\quad
\cT_{\M}:
\Ga\big(J^d_{\M}\,;\sA_{\M}\,\big)\slto\cS_{\M,d}.
\end{equation}

Now, we show the following
``relative Watson's lemma'' for $\M$-asymptotics:
\begin{theo}\label{thm:relWat}
Let $\bL$ and $\M$ be weight sequences
admitting nonzero proximate order that
satisfy $0<\o(\bL)\leq \o(\M)<2$ and
$\bL \precsim\M$.
Then, we have
\begin{equation}\label{relWat}
\Ga\big(
J^d_{\M}\,;\sA_{\M}^{<0}/\sA_{\bL}^{<0}\,\big)=0.
\end{equation}
\end{theo}
\begin{proof}
We first note the following sequence is exact:
\begin{equation}\label{es2}
\xymatrix{
0
\ar[r]
&\sA^{<0}_{\bL}
\ar[r]
&\sA^{<0}_{\M}
\ar[r]
&\sA^{<0}_{\M}/
\sA^{<0}_{\bL}
\ar[r]
&0.
}
\end{equation}
Since
$\ds
\Ga\big(J^d_{\M}\,;\sA_{\M}^{<0}\,\big)=0,
$
we obtain from \eqref{es2} the following exact sequence:
\begin{equation*}
\xymatrix{
0
\ar[r]
&\Ga\big(J^d_{\M}\,;\sA_{\M}^{<0}/\sA_{\bL}^{<0}\,\big)
\ar[r]
&H^1\big(J^d_{\M}\,;\sA_{\bL}^{<0}\,\big)
\ar[r]^{j}
&H^1\big(J^d_{\M}\,;\sA_{\M}^{<0}\,\big).
}
\end{equation*}
Therefore, it suffices to show that $j$ is injective.
Using \eqref{es1} and Proposition \ref{prp:CH},
we see that
\begin{equation}\label{es4}
\xymatrix{
\Ga\big(J^d_{\M}\,;\sA_{\bL}\,\big)
\ar[r]
&\Ga\big(J^d_{\M}\,;\fS_{\bL}\,\big)
\ar[r]
&H^1\big(J^d_{\M}\,;\sA_{\bL}^{<0}\,\big)
\ar[r]
&0
}
\end{equation}
is exact.
Since
$\ds
\Ga\big(J^d_{\M}\,;\fS_{\bL}\,\big)\se\fS_{\bL}
$,
and considering the sequence obtained by substituting $\bL$ by $\M$
in \eqref{es4},
we obtain the following commutative diagram:
\begin{equation}\label{diag1}
\xymatrix{
\Ga\big(J^d_{\M}\,;\sA_{\bL}\,\big)
\ar[r]\ar[d]
&\fS_{\bL}
\ar[r]^{\pd_{\bL}\quad\quad}
\ar[d]
&H^1\big(J^d_{\M}\,;\sA_{\bL}^{<0}\,\big)
\ar[r]^{\quad\qquad\iota^1}
\ar[d]^{j}
&0
\\
\Ga\big(J^d_{\M}\,;\sA_{\M}\,\big)
\ar[r]
&\fS_{\M}
\ar[r]^{\pd_{\M}\quad\quad}
&H^1\big(J^d_{\M}\,;\sA_{\M}^{<0}\,\big)
&
}
\end{equation}
Since the first row of \eqref{diag1} is exact,
for any $\fy\in {\rm Ker}j$ we can take $\widehat{f}\in\fS_{\bL}$
 such that $\pd_{\bL}(\widehat{f})=\fy$.
Therefore,
it suffices to show that
\begin{equation}\label{eq1}
\widehat{f}\in{\rm Im}\big[\cT_{\bL}:
\Ga\big(J^d_{\M}\,;\sA_{\bL}\,\big)
\lto\fS_{\bL}
\big].
\end{equation}
Since
$
\fS_{\bL}\lto\fS_{\M}
$
is injective
and the second row of \eqref{diag1} is exact,
we see that
$
\pd_{\M}(\widehat{f})=\iota^1\circ\pd_{\bL}(\widehat{f})=0,
$
and hence,
$\ds
\widehat{f}\in{\rm Im}\big[\cT_{\M}:
\Ga\big(J^d_{\M}\,;\sA_{\M}\,\big)
\lto\fS_{\M}
\big].
$
Therefore,
the proof of
\eqref{eq1} is reduced to \eqref{Taub1} below.
%
\end{proof}
\begin{theo}\label{thm:Taub}
Let $\bL$ and $\M$ be weight sequences
admitting nonzero proximate order that
satisfy $0<\o(\bL)\leq \o(\M)<2$ and
$\bL \precsim\M$.
Then, we have
\begin{align}
%
\Ga\big(J^d_{\M}\,;\sA_{\bL}\,\big)
&\slto
\Ga\big(J^d_{\M}\,;\sA_{\M}\,\big)
\times_{\fS_{\M}}
\fS_{\bL},
\label{Taub1}
\\
\Ga\big(J^d_{\M}\,;\sA_{\bL}\,\big)
&\slto
\Ga\big(J^d_{\M}\,;\sA_{\M}\,\big)
\times_{\fS_{\M}}
\Ga\big(J^d_{\bL}\,;\sA_{\bL}\,\big).
\label{Taub2}
\end{align}
\end{theo}
\begin{proof}
The proofs of
\eqref{Taub1} and \eqref{Taub2}
are the same as the proof of Proposition~\ref{prop.IntersectQuasianalClasses}(i).
\end{proof}
%
%

\subsection{Multisummability via quasi-functions}
\label{subsect.quasifunctions}

Let $\un{\M}=(\M_1,\dots,\M_m)$ be a tuple of
weight sequences admitting nonzero proximate order
that satisfy $\M_m\precsim\dots \precsim\M_1$ (so, $\o(\M_{j+1})\leq\o(\M_j)$ for $j=1,\dots,m-1$) and
$0<\o(\M_m)\leq \o(\M_1)<2$.
We take $\un{d}=(d_1,\dots,d_m)\in\R^m$ so that, with the notation in~\eqref{eq.OptimalInterval}, one has
$J^{d_1}_{\M_1}\supset\dots\supset J^{d_m}_{\M_m}$.
We set
$$
\fA_j:=
\Ga\big(J^{d_j}_{\M_j}\,;\sA_{\M_1}/\sA^{<0}_{\M_{j+1}}\,\big)
\quad\text{for}\quad
j=1,\dots,m,
$$
$$
\fB_j:=\Ga\big(J^{d_{j+1}}_{\M_{j+1}}\,;\sA_{\M_1}/\sA^{<0}_{\M_{j+1}}\,\big)
\quad\text{for}\quad
j=1,\dots,m-1,
$$
with the convention
$\sA^{<0}_{\M_{m+1}}=0$.
Then, we naturally have morphisms
$\fA_j\lto\fB_j$ and $\fA_{j+1}\lto\fB_j$,
and hence,
we have the following diagram:
$$
\xymatrix{
\fA_1
\ar[rd]
&
&
\fA_2
\ar[ld]\ar[rd]
&
&
\fA_3
\ar[ld]\ar[rd]
&
&
&
&
\fA_m
\ar[ld]
\\
&
\fB_1
&
&
\fB_2
&
&
\text{\phantom{$\fB_1$}}
&
\dots
&
\text{\phantom{$\fB_1$}}
&
}
$$
We define $\fA_{\un{\M}}^{\un{d}}$ by the fiber product
(in the category of $\C$-vector spaces)
\begin{equation}\label{fibProd}
\fA_1\times_{\fB_1}\fA_{2}\times_{\fB_{2}}
\dots
\times_{\fB_{m-1}}\fA_m
\end{equation}
of the above diagram:
Since there is a morphism from
$
((\fA_1\times_{\fB_1}\fA_{2})\times_{\fB_{2}}
\dots)
\times_{\fB_{j-1}}\fA_j
$
to $\fB_j$
defined by the composition of
the $j$-th projection to $\fA_j$
and the morphism
$\fA_j\lto\fB_j$,
we can inductively define
$
((\fA_1\times_{\fB_1}\fA_{2})\times_{\fB_{2}}
\dots)
\times_{\fB_{m-1}}\fA_m
$,
and the associativity of the fiber product
enables us to write it in the form \eqref{fibProd}.
Let $p_j:\fA_{\un{\M}}^{\un{d}}\lto\fA_j$ be
the $j$-th projection and let
$\cT_{\un{\M}}:\fA_{\un{\M}}^{\un{d}}\lto\fS_{\M_1}$
be the composition of $p_1$ and $\cT_{\M_1}$
(note that $\sA^{<0}_{\M_{2}}$ is
a subsheaf of $\sA^{<0}_{\M_{1}}={\rm Ker}\cT_{\M_1}$).
We set
$$
\cS_{\un{\M},\un{d}}:=
{\rm Im}\big[
\cT_{\un{\M}}:\fA_{\un{\M}}^{\un{d}}\lto\fS_{\M_1}
\big].
$$
\begin{defi}
Given $\un{\M}$ and $\un{d}$,
the elements $\widehat{f}\in\cS_{\un{\M},\un{d}}$ are called
\emph{$\un{\M}$-summable power series in direction $\un{d}$}.
We say that $\un{f}:=(f_1,\dots,f_m)\in\fA_{\un{\M}}^{\un{d}}$ is the
\emph{$\un{M}$-quasi-sum} of $\widehat{f}$
if $\cT_{\un{\M}}(\un{f})=\widehat{f}$.
\end{defi}
\begin{rema}
When $\un{\M}=(\M_1)$ and $\un{d}=(d_1)$,
$\ds
\fA_{\un{\M}}^{\un{d}}
=\Ga\big(J^{d_1}_{\M_1}\,;\sA_{\M_1}\,\big).
$
Therefore,
$\cS_{\un{\M},\un{d}}=\cS_{\M_1,d_1}$
in this case.
\end{rema}

We obtain from
Theorem \ref{thm:relWat} the following
\begin{theo}\label{thm:inj1}
Given $\un{\M}$ and $\un{d}$,
the morphism $\cT_{\un{\M}}:\fA_{\un{\M}}^{\un{d}}\lto\fS_{\M_1}$
is injective.
\end{theo}
\begin{proof}
Let
$
\un{f}\in{\rm Ker}\cT_{\un{\M}}.
$
Since
$\cT_{\M_1}(f_1)=0$,
we see
$
f_1\in
\Ga\big(J^{d_{1}}_{\M_{1}}\,;\sA_{\M_1}^{<0}/\sA^{<0}_{\M_{2}}\,\big).
$
Therefore, Theorem \ref{thm:relWat} asserts
that $f_1=0$.
Then, since $f_1=f_2$ as an element of $\fB_1$,
we see
$
f_2\in
\Ga\big(J^{d_{2}}_{\M_{2}}\,;\sA_{\M_2}^{<0}/\sA^{<0}_{\M_{3}}\,\big).
$
Continuing the discussion,
we inductively obtain $f_j=0$ $(j=1,\dots,m)$
by the use of \eqref{relWat} and \eqref{Wat},
and hence,
$\cT_{\un{\M}}$ is injective.
\end{proof}
Therefore, we find that
\begin{equation}\label{isom}
\cT_{\un{\M}}:\fA_{\un{\M}}^{\un{d}}
\slto
\cS_{\un{\M},\un{d}}.
\end{equation}
\begin{pro}\label{prp:incl}
Given $\un{\M}$ and $\un{d}$,
we set
$\un{\M}'=(\M_{i_1},\dots,\M_{i_n})$
and
$\un{d}'=(d_{i_1},\dots,d_{i_n})$
with
$
1\leq i_1<\dots<i_n\leq m
$
$
(n\leq m).
$
Then,
there exists an injective morphism
$
\iota:
\fA_{\un{\M}'}^{\un{d}'}
\lto
\fA_{\un{\M}}^{\un{d}}
$
such that
$
\cT_{\un{\M}'}=\cT_{\un{\M}}\circ\iota.
$
\end{pro}
\begin{proof}
We first consider the case when $n=m-1$.
Take $j$ such that
$\{1,\dots,m\}=\{j\}\cup\{i_1,\dots,i_n\}$. We treat two subcases:
\begin{enumerate}[(a)]
\item
\emph{Case $j=1$.}
Since
$\ds
\sA_{\M_2}/\sA^{<0}_{\M_{2}}\slto\fS_{\M_2}
$
is a constant sheaf on $\uC$,
we have
$$\ds
\Ga\big(J^{d_{1}}_{\M_{1}}\,;\sA_{\M_2}/\sA^{<0}_{\M_{2}}\,\big)
\slto
\Ga\big(J^{d_{2}}_{\M_{2}}\,;\sA_{\M_2}/\sA^{<0}_{\M_{2}}\,\big).
$$
Therefore,
we can take a unique element
$\ds
f_1\in\Ga\big(J^{d_{1}}_{\M_{1}}\,;\sA_{\M_2}/\sA^{<0}_{\M_{2}}\,\big)
$
for
$(f_2,\dots,f_m)\in\fA_{\un{\M}'}^{\un{d}'}$
so that $f_1=f_2$ as an element
in
$\ds
\Ga\big(J^{d_{2}}_{\M_{2}}\,;\sA_{\M_2}/\sA^{<0}_{\M_{2}}\,\big).
$
Since we have
$\ds
\sA_{\M_2}/\sA^{<0}_{\M_{j}}
\lto
\sA_{\M_1}/\sA^{<0}_{\M_{j}}
$
$(j\geq2)$,
we see that
$(f_1,f_2,\dots,f_m)$ defines an element
in $\fA_{\un{\M}}^{\un{d}}$,
and hence,
we have a $\C$-linear morphism
$
\iota:
\fA_{\un{\M}'}^{\un{d}'}
\lto
\fA_{\un{\M}}^{\un{d}}.
$

\item
\emph{Case $j\neq 1$.}
Since the following diagram
\begin{equation*}
\xymatrix{
\Ga\big(J^{d_{j-1}}_{\M_{j-1}}\,;\sA_{\M_1}/\sA^{<0}_{\M_{j+1}}\,\big)
\ar[r]\ar[d]
&
\Ga\big(J^{d_{j}}_{\M_{j}}\,;\sA_{\M_1}/\sA^{<0}_{\M_{j+1}}\,\big)
\ar[d]
\\
\Ga\big(J^{d_{j-1}}_{\M_{j-1}}\,;\sA_{\M_1}/\sA^{<0}_{\M_{j}}\,\big)
\ar[r]
&\Ga\big(J^{d_{j}}_{\M_{j}}\,;\sA_{\M_1}/\sA^{<0}_{\M_{j}}\,\big)
}
\end{equation*}
is commutative,
we find that the map
$\ds
g
\mapsto
\big(\,g\ {\rm mod}\ \sA^{<0}_{\M_{j}},\,
g|_{J^{d_{j}}_{\M_{j}}}
\big)
$
defines a $\C$-linear morphism
$
\Ga\big(J^{d_{j-1}}_{\M_{j-1}}\,;\sA_{\M_1}/\sA^{<0}_{\M_{j+1}}\,\big)
\lto
\fA_{j-1}\times_{\fB_{j-1}}\fA_{j}.
$
We set
$$
\widetilde{f}:=
\big(f_1,\dots,
f_{j-1}\ {\rm mod}\ \sA^{<0}_{\M_{j}},\,
f_{j-1}|_{J^{d_{j}}_{\M_{j}}},
f_{j+1},\dots,f_m
\big)
$$
for
$(f_1,\dots,f_{j-1},f_{j+1},\dots,f_m)\in\fA_{\un{\M}'}^{\un{d}'}$.
Then,
we see that $\widetilde{f}$ defines an element
in $\fA_{\un{\M}}^{\un{d}}$
and we obtain a $\C$-linear morphism
$
\iota:
\fA_{\un{\M}'}^{\un{d}'}
\lto
\fA_{\un{\M}}^{\un{d}}.
$
\end{enumerate}
We immediately see that the morphism $\iota$
satisfies
$
\cT_{\un{\M}'}=\cT_{\un{\M}}\circ\iota.
$
The injectivity of $\iota$ follows from
that of $\cT_{\un{\M}'}$ and $\cT_{\un{\M}}$.

Next, we consider the case
when $\un{\M}'$ is given in general.
We take a sequence of tuples of weight sequences
$
\un{\M}_1(=\un{\M}),\un{\M}_2,\dots,
\un{\M}_{m-n}(=\un{\M}')
$
and
directions
$
\un{d}_1(=\un{d}),\un{d}_2,\dots,
\un{d}_{m-n}(=\un{d}')
$
by eliminating a weight sequence and a direction
one by one.
Then, we obtain a morphism
$
\iota:
\fA_{\un{\M}'}^{\un{d}'}
\lto
\fA_{\un{\M}}^{\un{d}}.
$
by successively taking the composition of
the above morphisms.
We see by the definition of the morphisms
that the morphism $\iota$
does not depend
on the choice of such a sequence
and satisfies
$
\cT_{\un{\M}'}=\cT_{\un{\M}}\circ\iota.
$
\end{proof}
\begin{cor}\label{crl:incl}
Given $\un{\M}$ and $\un{d}$,
we have an injective morphism
$
\iota:
\fA_{\M_j,d_j}
\lto
\fA_{\un{\M},\un{d}}
$
for $j=1,\dots,m$.
\end{cor}

\subsection{Multisummability via the decomposition}
\label{subsect.decomposition}

Since we have an isomorphism \eqref{isom},
it follows from Corollary \ref{crl:incl} that
linear combinations of $\M_j$-summable power series
are in $\cS_{\un{\M},\un{d}}$.
In this section,
we show more precisely
that every $\un{\M}$-summable power series
is written in this way.
We set
$$
\fC_j:=\fA_{\M_j,d_j}=
\Ga\big(J^{d_j}_{\M_j}\,;\sA_{\M_{j}}\,\big)
\quad\text{for}\quad
j=1,\dots,m,
$$
$$
\fD_j:=\Ga\big(J^{d_{j}}_{\M_{j}}\,;\sA_{\M_{j+1}}\,\big)
\quad\text{for}\quad
j=1,\dots,m-1.
$$
We have morphisms
$\fD_j\lto\fC_j$ and $\fD_{j}\lto\fC_{j+1}$,
and hence,
we have the following diagram:
$$
\xymatrix{
&
\fD_1
\ar[ld]\ar[rd]
&
&
\fD_2
\ar[ld]\ar[rd]
&
&
\text{\phantom{$\fD_1$}}
\ar[ld]
&
\dots
&
\text{\phantom{$\fD_1$}}\ar[rd]
&
\\
\fC_1
&
&
\fC_2
&
&
\fC_3
&
&
&
&
\fC_m
}
$$
We define $\fC_{\un{\M}}^{\un{d}}$ by the fiber coproduct
(in the category of $\C$-vector spaces)
\begin{equation*}
\fC_1+_{\fD_1}\fC_{2}+_{\fD_{2}}
\dots
+_{\fD_{m-1}}\fC_m
\end{equation*}
of the above diagram (notice that the fiber coproduct is associative).
Since we have morphisms
$\fC_j\lto\fS_{\M_j}\lto\fS_{\M_1}$
$(j=1,\dots,m)$
compatible with the diagram,
we see that they define a morphism
$
\cT'_{\un{\M}}:
\fC_{\un{\M}}^{\un{d}}\lto
\fS_{\M_1}.
$
Then,
we have
\begin{theo}\label{thm:inj2}
Given $\un{\M}$ and $\un{d}$,
the morphism
$
\cT'_{\un{\M}}:
\fC_{\un{\M}}^{\un{d}}\lto
\fS_{\M_1}
$
is injective.
\end{theo}
\begin{proof}
We prove by induction.
When $m=1$,
the injectivity of $\cT'_{\un{\M}}$
follows from \eqref{Wat}.
Next, assume that it holds for $m-1$.
Let $g_j\in\fC_j$
$(j=1,\dots,m)$
and assume that
$g_1+g_2+\dots+g_m\in
{\rm Ker}\big[\cT'_{\un{\M}}:
\fC_{\un{\M}}^{\un{d}}\lto
\fS_{\M_1}\big].
$
Then,
we find
$
\cT'_{\M_1}(g_1)=
\cT'_{\un{\M}}(-g_2-\dots-g_m).
$
Since
$
\cT'_{\un{\M}}(-g_2-\dots-g_m)\in
\fS_{\M_2},
$
it follows from \eqref{Taub1} that
$
g_1\in\fD_1.
$
Therefore,
$(g_1+g_2)+\dots+g_m$ defines an element in
$
{\rm Ker}\big[\cT'_{\un{\M}'}:
\fC_{\un{\M}'}^{\un{d}'}\lto
\fS_{\M_2}\big]
$
with
$\un{\M}'=(\M_{2},\dots,\M_{m})$
and
$\un{d}'=(d_{2},\dots,d_{m})$.
By the induction hypothesis,
we obtain $(g_1+g_2)+\dots+g_m=0$.
This concludes the proof.
\end{proof}
Therefore,
it follows from
\eqref{Wat} and Theorem \ref{thm:inj2} that
\begin{equation*}
\cT'_{\un{\M}}:
\fC_{\un{\M}}^{\un{d}}\slto
\sum_{j=1}^m
\cS_{\M_j,d_j}.
\end{equation*}

Since $\fC_j=\fA_{\M_j}^{d_j}$,
the morphism in Corollary \ref{crl:incl}
defines a morphism
$
\iota:
\fC_{\un{\M}}^{\un{d}}\lto
\fA_{\un{\M}}^{\un{d}},
$
which is compatible with $\cT_{\un{\M}}$ and $\cT'_{\un{\M}}$.
Notice that each morphism
$
\fC_i
\lto
\fA_{\un{\M},\un{d}}
$
is given by the morphisms
\begin{enumerate}[---------]
\item[$(i\leq j)$]
$
\fC_i
\lto
\Ga\big(J^{d_{j}}_{\M_{j}}\,;\sA_{\M_i}\,\big)
\lto
\Ga\big(J^{d_{j}}_{\M_{j}}\,;\sA_{\M_i}/\sA^{<0}_{\M_{j}}\,\big)
\lto
\fA_j,
$

\item[$(i>j)$]
$
\fC_i\lto
\Ga\big(J^{d_{i}}_{\M_{i}}\,;\sA_{\M_i}/\sA^{<0}_{\M_{i}}\,\big)
\slto
\Ga\big(J^{d_{j}}_{\M_{j}}\,;\sA_{\M_i}/\sA^{<0}_{\M_{i}}\,\big)
\lto
\fA_j.
$

\end{enumerate}
Then, we have the following
\begin{theo}\label{thm:isom1}
Given $\un{\M}$ and $\un{d}$,
the morphism $\iota$
gives an isomorphism
\begin{equation*}
\iota:
\fC_{\un{\M}}^{\un{d}}\slto
\fA_{\un{\M}}^{\un{d}}.
\end{equation*}
\end{theo}
\begin{proof}
Since $\cT'_{\un{\M}}=\cT_{\un{\M}}\circ\iota$,
injectivity of $\iota$ follows from
the injectivity of $\cT_{\un{\M}}$ and $\cT'_{\un{\M}}$.
We prove surjectivity of $\iota$.
Let us consider the following exact sequence:
\begin{equation}\label{es5}
\xymatrix{
0
\ar[r]
&\sA^{<0}_{\M_{m}}
\ar[r]
&\sA_{\M_{1}}
\ar[r]
&\sA_{\M_{1}}/
\sA^{<0}_{\M_{m}}
\ar[r]
&0.
}
\end{equation}
Since the morphism
$
\sA^{<0}_{\M_{m}}
\lto\sA_{\M_{1}}
$
factorizes through $\sA_{\M_{m}}$,
applying the functors $\Ga(J^{d_{m-1}}_{\M_{m-1}}\,;\,\cdot\,)$
and $\Ga(J^{d_m}_{\M_m}\,;\,\cdot\,)$
to \eqref{es5},
we find by Proposition \ref{prp:CH}
the following commutative diagram:
\begin{equation}\label{diag2}
\xymatrix{
0
\ar[r]
&
\Ga\big(J^{d_{m-1}}_{\M_{m-1}}\,; \sA_{\M_{1}}\big)
\ar[r]\ar[d]
&
\fA_{m-1}
\ar[r]^{\pd_1\quad\quad}
\ar[d]^{j_1}
&H^1\big(J^{d_{m-1}}_{\M_{m-1}}\,;\sA_{\M_m}^{<0}\,\big)
\ar[r]
\ar[d]^{j_2}
&0
\\
0
\ar[r]
&
\fA_m
\ar[r]
&
\fB_{m-1}
\ar[r]^{\pd_2\quad\quad}
&H^1\big(J^{d_m}_{\M_m}\,;\sA_{\M_m}^{<0}\,\big)
\ar[r]
&0.
}
\end{equation}
By the same reasoning,
applying the functor
$\Ga(J\,;\,\cdot\,)$
to \eqref{es1} with
$
\fS_{\M}\se
\sA_{\M}/\sA^{<0}_{\M}
$
and
$\M=\M_m$,
we have the following exact sequence
when an arc $J$ of $\uC$ contains
$J^{d_m}_{\M_m}$:
\begin{equation}\label{es6}
\xymatrix{
0
\ar[r]
&
\Ga\big(J\,;\sA_{\M_m}\,\big)
\ar[r]
&\Ga\big(J\,;\sA_{\M_m}/\sA^{<0}_{\M_m}\,\big)
\ar[r]
&H^1\big(J\,;\sA_{\M_m}^{<0}\,\big)
\ar[r]
&0.
}
\end{equation}
Since \eqref{es6} splits and
$
\sA_{\M}/\sA^{<0}_{\M}\se
\fS_{\M}
$
is a constant sheaf,
we have the following isomorphism:
\begin{equation}\label{isom3}
\Ga\big(J^{d_{m-1}}_{\M_{m-1}}\,; \sA_{\M_{m}}\big)
\oplus
H^1\big(J^{d_{m-1}}_{\M_{m-1}}\,;\sA_{\M_m}^{<0}\,\big)
\slto
\fC_m
\oplus
H^1\big(J^{d_{m}}_{\M_{m}}\,;\sA_{\M_m}^{<0}\,\big).
\end{equation}
Notice that the part
$\ds
H^1\big(J^{d_{m-1}}_{\M_{m-1}}\,;\sA_{\M_m}^{<0}\,\big)
\lto
H^1\big(J^{d_{m}}_{\M_{m}}\,;\sA_{\M_m}^{<0}\,\big)
$
in the morphism \eqref{isom3}
is given by $j_2$.

Now, let
$
f=
(f_1,\dots,f_m)\in
\fA_{\un{\M}}^{\un{d}}.
$
Then,
since $j_1(f_{m-1})\in{\rm Ker}\,\pd_2$
and the diagram \eqref{diag2} is commutative,
we have $\pd_1(f_{m-1})\in{\rm Ker}\,j_2$.
Therefore,
using the isomorphism \eqref{isom3},
we can take $g_m\in\fC_m$
such that
$
\iota(g_m)
=(g_{m,1},\dots,g_{m,m})\in
\fA_{\un{\M}}^{\un{d}}
$
satisfies
$\pd_1(g_{m,m-1})=\pd_1(f_{m-1})$.
(Recall the definition of $\iota$.
$g_{m,m-1}$ is defined by using
the isomorphism
$
\Ga\big(J^{d_{m-1}}_{\M_{m-1}}\,;
\sA_{\M_m}/\sA^{<0}_{\M_m}\,\big)
$
$
\slto
\Ga\big(J^{d_{m}}_{\M_{m}}\,;
\sA_{\M_m}/\sA^{<0}_{\M_m}\,\big)
$
equivalent to \eqref{isom3}.)
Since the first row of \eqref{diag2} is exact,
we can take
$
\widetilde{f}_{m-1}\in
\Ga\big(J^{d_{m-1}}_{\M_{m-1}}\,; \sA_{\M_{1}}\big)
$
such that
$
f_{m-1}-g_{m,m-1}=\widetilde{f}_{m-1}
\,{\rm mod}\,\sA^{<0}_{\M_m}.
$
Then,
$
(f_1-g_{m,1},\dots,f_{m-2}-g_{m,m-2},\widetilde{f}_{m-1})
$
defines an element in
$
\fA_{\un{\M}'}^{\un{d}'}
$
with
$
\un{\M}'=(\M_1,\dots,\M_{m-1})
$
and
$
\un{d}'=(d_1,\dots,d_{m-1}).
$

Now,
repeating the discussion,
we can choose
$g_j\in\fC_j$
$(j=1,\dots,m)$
such that
$
f=\iota(g_1)+\dots+\iota(g_m).
$
This shows the surjectivity of $\iota$.
\end{proof}
\begin{cor}\label{crl:isom1}
Given $\un{\M}$ and $\un{d}$,
we have
\begin{equation*}
\cS_{\un{\M},\un{d}}
=
\sum_{j=1}^m
\cS_{\M_j,d_j}.
\end{equation*}
\end{cor}

\subsection{Multisummability via iterated Laplace transforms}
\label{subsect.iteratedLaplace}

In this section,
we construct the $\un{\M}$-quasi-sum of every
$\un{\M}$-summable power series
using iterated Laplace transforms, so presenting a variation of the treatment of Theorem~\ref{th.multisum.construct}.

Since $J^{d_j}_{\M_j}=J^{d_{j+1}}_{\M_{j+1}}$
when
$\o(\M_j)=\o(\M_{j+1})$,
we see that
$
\fA_j\times_{\fB_j}\fA_{j+1}\se
\fA_{j+1}
$
and
$
\fC_j+_{\fD_j}\fC_{j+1}\se
\fC_{j},
$
and hence,
$
\fA_{\un{\M}}^{\un{d}}\se
\fA_{\un{\M}'}^{\un{d}'}
$
and
$
\fC_{\un{\M}}^{\un{d}}\se
\fC_{\un{\M}'}^{\un{d}'},
$
where
$\un{\M}'$ (resp. $\un{d}'$) is a tuple obtained by
eliminating $\M_{j+1}$ (resp. $d_{j+1}$)
from $\un{\M}$ (resp. $\un{d}$).
Therefore,
we may assume
that the following condition is satisfied:
\begin{equation*}
\o(\M_j)>\o(\M_{j+1})
\quad
(j=1,\dots,m-1).
\end{equation*}

\begin{defi}\label{dfn:e-multi}
Given $\un{\M}$ and $\un{d}$,
we take a strong kernel $e_j$ of $\M_j$-summability
and set
$\un{e}=(e_1,\dots,e_m)$.
Let $\widehat{f}\in\fS$
and we inductively define
$g_1,\dots,g_m$ by
$$
g_1:=\widehat{T}^-_{e_1}(\widehat{f}),
\quad
g_{j+1}:=
\int_0^{\infty(d_j)}
e_{j+1}\triangleleft e_{j}(u/z)g_j(u)\frac{du}{u}.
$$
We say that $\widehat{f}\in\fS$ is
\emph{$T_{\un{e}}$-summable
in direction $\un{d}$}
if $g_1\in\cS$,
$g_j\in\cO^{\M_{j}/\M_{j+1}}(S(d_j,\ep))$
$(j=1,\dots,m-1)$
and $g_m\in\cO^{\M_{m}}(S(d_m,\ep))$
for some $\ep>0$.
We call $T_{e_m}(g_m)$
\emph{the sum of $\widehat{f}$}
and denote it by
$
\calS_{\un{e},\un{d}}\widehat{f}.
$
\end{defi}
\begin{rema}
By the definition of $g_{j}$ $(j\geq2)$ and
Theorem~\ref{teorrelacdesartransfBL},
we see that
$g_{j}\sim_{\M_{1}/\M_{j}}\widehat{T}^-_{e_{j}}(\widehat{f})$
in a sectorial region
$G(d_{j-1},\o(\M_{j-1})-\o(\M_{j})+\ep)$.
\end{rema}

Now, we give a different proof of the important characterization
of $\un{\M}$-summable series, obtained in Theorem~\ref{th.multisum.construct} for the case of two sequences.
\begin{theo}
Given $\un{\M}$ and $\un{d}$,
the following statements are equivalent:
\begin{enumerate}[(i)]
\item
$\widehat{f}\in\cS_{\un{\M},\un{d}}$.

\item There exist $\widehat{f}_j\in\cS_{\M_j,d_j}$
$(j=1,\dots,m)$ such that $\widehat{f}$ is written as
$$
\widehat{f}=\widehat{f}_1+\dots+\widehat{f}_m.
$$

\item
For every tuple $\un{e}$ of
strong kernels $e_j$ of $\M_j$-summability,
$\widehat{f}$ is
$T_{\un{e}}$-summable
in direction $\un{d}$.

\item
For some tuple $\un{e}$ of
strong kernels $e_j$ of $\M_j$-summability,
$\widehat{f}$ is
$T_{\un{e}}$-summable
in direction $\un{d}$.

\end{enumerate}
\end{theo}
\begin{proof}
The equivalence of (i) and (ii)
follows from Corollary \ref{crl:isom1}
and (iii) $\Longrightarrow$ (iv) is trivial.
Therefore, it suffices to show
(ii) $\Longrightarrow$ (iii) and
(iv) $\Longrightarrow$ (i).

We first prove
(ii) $\Longrightarrow$ (iii).
Let $\un{e}$ be an arbitrary tuple of strong kernels.
Since $\un{e}$-summability is preserved under the summation,
we may assume that $\widehat{f}\in\cS_{\M_k,d_k}$
for some $k$.
Let $g_1,\dots,g_m$ be the functions defined in
Definition \ref{dfn:e-multi}.
Since $\widehat{f}\in\fS_{\M_k}$,
there exist constants $C,A>0$ such that
\begin{equation}\label{est:1}
|g_1(z)|
\leq
C\sum_{p=0}^{\infty}
\frac{M_{k,p}}{M_{1,p}}A^p|z|^p,
\end{equation}
and hence,
we have
$g_1(z)\in\cS$.
When $k=1$,
we find by the assumption and Theorem~\ref{th.Msummable.equiv.esummable} that
$g_1\in\cO^{\M_{1}}(S(d_1,\ep))$ for some $\ep>0$.
Then, we see by
Proposition~\ref{prop.acceleration.kernel}.3
that $g_2=T^{-}_{e_2}\circ T_{e_1}(g_1)$
and
we obtain from Theorem~\ref{prop.hol.Bor}
that
$g_2\in\cO^{\M_{2}}(S(d_1,\o(\M_1)-\o(\M_2)+\ep))$.
Since $J^{d_2}_{\M_2}\subset J^{d_{1}}_{\M_{1}}$,
we see that
$
S(d_2,\ep)
\subset
S(d_1,\o(\M_1)-\o(\M_2)+\ep),
$
and hence,
$g_2\in\cO^{\M_{2}}(S(d_2,\ep))$.
Repeating the same discussion,
we can inductively confirm that
$g_j\in\cO^{\M_{j}}(S(d_j,\ep))$
for $j=1,\dots,m$.
Since
$
\M_{j}\precsim
\M_{j}/\M_{j+1}
$
indicates
$\cO^{\M_{j}}(S(d_j,\ep))
\subset
\cO^{\M_{j}/\M_{j+1}}(S(d_j,\ep))$,
$\widehat{f}$ is
$\un{e}$-summable.
When $k\neq1$,
we can deduce from the estimate \eqref{est:1}
and Proposition~\ref{propKomatsu}
that $g_1\in\cO^{\M_{1}/\M_{k}}(\C)$.
Then,
$
g_2\sim_{\M_1/\M_2}
\widehat{T}^{-}_{e_2}(\widehat{f})\in\cS
$
in a sectorial region
$G(d_1,\o(\M_1)-\o(\M_2)+\ep)$
for $\ep>0$,
and hence,
$
g_2
$
is analytic at the origin
by Theorem~\ref{TheoWatsonlemma}.
Repeating the discussion,
we can inductively confirm that
$
g_j=
\widehat{T}^{-}_{e_j}(\widehat{f})
\in\cO^{\M_{j}/\M_{k}}(\C)
$
for $j=1,\dots,k-1$
and $g_k\in\cS$.
Then,
by the assumption $\widehat{f}\in\cS_{\M_k,d_k}$,
we have $g_k\in\cO^{\M_{k}}(S(d_k,\ep))$,
and hence,
applying the same discussion as in case $k=1$,
we see
$g_j\in\cO^{\M_{j}}(S(d_j,\ep))$
for $j=k,\dots,m$.
Therefore,
we find that
$\widehat{f}$ is
$T_{\un{e}}$-summable.

Next, we show
(iv) $\Longrightarrow$ (i).
Assume that $\widehat{f}$ is
$T_{\un{e}}$-summable.
We first note that
it follows from Theorem~\ref{teorrelacdesartransfBL} and Theorem~\ref{th.asymp.acceleration}.1
that
$
\calS_{\un{e},\un{d}}\widehat{f}
\sim_{\M_1}
\widehat{f}
$
in a sectorial region
$G(d_m,\o(\M_m)+\ep)$
for $\ep>0$,
and hence,
$$
f_m:=
\calS_{\un{e},\un{d}}\widehat{f}
\in\fA_m.
$$
Next, we show that
$
f_m
\ {\rm mod}\ \sA^{<0}_{\M_{m}}
$
extends to $\fA_{m-1}$.
Since $g_1(z)$ is holomorphic
in a sectorial region
$G(d_{m-1},\o(\M_{m-1})-\o(\M_m)+\ep)$
for $\ep>0$,
we can define a family of functions
$$
T^{\ze}_{e_{m}}g_m
:=\int_0^{\ze}
e_{m}(u/z)g_m(u)\frac{du}{u}
$$
for
$
\ze\in
G(d_{m-1},\o(\M_{m-1})-\o(\M_m)+\ep).
$
We note that
$
T^{\ze}_{e_{m}}g_m\sim_{\M_1}\widehat{f}
$
in a sectorial region
$G(\arg \ze,\o(\M_m))$
(see the proof of \cite[Th. 6.1]{SanzFlat}
for the details).
Further,
we consider
$$
T^{\ga}_{e_{m}}g_m:=
\int_{\ga}
e_{m}(u/z)g_m(u)\frac{du}{u}
$$
for a closed $C^1$-path $\ga$ in
$
G(d_{m-1},\o(\M_{m-1})-\o(\M_m)+\ep).
$
Since
$T^{\ga}_{e_{m}}g_m
\sim_{\M_m}0$
in a sectorial region
$G(d,\o(\M_m)-\de)$
when
$\ga$
is a path in
$S(d,\de)$
for $d\in\R$ and $\de>0$,
we find that
$
f_{m-1}:=
T^{\ze}_{e_{m}}g_m
\ {\rm mod}\ \sA^{<0}_{\M_{m}}
$
extends to an element in $\fA_{m-1}$
by moving the end point $\ze$ of
the integration in $T^{\ze}_{e_{m}}g_m$.
Then, we find
$
f_{m-1}-f_m\ {\rm mod}\ \sA^{<0}_{\M_{m}}
\sim_{\M_m}0
$
in a sectorial region
$G(d_m,\o(\M_m)+\ep)$
for $\ep>0$,
and hence,
$
f_{m-1}-f_m\ {\rm mod}\ \sA^{<0}_{\M_{m}}
\in
\Ga\big(J^{d_{m}}_{\M_{m}}\,;\sA^{<0}_{\M_1}/\sA^{<0}_{\M_{m}}\,\big).
$
Therefore,
we see by
Theorem \ref{thm:relWat}
that
$
f_{m-1}=f_m\ {\rm mod}\ \sA^{<0}_{\M_{m}}
$
in
$\fB_{m-1}$.

Now,
we set
$$
f_{j}:=
T^{\ze}_{e_{j+1}}g_{j+1}
\ {\rm mod}\ \sA^{<0}_{\M_{j+1}}
\ \text{for}\
j=1,2,\dots,m-1.
$$
Then,
by the same discussion as in case $j=m-1$,
we find that
$f_j\in\fA_{j}$ and
$
f_{j}=f_{j+1}\ {\rm mod}\ \sA^{<0}_{\M_{j+1}}
$
in
$\fB_{j}$
for $j=1,2,\dots,m-1.$
Therefore,
$\un{f}=(f_1,\dots,f_m)\in\fA_{\un{\M}}^{\un{d}}$
and $\cT_{\un{\M}}(\un{f})=\widehat{f}$,
what concludes the proof.
\end{proof}

\vskip.4cm
\noindent\textbf{Affiliation}:\\
J.~Jim\'{e}nez-Garrido, J.~Sanz:\\
Departamento de \'Algebra, An\'alisis Matem\'atico, Geometr{\'\i}a y Topolog{\'\i}a\\
Universidad de Va\-lla\-do\-lid\\
Facultad de Ciencias, Paseo de Bel\'en 7, 47011 Valladolid, Spain.\\
Instituto de Investigaci\'on en Matem\'aticas IMUVA\\
E-mails: jjjimenez@am.uva.es (J.~Jim\'{e}nez-Garrido), jsanzg@am.uva.es (J. Sanz).
\\
\vskip.1cm
\noindent S. Kamimoto:\\
Graduate School of Sciences\\
Hiroshima University\\
1-3-1 Kagamiyama, Higashi-Hiroshima, Hiroshima 739-8526, Japan\\
E-mail: kamimoto@hiroshima-u.ac.jp
\\
\vskip.1cm
\noindent A. Lastra:\\
Departamento de F\'isica y Matem\'aticas\\
Universidad de Alcal\'a\\
E--28871. Alcal\'a de Henares, Madrid, Spain.\\
E-mail: alberto.lastra@uah.es

\end{document}